\documentclass[11pt]{amsart}
\usepackage{comment}
\usepackage{amssymb,amsmath,amsthm,mathrsfs,amstext,amssymb}
\usepackage{dsfont}
\usepackage{tabularx}
\usepackage{array}
\usepackage[square,sort,comma,numbers]{natbib}
\usepackage{url} 
\usepackage{hyperref}
\usepackage{accents}
\usepackage[shortlabels]{enumitem}
\usepackage{tikz-cd}

\newtheorem{theorem}{Theorem}
\numberwithin{theorem}{section}
\newtheorem*{theorem*}{Theorem}
\newtheorem{maintheorem}[theorem]{Main Theorem}
\newtheorem*{maintheorem*}{Main Theorem}
\newtheorem{lemma}[theorem]{Lemma}
\newtheorem*{lemma*}{Lemma}
\newtheorem{fact}[theorem]{Fact}
\newtheorem*{fact*}{Fact}
\newtheorem{corollary}[theorem]{Corollary}
\newtheorem*{corollary*}{Corollary}
\newtheorem{proposition}[theorem]{Proposition}
\newtheorem*{proposition*}{Proposition}

\theoremstyle{definition}

\newtheorem{definition}[theorem]{Definition}
\newtheorem*{definition*}{Definition}
\newtheorem{remark}[theorem]{Remark}
\newtheorem*{remark*}{Remark}
\newtheorem{question}[theorem]{Question}
\newtheorem*{question*}{Question}

\newcommand{\A}{{\mathcal{A}}}
\newcommand{\B}{{\mathcal{B}}}
\newcommand{\C}{{\mathcal{C}}}

\newcommand{\T}{{\mathcal{T}}}

\newcommand{\CC}{{\mathbb{C}}}

\newcommand{\HH}{{\mathbb{H}}}

\newcommand{\PP}{{\mathbb{P}}}
\newcommand{\QQ}{{\mathbb{Q}}}

\newcommand{\TT}{{\mathbb{T}}}

\newcommand{\ZZ}{{\mathbb{Z}}}

\renewcommand{\a}{{\mathfrak{a}}}

\renewcommand{\c}{{\mathfrak{c}}}
\renewcommand{\d}{{\mathfrak{d}}}

\renewcommand{\i}{{\mathfrak{i}}}

\renewcommand{\t}{{\mathfrak{t}}}

\newcommand{\Chi}{\mathrm{X}}

\DeclareMathOperator{\dom}{dom}

\DeclareMathOperator{\restr}{\upharpoonright}
\DeclareMathOperator{\concat}{{^\smallfrown}}

\newcommand{\simpleset}[1]{{\{{#1}\}}}
\newcommand{\simpleseq}[1]{{\langle{#1}\rangle}}
\newcommand{\set}[2]{{\{ {#1} \mid {#2} \}}}
\newcommand{\seq}[2]{{\langle {#1} \mid {#2} \rangle}}

\DeclareMathOperator{\completesubposet}{{\mathbin{\leqslant \!\! \circ}}}
\DeclareMathOperator{\extends}{{\mathbin{\leq}}}

\DeclareMathOperator{\compat}{{\mathbin{||}}}

\DeclareMathOperator{\forces}{{ \, \Vdash \, }}

\DeclareMathOperator{\name}{name}

\DeclareMathOperator{\cof}{cof}

\DeclareMathOperator{\spec}{spec}
\DeclareMathOperator{\aE}{\a_{\text{\normalfont e}}}
\DeclareMathOperator{\aG}{\a_{\text{\normalfont g}}}
\DeclareMathOperator{\aT}{\a_{\text{\normalfont T}}}

\DeclareMathOperator{\cantorspace}{{^\omega 2}}

\DeclareMathOperator{\fincantorspace}{{^{<\omega} 2}}

\DeclareMathOperator{\supp}{supp}
\DeclareMathOperator{\hsupp}{hsupp}

\DeclareMathOperator{\tree}{tree}

\DeclareMathOperator{\subsequence}{\trianglelefteq}

\DeclareMathOperator{\id}{id}

\DeclareMathOperator{\actson}{{\curvearrowright}}

\DeclareMathOperator{\Aut}{Aut}

\title{Realizing arbitrarily large spectra of $\aT$}

\author{V. Fischer}
\address{Institute of Mathematics, University of Vienna, Kolingasse 14-16, 1090 Vienna, Austria}
\email{vera.fischer@univie.ac.at}

\author{L. Schembecker}
\address{Institute of Mathematics, University of Vienna, Kolingasse 14-16, 1090 Vienna, Austria}
\email{lukas.schembecker@univie.ac.at}

\begin{document}
	\maketitle
	
	\begin{abstract}
		We improve the state-of-the-art proof techniques for realizing various spectra of $\aT$ in order to realize arbitrarily large spectra.
		Thus, we make significant progress in addressing a question posed by Brian in his work~\cite{Brian_2021}.
		As a by-product, we obtain many complete subforcings and an algebraic analysis of the automorphisms of the forcing which adds a witness for the spectrum of $\aT$ of desired size.
	\end{abstract}
	
	\section{Introduction} \label{SEC_Introduction}
	
	Fundamentally, combinatorial set theory studies the possible sizes and relations between special subsets of reals.
	Usually, these special subsets are defined by some combinatorial property, e.g.\ mad families, independent families or partitions of Baire space into compact sets.
	Classically, the corresponding cardinal characteristics, i.e.\ the minimal sizes of such special subsets, and their relations are of main interest.
	However, a more recent approach is the study of their corresponding spectra, i.e.\ of all possible sizes of such special subsets at the same time.
	For some fixed type of combinatorial family of reals its spectrum can be studied from the following two angles.
	
	On one hand, one may consider which properties of the spectrum are provable in {\sf ZFC}.
	On the other hand, given a set of cardinals $\Theta$ with some additional assumptions one may construct forcing extensions in which $\Theta$ is precisely realized as the spectrum.
	Thus, the ultimate goal is to reduce the additional assumptions on $\Theta$ until they agree with the provable properties of the spectrum in {\sf ZFC}, so that we obtain a complete classification of the possible spectra of some type of combinatorial family of reals.
	
	Usually, the spectrum of some type of family may be rather arbitrary, so that there are not many provable properties in {\sf ZFC}.
	However, recent progress suggests that the following properties are shared between different spectra.
	First, usually by some straightforward combinatorial argument the continuum $\c$ is in the spectrum (a notable exception is the tower number $\t$).
	By König's Theorem we obtain the following necessary restriction on $\Theta$:
	\begin{enumerate}[(I)]
		\item $\max(\Theta)$ exists and has uncountable cofinality.
	\end{enumerate}
	Secondly, there seems to be the following additional restriction on $\Theta$:
	\begin{enumerate}[(I)]
		\setcounter{enumi}{1}
		\item $\Theta$ is closed under singular limits.
	\end{enumerate}
	For example, in \cite{Hechler_1972} Hechler proved that $\spec(\a)$ is closed under singular limits.
	Similarly, recently Brian proved in ~\cite{Brian_2021} that also $\spec(\aT)$ (cf.\ Definition~\ref{DEF_aT}) is closed under singular limits.
	However, for most other types of families it is still not known if this restriction is necessary, i.e.:
	\begin{question*}
		Are $\spec(\i)$, $\spec(\aE)$ and $\spec(\aG)$ closed under singular limits?
	\end{question*}
	Finally, specifically for the spectrum of $\aT$ Brian recently provided another necessary assumption given by {\sf ZFC}.
	\begin{theorem*}[Brian, 2022, \cite{Brian_2022}]
		Assume $0^\dagger$ does not exist.
		If $\theta$ has countable cofinality and we have $\theta \in \spec(\aT)$, then also $\theta^+ \in \spec(\aT)$.
	\end{theorem*}
	
	In particular, a model in which $\theta \in \spec(\aT)$ and $\theta^+ \notin \spec(\aT)$ implies that $0^\dagger$ exists, so that there exists an inner model with a measurable cardinal.
	Hence, such a model cannot be constructed relative to {\sf ZFC}.
	Note that this result is in stark contrast to the situation for the spectrum of $\a$.
	In this, case Shelah and Spinas proved in \cite{ShelahSpinas_2015} that consistently $\aleph_{\omega} \in \spec(\a)$, but $\aleph_{\omega + 1} \notin \spec(\a)$.
	Hence, despite their similarities there are distinct discrepancies between the spectra of different types of families.
	
	On the other hand, the realization of various spectra with the means of forcing was first studied for almost disjoint families.
	For mad families Hechler proved that any set of uncountable cardinals may be contained in the spectrum.
	
	\begin{theorem*}[Hechler, 1972, \cite{Hechler_1972}]
		Let $\Theta$ be any set of uncountable cardinals.
		Then, there is a c.c.c.\ forcing extension in which $\Theta \subseteq \spec(\a)$ holds.
	\end{theorem*}
	
	In order to exclude values from the spectrum and precisely realize $\Theta$ as some spectrum, one usually employs an isomorphism-of-names argument.
	For example, Blass proved that under the following additional assumptions on the set $\Theta$, in Hechler's model the set $\Theta$ is already precisely realized as the spectrum of mad families:
	
	\begin{theorem*}[Blass, 1993, \cite{Blass_1993}]
		Assume $\sf{GCH}$ and let $\Theta$ be a set of uncountable cardinals such that
		\begin{enumerate}[\normalfont (I)]
			\item $\max(\Theta)$ exists and has uncountable cofinality,
			\item $\Theta$ is closed under singular limits,
			\item $\aleph_1 \in \Theta$,
			\item If $\theta \in \Theta$ with $\cof(\theta) = \omega$, then $\theta^+ \in \Theta$.
		\end{enumerate}
		Then, there is a c.c.c.\ forcing extension in which $\spec(\a) = \Theta$ holds.
	\end{theorem*}
	
	Employing a more sophisticated isomorphism-of-names argument, Shelah and Spinas later improved this result by weakening assumption (III) and removing assumption (IV):
	
	\begin{theorem*}[Shelah, Spinas, 2015, \cite{ShelahSpinas_2015}]
		Assume $\sf{GCH}$ and let $\Theta$ be a set of uncountable cardinals such that
		\begin{enumerate}[\normalfont (I)]
			\item $\max(\Theta)$ exists and has uncountable cofinality,
			\item $\Theta$ is closed under singular limits,
			\item $\min(\Theta)$ is regular.
		\end{enumerate}
		Then, there is a c.c.c.\ forcing extension in which $\spec(\a) = \Theta$ holds.
	\end{theorem*}
	
	By the previous discussion (I) and (II) are necessary assumptions.
	However, $\a$ may be singular, so that (III) is definitely not necessary.
	In fact, Brendle proved that $\a$ may be any uncountable singular cardinal, even of countable cofinality \cite{Brendle_2003}.
	Thus, an answer to the following question would yield a complete classification of all the possible spectra of $\a$:
	
	\begin{question*}
		Can assumption (III) be removed from the previous theorem?
	\end{question*}
	
	Similar progress has been made for independent families by Fischer and Shelah \cite{FischerShelah_2022} and partitions of Baire space into compact sets by Brian \cite{Brian_2021}, which is the main focus of this paper:
	
	\begin{theorem*}[Brian, 2021, \cite{Brian_2021}]\label{THM_AT_Brian}
		Assume $\sf{GCH}$ and let $\Theta$ be a set of uncountable cardinals such that
		\begin{enumerate}[\normalfont (I)]
			\item $\max(\Theta)$ exists and has uncountable cofinality,
			\item $\Theta$ is closed under singular limits,
			\item If $\theta \in \Theta$ with $\cof(\theta) = \omega$, then $\theta^+ \in \Theta$,
			\item $\min(\Theta)$ is regular,
			\item $\left|\Theta\right| < \min(\Theta)$.
		\end{enumerate}
		Then, there is a c.c.c.\ forcing extension in which $\spec(\aT) = \Theta$ holds.
	\end{theorem*}
	
	Again, by the previous discussion (I), (II) and (III) are necessary assumptions.
	Assumption (IV) is not necessary as $\aT$ may be any singular cardinal of uncountable cofinality.
	However, unlike $\a$ it is still open if $\aT$ may have countable cofinality.
	Assumption (V) is also not necessary as we may force any set of uncountable cardinals to be contained in $\spec(\aT)$ similar to Hechler's theorem for $\spec(\a)$.
	In other words, assumption (V) implies that once the minimum of $\Theta$ has been fixed, only a bounded set of cardinals may be realized with the methods employed by Brian.
	Thus, in~\cite{Brian_2021} he asked if it is possible to remove assumption (V).
	Inspired by the methods of Shelah and Spinas for $\spec(\a)$ and towards obtaining a complete classification of the possible spectra of $\aT$, we prove the following Main Theorem~\ref{THM_AT_Spectrum} and give a partial answer to Brian's question:
	
	\begin{maintheorem*}
		Assume $\sf{GCH}$ and let $\Theta$ be a set of uncountable cardinals such that
		\begin{enumerate}[\normalfont (I)]
			\item $\max(\Theta)$ exists and has uncountable cofinality,
			\item $\Theta$ is closed under singular limits,
			\item If $\theta \in \Theta$ with $\cof(\theta) = \omega$, then $\theta^+ \in \Theta$,
			\item $\aleph_1 \in \Theta$.
		\end{enumerate}
		Then, there is a c.c.c.\ forcing extension in which $\spec(\aT) = \Theta$ holds.
	\end{maintheorem*}
	
	Thus, we are indeed able to realize arbitrarily large spectra, however our current proof methods require us to strengthen assumption (IV).
	In Section~\ref{SEC_Sketch} we outline the proof of Main Theorem~\ref{THM_AT_Spectrum} and discuss how to possibly avoid the strengthening of (IV) in order to obtain a full answer to Brian's question.
	Nevertheless, the following summarizes how the proof of Main Theorem~\ref{THM_AT_Spectrum} extends the current proof methods and techniques for realizing various spectra:
	
	Generally, the forcing used to obtain our result is very similar to the forcing used in Brian's result above, but with a distinct modification in order to allow a more sophisticated isomorphism-of-names argument.
	Inspired by Shelah's and Spinas' result for $\spec(\a)$ the main feature of our argument is the restriction to isomorphic complete subforcings of the entire forcing.
	In contrast, Brian's argument only uses automorphism of the entire forcing, which leads to his restriction (V).
	
	The main difficulty of our proof is showing that we indeed have many complete subforcings (see Theorem~\ref{THM_Complete_Subforcings}).
	In the situation for $\spec(\a)$ there is a Suslin-c.c.c.\ product-like forcing, which adds a maximal almost disjoint family of desired size.
	Thus, the existence of complete subforcings is easy to prove in that case.
	In contrast, for $\spec(\aT)$ there is no known such Suslin-c.c.c.\ product-like forcing and instead we have to use an iteration of c.c.c.\ forcings in order to obtain a witness for $\aT$ of desired size.
	To establish the existence of complete subforcings, we introduce the following novel advancements.
	
	First, compared to Brian's forcing in \cite{Brian_2021}, our forcing (see Definition~\ref{DEF_Forcing}) has a distinct modification, which allows for more automorphisms.
	In Section~\ref{SEC_Extend_Automorphisms} we provide a very algebraic framework of these automorphisms.
	Nevertheless, we strive for a self-contained presentation.
	Secondly, since we do not have a Suslin forcing, we cannot simply use the standard notion of a canonical projection of a nice name of a real (cf.\ \cite{FischerTornquist_2015}).
	Instead, in Definition~\ref{DEF_Nice_Name_WRT_Trees} we introduce the technical notion of a nice name for a finite set of reals with respect to a sequence of names for trees.
	The canonical projection of this technical nice name then has the desired properties in our proof to obtain complete subforcings.
	
	Finally, in Section~\ref{SEC_Proof} we prove the isomorphism-of-names argument needed for our Main Theorem~\ref{THM_AT_Spectrum}. 
	However, again the situation is more complicated than for the spectrum of $\a$ by Shelah and Spinas, because we are working with an iteration.
	To this end, in Section~\ref{SEC_Extending_Isomorphisms} we provide a very algebraic/categorical framework for the isomorphisms between the many complete subforcings just discussed.
	Lastly, we use these isomorphisms to show that the corresponding isomorphism-of-names argument can be carried out for the iteration.
	Thus, the main insight is that this more sophisticated isomorphism-of-names argument can not only be applied in a product-like context as for $\spec(\a)$, but also in a more intricate iteration-like context as for $\spec(\aT)$.
	
	\section{Preliminaries} \label{SEC_Preliminaries}
	
	In this section we introduce the cardinal characteristic $\aT$ and its associated spectrum $\spec(\aT)$.
	We will also define a c.c.c.\ forcing which forces the existence of witnesses in $\spec(\aT)$ of various sizes.
	In Definition~\ref{DEF_Forcing} we define a slightly tweaked version of this forcing in order to realize arbitrarily large spectra of $\aT$ in our Main Theorem~\ref{THM_AT_Spectrum}.
	
	\begin{definition} \label{DEF_aT}
		We define the spectrum
		$$
			\spec(\aT) := \set{\kappa > \aleph_0}{\text{There is a partition of } \cantorspace \text{ into } \kappa \text{-many closed sets}}.
		$$
		and define the cardinal characteristic $\aT := \min(\spec(\aT))$.
	\end{definition}
	
	We arbitrarily fixed $\cantorspace$ as our Polish space of choice here.
	However, Miller proved that a witness for $\aleph_1 \in \spec(\aT)$ does not depend on the underlying Polish space:
	
	\begin{theorem}[Miller, 1980, \cite{Miller_1980}]
		There is a partition of $\cantorspace$ into $\aleph_1$-many closed sets iff there is a partition of some Polish space into $\aleph_1$-many closed sets iff every Polish space has a partition into $\aleph_1$-many closed sets.
	\end{theorem}
 	
 	More generally, Spinas proved in \cite{Spinas_1997} that $\aT$ is independent of the underlying Polish space and that $\d \leq \aT$.
 	Brian extended this result in the following way:
	
	\begin{theorem}[\cite{Brian_2021}]
		Let $\kappa$ be an uncountable cardinal.
		Then, all six statements of the following form are equivalent:
		$$
			\text{Some/Every uncountable Polish space can be partitioned into } \kappa \text{ compact/closed/}F_\sigma \text{-sets.}
		$$
	\end{theorem}
	
	Hence, neither the cardinal characteristic $\aT$ nor its spectrum $\spec(\aT)$ depend on the underlying Polish space, or if partitions into compact, closed or $F_\sigma$-sets are considered.
	In order to force a desired constellation of $\spec(\aT)$, we will add partitions of $\cantorspace$ into $F_\sigma$-sets.
	To this end, we will use the usual identification of non-empty closed set of $\cantorspace$ and branches of trees:
	
	\begin{definition} \label{DEF_Tree}
		A tree $T$ is a non-empty subset of $\fincantorspace$ such that
		\begin{enumerate}
			\item for all $s \in \fincantorspace$ and $t \in T$ with $s \subsequence t$ we have $s \in T$,
			\item for all $s \in T$ we have $s \concat 0 \in T$ or $s \concat 1 \in T$ (or both).
		\end{enumerate}
		We denote with $[T]$ the set of branches of $T$:
		$$
			[T] := \set{f \in \cantorspace}{\text{for all } n < \omega \text{ we have } f \restr n \in T}.
		$$
		We call $T$ nowhere dense if it additionally satisfies
		\begin{enumerate}
			\setcounter{enumi}{2}
			\item for all $s \in T$ there is a $t \in \fincantorspace$ with $s \subsequence t$ and $t \notin T$.
		\end{enumerate}
	\end{definition}
	
	\begin{remark} \label{REM_Identify_Trees_Closed_Sets}
		Given a tree $T$, the set $[T]$ is a non-empty closed set of $\cantorspace$.
		Conversely, given any non-empty closed set $C$ the set
		$$
			\tree(C) := \set{s \in \fincantorspace}{\text{there is an } f \in C \text{ with } s \subsequence f}
		$$
		is a non-empty tree.
		Since, $[\tree(C)] = C$ and $\tree([T]) = T$ we may identify trees and non-empty closed sets of $\cantorspace$ under these bijections.
		Furthermore, if $T$ is nowhere dense, then also $[T]$ is nowhere dense and conversely if $C$ is nowhere dense, then also $\tree(C)$ is nowhere dense.
		Hence, this identification restricts to nowhere dense trees and nowhere dense closed subsets.
	\end{remark}

	\begin{definition} \label{DEF_Almost_Disjoint}
		Let $S, T$ be trees.
		We call $S$ and $T$ almost disjoint iff $S \cap T$ is finite.
	\end{definition}

	Note that by König's lemma two trees $S$ and $T$ are almost disjoint exactly iff $[S] \cap [T] = \emptyset$.
	In order to force the existence of a witness for $\kappa \in \spec(\aT)$, we will add $\kappa$-many countable families $\set{\T_\alpha}{\alpha < \kappa}$ of nowhere dense trees which satisfy
	\begin{enumerate}
		\item for all $\alpha < \beta < \kappa$ and $S \in \T_\alpha, T \in \T_\beta$ the trees $S$ and $T$ are almost disjoint,
		\item for all $f \in \cantorspace$ there is an $\alpha < \kappa$ with $f \in \bigcup_{T \in \T_\alpha} [T]$.
	\end{enumerate}
	Notice that for $\alpha < \kappa$ and $S \neq T \in \T_\alpha$ we do not require that $S$ and $T$ are almost disjoint.
	However, the two conditions above imply that $\set{\bigcup_{T \in \C_\alpha} [T]}{\alpha < \kappa}$ is a partition of $\cantorspace$ into $\kappa$-many $F_\sigma$-sets.
	Next, in order to approximate new nowhere dense trees with finite conditions we fix the following notions.
	
	\begin{definition} \label{DEF_nTree}
		Let $n < \omega$.
		An $n$-tree $T$ is a non-empty subset of ${^{\leq n}{2}}$ such that
		\begin{enumerate}
			\item for all $s \in {^{\leq n}{2}}$ and $t \in T$ with $s \subsequence t$ we have $s \in T$,
			\item for all $s \in T$ there is a $t \in T \cap {^{n}{2}}$ with $s \subsequence t$.
		\end{enumerate}
		We denote with $[T]$ the set of leaves $T \cap {^{n}{2}}$ of $T$.
		Given $n \leq m$, an $n$-tree $S$ and an $m$-tree $T$ we write $S \subsequence T$ iff $T$ end-extends $S$, i.e.\ $T \cap {^{\leq n}{2}} = S$.
	\end{definition}

	\begin{definition} \label{DEF_Forcing_Witness}
		Let $\T$ be a family of nowhere dense trees.
		We define the forcing $\TT_0(\T)$ to be the set of all pairs $p = (T_p, F_p)$, where $T_p$ is an $n_p$-tree for some $n_p < \omega$ and $F_p \subseteq \cantorspace$ is finite such that for all $f \in F_p$ we have $f \notin \bigcup_{T \in \T}[T]$ and  $f \restr n_p \in [T_p]$.
		
		Given two conditions $p, q \in \TT_0(\T)$ we define $q \extends p$ iff $n_p \leq n_q$, $F_p \subseteq F_q$ and $T_p \subsequence T_q$.
		Further, we define $\TT(\T)$ to be the finitely supported product of size $\omega$
		$$
			\TT(\T) := \prod_{\omega}\TT_0(\T).
		$$
	\end{definition}
	
	We just summarize the crucial properties of $\TT_0(\T)$ and $\TT(\T)$ as they follow from standard density and forcing arguments.
	See \cite{FischerSchembecker_2022} for more details for a very similar forcing.
	
	\begin{remark} \label{REM_Diagonalization}
		$\TT_0(\T)$ is $\sigma$-centered, so also $\TT(\T)$ is $\sigma$-centered.
		Further, if $G$ is $\TT_0(\T)$-generic in $V[G]$ the set
		$$
			T^G := \bigcup \set{T_p}{p \in G}
		$$
		is a nowhere dense tree such that $T^G$ and $T$ are almost disjoint for all $T \in \T$.
		Analogously, if $G$ is $\TT(\T)$-generic we denote with $\seq{T^G_n}{n < \omega}$ the $\omega$-many new nowhere dense trees by $\TT(\T)$.
		We have the following diagonalization properties:
		\begin{enumerate}[(D1)]
			\item For all $n < \omega$ the tree $T^G_n$ is almost disjoint from every $T \in \T$.
			\item For all $f \in (\cantorspace)^V$ with $f \notin \bigcup_{T \in \T}[T]$ we have $f \in \bigcup_{n < \omega} [T^G_n]$.
		\end{enumerate}
		Note that in general $T^G_n$ and $T^G_m$ need not be almost disjoint for $n \neq m$.
		The diagonalization properties immediately yield the following lemma:
	\end{remark}

	\begin{lemma} \label{LEM_Realize_Witness}
		Let $\kappa$ be an uncountable cardinal.
		Then, there is a c.c.c.\ forcing which forces the existence of a witness for $\kappa \in \spec(\aT)$.
	\end{lemma}
	
	\begin{proof}
		Sketch.
		Consider the following iteration:
		Start with the finitely supported product of Cohen forcing of size $\kappa$.
		In the generic extension, let $\T_1$ be the set $\seq{T_\alpha}{\alpha < \kappa}$, where $T_\alpha$ is the nowhere dense tree with only branch the $\alpha$-th Cohen real.
		Then, force with $\TT(\T_1)$ to obtain $\omega$-many new nowhere dense trees $\seq{T_n}{n < \omega}$ with properties $(\text{D1})$ and $(\text{D2})$ in Remark~\ref{REM_Diagonalization}.
		Extend $\T_1$ to $\T_2 := \T_1 \cup \set{T_n}{n < \omega}$ and continue iterating $\TT(\T_\alpha)$ the same way $\aleph_1$-many times with finite support.
		In the end we obtain $\kappa + \aleph_1 = \kappa$-many $F_\sigma$-sets which are disjoint by $(\text{D1})$ and cover $\cantorspace$ by $(\text{D2})$ and since $\aleph_1$ has uncountable cofinality.
	\end{proof}
	
	In order to realize a whole spectrum of $\aT$, in Definition~\ref{DEF_Forcing} we define our forcing as a product of a slightly tweaked version of this iteration.
	A bookkeeping argument and Lemma~\ref{LEM_Realize_Witness} may be used to force $\aT = \kappa$ and $\c = \lambda$ for any regular $\kappa$ and $\lambda > \kappa$ of uncountable cofinality \cite{FischerSchembecker_2022}.
	Further, since $\d \leq \aT$, for $\kappa$ of uncountable cofinality any model of $\d = \kappa = \c$ satifies $\aT = \kappa$.
	However, this leaves open the following question:
	\begin{question}
		Can $\aT$ be singular of uncountable cofinality and $\aT < \c$?
	\end{question}
	In \cite{Brendle_2003} Brendle constructed a model of $\a = \aleph_{\aleph_0}$.
	While we may use Lemma~\ref{LEM_Realize_Witness} to force the existence of a witness for $\aleph_{\aleph_0} \in \spec(\aT)$ the following question is still open:
	
	\begin{question}
		Can $\aT$ be of countable cofinality? In particular is $\aT = \aleph_{\aleph_0}$ consistent?
	\end{question}

	Note that $\d < \aT$ must hold in such a model as $\d$ can only have uncountable cofinality.
	
	\section{Realizing arbitrarily large spectra of $\aT$} \label{SEC_Sketch}
	
	The culmination of this paper is the following Main Theorem.
	In this section we will describe the proof ingredients and summarize the role of each section towards this goal.
	
	\begin{maintheorem}\label{THM_AT_Spectrum}
		Assume $\sf{GCH}$ and let $\Theta$ be a set of uncountable cardinals such that
		\begin{enumerate}[\normalfont (I)]
			\item $\max(\Theta)$ exists and has uncountable cofinality,
			\item $\Theta$ is closed under singular limits,
			\item If $\theta \in \Theta$ with $\cof(\theta) = \omega$, then $\theta^+ \in \Theta$,
			\item $\aleph_1 \in \Theta$.
		\end{enumerate}
		Then, there is a c.c.c.\ forcing extension in which $\spec(\aT) = \Theta$ holds.
	\end{maintheorem}

	Our proof strategy is inspired by Shelah's and Spinas' work on the spectrum of mad families mentioned above.
	They realize a desired spectrum $\Theta$ with a large product of Hechler forcing.
	In order to exclude cardinalities, their proof essentially hinges on the following fact:
	
	\begin{fact}\label{FAC_Hechler}
		Let $\HH_I$ be Hechler's forcing for adding an almost disjoint family indexed by $I$.
		If $I \subseteq J$ then $\HH_I \completesubposet \HH_J$, i.e.\ $\HH_I$ is a complete subforcing of $\HH_J$.
	\end{fact}

	Their isomorphism-of-names argument then uses this fact in this product-like setting by reducing to countable subforcings of their whole forcing and using appropriate isomorphisms between these countable subforcings.
	In contrast, the isomorphism-of-names argument by Brian mentioned in his theorem above directly employs automorphisms of the whole forcing, which is less flexible.
	In our proof of our Main Theorem~\ref{THM_AT_Spectrum} we adapt the proof strategy of Shelah and Spinas for the iteration-like situation of $\aT$ instead of the product-like situation of $\a$.
	Consequently, this paper is structured as follows:
	
	In Section~\ref{SEC_Forcing} we define the c.c.c.\ forcing (see Definition~\ref{DEF_Forcing}) which yields Main Theorem~\ref{THM_AT_Spectrum}.
	Similarly to Brian's forcing in~\cite{Brian_2021}, our forcing adds a witness for $\theta \in \spec(\aT)$ for every $\theta \in \Theta$.
	However, in contrast we define our forcing directly as an iteration.
	Moreover, we fix a larger family of trees after adding many Cohen reals in the first step of our iteration.
	Hence, the family of trees is closed under more automorphisms of the initial Cohen forcing.
	In Section~\ref{SEC_Extend_Automorphisms} we provide a very algebraic framework, how to extend automorphisms of the Cohen forcing to automorphisms of the entire forcing, which culminates in Corollary~\ref{COR_Induced_Sequence_Of_Actions}.
	Next, throughout the paper we will need to work with nice conditions of our iteration, which describes all the forcing information in a given condition.
	Hence, in Section~\ref{SEC_Dense_Subset} we inductively define the notion of a nice condition (see Definition~\ref{DEF_Nice_Condition}) and prove their density in Lemma~\ref{LEM_Ds_Are_Dense}.
	We also define the hereditary support of a condition (see Definition~\ref{DEF_Hereditary_Support}) and study the behaviour of nice conditions under the automorphisms described in Section~\ref{SEC_Extend_Automorphisms} (see Lemma~\ref{LEM_Ds_Fixed_Under_Action} and Lemma~\ref{LEM_Condition_Below_Stable_For_Automorphism}).
	
	Section~\ref{SEC_Complete_Embeddings} is the heart of the entire proof.
	We show that our forcing from Section~\ref{SEC_Forcing} has enough complete subforcings to imitate the isomorphism-of-names argument by Shelah and Spinas.
	However, we do not obtain a direct analogue to Fact~\ref{FAC_Hechler} above, but the slightly weaker Theorem~\ref{THM_Complete_Subforcings}:
	
	\begin{theorem*}
		Let $\Phi \subseteq \Psi$ be a $\Theta$-subindexing function and assume $\Phi$ is countable.
		Then, $\PP_{\alpha}^\Phi \completesubposet \PP_{\alpha}^\Psi$ for all $\alpha \leq \aleph_1$.
	\end{theorem*}

	Hence, with our current methods we can only show that we have complete subforcings if the index set is sufficiently small (countable in the sense of Definition~\ref{DEF_Indexing_Function}) and the iteration is at most of length $\aleph_1$.
	This is precisely where we require the strengthening of (IV) in our Main Theorem~\ref{THM_AT_Spectrum}.
	In other words, if Theorem~\ref{THM_Complete_Subforcings} can be proven for longer iterations, requirement (IV) can be again relaxed to the requirement \textquoteleft$\min(\Theta)$ is regular\textquoteright, which would yield a full answer to Brian's question.
	
	Theorem~\ref{THM_Complete_Subforcings} is proved in an elaborate inductive fashion.
	First, in order to show that the embedding of some subforcing is well-defined we will need the additional automorphisms our forcing possesses due to our modifications.
	Secondly, in order to show that these embeddings are indeed complete, we introduce the technical notion of a nice name for a finite set of reals with respect to a sequence of names for trees (see Definition~\ref{DEF_Nice_Name_WRT_Trees}).
	Then, the canonical projection of such a nice name (see Lemma~\ref{LEM_Reduction_Of_Nice_Name}) will have the desired properties in order to define a reduction of a condition in our forcing (see Lemma~\ref{LEM_Projection_Of_Nice_WRT_Sequence}).

	Finally, in Section~\ref{SEC_Extending_Isomorphisms} we give an algebraic/categorical analysis of isomorphisms between the complete subforcings given by Theorem~\ref{THM_Complete_Subforcings}.
	We then put everything together and provide the remaining isomorphism-of-names argument needed for Main Theorem~\ref{THM_AT_Spectrum} in Section~\ref{SEC_Proof}.
	
	\section{Defining the iteration} \label{SEC_Forcing}
	
	In this section we define the forcing used to prove Main Theorem~\ref{THM_AT_Spectrum}.
	For the remainder of this paper let $\Theta$ be fixed as in Main Theorem~\ref{THM_AT_Spectrum}.
	
	\begin{definition} \label{DEF_Indexing_Function}
		A $\Theta$-indexing function is a partial function $\Phi: \Theta \to V$.
		For two $\Theta$-indexing functions $\Phi, \Psi$, we write $\Phi \subseteq \Psi$ iff $\Phi$ is a $\Theta$-subindexing function of $\Psi$, i.e.\ $\dom(\Phi) \subseteq \dom(\Psi)$ and for all $\theta \in \dom(\Phi)$ we have $\Phi(\theta) \subseteq \Psi(\theta)$.
		Finally, we call a $\Theta$-indexing $\Phi$ function countable iff $\dom(\Phi)$ is countable and for every $\theta \in \dom(\Phi)$ we have that $\Phi(\theta)$ is countable.
	\end{definition}

	\begin{definition} \label{DEF_Indexed_Cohen}
		Let $\Phi$ be a $\Theta$-indexing function.
		Define $\CC^{\Phi}$ to be the partial order adding new Cohen reals indexed by pairs $(\theta, i)$ where $\theta \in \dom(\Theta)$ and $i \in \Phi(\theta)$, i.e.\
		$$
			\CC^\Phi := \set{s:\bigcup_{\theta \in \dom(\Phi)} (\simpleset{\theta} \times \Phi(\theta)) \to \CC}{\supp(s) \text{ is finite}}.
		$$
		Further, we write $\dot{c}^{\Phi, \theta}_i$ for the canonical $\CC^\Phi$-name for the Cohen real indexed by $(\theta, i)$ and $\dot{T}^{\Phi, \theta}_i$ for the canonical $\CC^\Phi$-name for the tree with only branch $\dot{c}^{\Phi, \theta}_i$.
	\end{definition}
	
	\begin{remark} \label{REM_Strong_Projection}
		Clearly, if $\Phi \subseteq \Psi$ we have that $\CC^{\Phi} \completesubposet \CC^{\Psi}$.
		In fact there is a strong projection from $\CC^\Psi$ onto $\CC^\Phi$, which just forgets all Cohen information outside of $\Phi$'s indexing.
		We denote this complete embedding by $\iota^{\Phi, \Psi}:\CC^\Phi \to \CC^\Psi$.
		Notice that for $\theta \in \dom(\Phi)$ and $i \in \Phi(\theta)$ we have 
		$$
			\iota^{\Phi, \Psi}(\dot{c}_i^{\Phi, \theta}) = \dot{c}^{\Psi, \theta}_i \text{ and } \iota^{\Phi, \Psi}(\dot{T}_i^{\Phi, \theta}) = \dot{T}^{\Psi, \theta}_i.
		$$
	\end{remark}
	
	$\CC^\Phi$ will be the first step of our iteration. 
	Note that $\CC^\Phi$ has a vast amount of automorphisms and we need to extend some of these automorphisms through our iteration.
	In fact, we will need even more - we also need to preserve the group structure of the automorphisms.
	Hence, it is very natural to use the language of group actions and morphisms between group actions to express these properties.
	
	\begin{definition} \label{DEF_Action}
		Let $\Gamma$ denote the group $\bigoplus_{\omega}\ZZ / 2$ with group operation $+$.
		We define a group action $\Gamma \actson \CC$ for $\gamma \in \Gamma$, $s \in \CC$ by $\dom(\gamma.s) := \dom(s)$ and for $n \in \dom(s)$
		$$
			(\gamma.s)(n) :=
			\begin{cases}
				s(n) & \text{if } \gamma(n) = 0,\\
				1 - s(n) & \text{otherwise}.
			\end{cases}
		$$
		Hence, an element $\gamma \in \Gamma$ flips the Cohen information at place $n$ precisely iff $\gamma(n) = 1$.
	\end{definition}

	\begin{remark} \label{DEF_Action_Preserves_Order}
		Note that the action $\Gamma \actson \CC$ preserves the order, that is $\gamma.s \leq \gamma.t$ for all $s,t \in \CC$ with $s \extends t$.
		In other words, the action $\Gamma \actson \CC$ is equivalent to a group homomorphism from $\pi:\Gamma \to \Aut(\CC)$.
		Also, remember that every automorphism of $\CC$ is an involution.
	\end{remark}

	\begin{definition} \label{DEF_Action_Induced}
		Let $\Phi$ be a $\Theta$-indexing function, $\theta \in \dom(\Phi)$ and $i \in \Phi(\theta)$.
		Then, we have an induced group action of $\Gamma$ acting on the $(\theta, i)$-th component of $\CC^\Phi$, which we denote with $\Gamma \smash{\overset{\theta, i}{\actson}} \CC^\Phi$.
		In other words, we have that the inclusion map $\iota^{\Phi, \theta}_i: \CC \to \CC^\Phi$ is a morphism of $\Gamma$-sets, i.e.\ the following diagram commutes for every $\gamma \in \Gamma$:
		\begin{center}
		\begin{tikzcd}
			\CC \arrow[r, "{\iota^{\Phi, \theta}_i}"] \arrow[d, "\pi(\gamma)"] & \CC^{\Phi} \arrow[d, "{\pi^{\Phi, \theta}_i(\gamma)}"] \\
			\CC \arrow[r, "{\iota^{\Phi, \theta}_i}"]                          & \CC^{\Phi}                                            
		\end{tikzcd}
		\end{center}
		where $\pi^{\Phi, \theta}_i$ is the group homomorphism corresponding to $\Gamma \smash{\overset{\theta, i}{\actson}} \CC^\Phi$.
	\end{definition}

	\begin{remark} \label{REM_Morphism_Of_Gamma_Sets_On_Cohen}
		Since we have various group actions of $\Gamma \actson \CC^\Phi$, we will usually use the corresponding group homomorphisms $\pi^{\Phi, \theta}_i:\Gamma \to \Aut(\CC^\Phi)$ to avoid confusion.
		Also, note that more generally for any $\Theta$-subindexing function $\Phi \subseteq \Psi$, $\theta \in \dom(\Phi)$ and $i \in \Phi(\theta)$ we have that $\iota^{\Phi, \Psi}$ is a morphisms of $\Gamma$-sets, i.e.\ the following diagram commutes for every $\gamma \in \Gamma$:
		\begin{center}
		\begin{tikzcd}
			\CC^\Phi \arrow[r, "{\iota^{\Phi, \Psi}}"] \arrow[d, "{\pi^{\Phi, \theta}_i(\gamma)}"] & \CC^{\Psi} \arrow[d, "{\pi^{\Psi, \theta}_i(\gamma)}"] \\
			\CC^\Phi \arrow[r, "{\iota^{\Phi, \Psi}}"]                                             & \CC^{\Psi}                                            
		\end{tikzcd}
		\end{center}
	\end{remark}

	\begin{definition} \label{DEF_Names_For_Trees}
		Let $\Phi$ be a $\Theta$-indexing function, $\theta \in \dom(\Phi)$ and $i \in \Phi(\theta)$.
		Denote with $\dot{\T}^{\Phi, \theta}_{i}$ the canonical $\CC^\Phi$-name for the set
		$$
			\set{\pi^{\Phi, \theta}_i(\gamma)(\dot{T}^{\Phi, \theta}_i)}{\gamma \in \Gamma}.
		$$
		Similarly, we let $\dot{\T}^{\Phi, \theta}$ denote the canonical $\CC^\Phi$-name for the set
		$$
			\set{\pi^{\Phi, \theta}_i(\gamma)(\dot{T}^{\Phi, \theta}_i)}{i \in \Phi(\theta) \text{ and } \gamma \in \Gamma}.
		$$
	\end{definition}

	\begin{remark} \label{REM_Properties_Of_Automorphism_Subgroup}
		Since $\Gamma$ is countable, also $\dot{\T}^{\Phi, \theta}_{i}$ is countable and $\dot{\T}^{\Phi, \theta}$ of size $\left|\Phi(\theta)\right| \cdot \aleph_0$, hence countable in case that $\Phi$ is countable.
		Further, using Remark~\ref{REM_Strong_Projection}~and~\ref{REM_Morphism_Of_Gamma_Sets_On_Cohen} it is easy to verify the following properties for every $\Theta$-subindexing function $\Phi \subseteq \Psi$, $\theta \in \dom(\Phi)$ and $i \in \Phi(\theta)$:
		\begin{enumerate}[$\bullet$]
			\item $\dot{\T}^{\Phi, \theta}$ is the canonical $\CC^\Phi$-name for $\bigcup_{i \in \Phi(\theta)}\dot{\T}^{\Phi, \theta}_{i}$,
			\item $\iota^{\Phi, \Psi}(\dot{\T}^{\Phi, \theta}_{i}) = \dot{\T}^{\Psi, \theta}_{i}$,
			\item $\CC^\Phi \forces \bigcup_{T \in \dot{\T}^{\Phi, \theta}_{i}} [T] = \set{f \in \cantorspace}{f =^* \dot{c}^{\Phi, \theta}_i}$.
		\end{enumerate}
	\end{remark}

	Next, given a $\Theta$-indexing function $\Phi$ we define the forcing iteration realizing the desired spectrum of $\aT$ for Main Theorem~\ref{THM_AT_Spectrum}.
	The forcing is a finite support iteration of c.c.c.\ forcings of length $\aleph_1$:

	\begin{definition} \label{DEF_Forcing}
		Let $\Phi$ be a $\Theta$-indexing function.
		We will define a finite support iteration $\seq{\PP^{\Phi}_\alpha, \dot{\QQ}^{\Phi}_\beta}{\alpha \leq \aleph_1, \beta < \aleph_1}$, $\PP_{\alpha + 1}^\Phi$-names $\dot{T}^{\Phi, \theta}_{\alpha, n}$ for nowhere dense trees for $\theta \in \dom(\Phi)$, $0 < \alpha < \aleph_1$ and $n < \omega$, and $\PP_\alpha^\Phi$-names $\dot{\T}^{\Phi, \theta}_{\alpha}$ for families of nowhere dense trees for $\theta \in \dom(\Phi)$ and $0 < \alpha \leq \aleph_1$:
		
		\begin{enumerate}[$\bullet$]
			\item Let $\dot{\QQ}^\Phi_0$ be the forcing $\CC^{\Phi}$.
			Then, we already defined the $\CC^\Phi$-names $\dot{\T}^{\Phi, \theta}$ in Definition~\ref{DEF_Names_For_Trees} for every $\theta \in \dom(\Phi)$.
			Then, let $\dot{\T}^{\Phi, \theta}_1$ the corresponding canonical $\PP^\Phi_1$-names.
			\item For $\alpha > 0$ let $\dot{\QQ}^{\Phi}_\alpha$ be the canonical $\PP^{\Phi}_\alpha$-name for the finitely supported product
			$$
				\prod_{\theta \in \supp(\Phi)} \TT(\T^{\Phi, \theta}_{\alpha}).
			$$
			Also, for every $\theta \in \dom(\Phi)$ let $\dot{T}^{\Phi, \theta}_{\alpha, n}$ be the canonical $\PP^\Phi_{\alpha + 1}$-names for the $\omega$-many new nowhere dense trees added by $\TT(\T^{\Phi, \theta}_{\alpha})$, where $n < \omega$.
			Finally, let $\dot{\T}^{\Phi, \theta}_{\alpha + 1}$ be the canonical $\PP^{\Phi}_{\alpha + 1}$-name for $\dot{\T}^{\Phi, \theta}_{\alpha} \cup \set{\dot{T}^{\Phi, \theta}_{\alpha, n}}{n \in \omega}$.
			\item At limit $\alpha$ for every $\theta \in \dom(\Phi)$ let $\dot{\T}^{\Phi, \theta}_\alpha$ be the canonical $\PP^{\Phi}_\alpha$-name for $\bigcup_{\beta < \alpha} \dot{\T}^{\Phi, \theta}_\beta$.
		\end{enumerate}
	\end{definition}
	
	Grouping together the $\omega$-many new trees added at each successor step into one $F_\sigma$-set, we have that for every $\theta \in \dom(\Phi)$ the family $\dot{\T}^{\Phi, \theta}_{\aleph_1}$ will be witness of a partition of Cantor space into $F_\sigma$-sets of size $\left|\Phi(\theta)\right| \cdot \aleph_1$.
	Thus, if every $\Phi(\theta)$ is a set of size $\theta$, then $\Theta \subseteq \spec(\aT)$ is forced by $\PP^{\Phi}_{\aleph_1}$ as in Lemma~\ref{LEM_Realize_Witness}.
	Thus, it only remains to prove the reverse inclusion.
	
	\section{Extending group actions through the iteration} \label{SEC_Extend_Automorphisms}
	
	Since $\PP^\Phi_1 \cong \CC^\Phi$, in the last section we essentially considered group actions $\Gamma \smash{\overset{\theta, i}{\actson}} \PP^\Phi_1$.
	In this section, we will show that there is a canonical way to extend these group actions through the iteration, i.e.\ to group actions $\Gamma \smash{\overset{\theta, i}{\actson}} \PP^{\Phi}_{\alpha}$ for $0 < \alpha \leq \aleph_1$,.
	This process leads to the notion of an induced sequence of group actions in Corollary~\ref{COR_Induced_Sequence_Of_Actions}.
	We write $\iota_1^{\Phi, \Psi}:\PP^\Phi_1 \to \PP^\Psi_1$ for the complete embedding corresponding to $\iota^{\Phi, \Psi}: \CC^\Phi \to \CC^\Psi$ and $\pi^{\Phi, \theta}_{1, i}:\Gamma \to \Aut(\PP^\Phi_1)$ for the group homomorphism corresponding to $\pi^{\Phi, \theta}_{i}:\Gamma \to \Aut(\CC^\Phi)$.
	
	\begin{definition} \label{DEF_Increasing_Sequence_Of_Actions}
		Let $\Phi$ be a $\Theta$-indexing function, $\theta \in \dom(\Phi)$, $i \in \Phi(\theta)$ and $\epsilon \leq \aleph_1$.
		We say that
		$$
			\seq{\pi^{\Phi, \theta}_{\alpha,i}:\Gamma \to \Aut(\PP^\Phi_\alpha)}{0 < \alpha \leq \epsilon}
		$$
		 is an increasing sequence of $\Gamma$-actions iff every $\pi^{\Phi, \theta}_{\alpha, i}$ is a group homomorphism (i.e.\ an action of $\Gamma$ on $\PP^\Phi_\alpha$), for all $0 < \alpha \leq \epsilon$, $\eta \in \dom(\Phi)$ and $\gamma \in \Gamma$ we have that
		 $$
		 	\pi^{\Phi, \theta}_{\alpha, i}(\gamma)(\dot{\T}^{\Phi, \eta}_{\alpha}) = \dot{\T}^{\Phi,\eta}_{\alpha}
		 $$
		 and for all $0 < \alpha \leq \beta \leq \epsilon$ the canonical embedding $\iota^\Phi_{\alpha, \beta}:\PP^\Phi_\alpha \to \PP^\Phi_\beta$ is a morphism of $\Gamma$-sets, i.e\ the following diagram commutes for every $\gamma \in \Gamma$:

		\begin{center}
		\begin{tikzcd}
			\PP^\Phi_\alpha \arrow[d, "{\pi^{\Phi,\theta}_{\alpha,i}(\gamma)}"] \arrow[r, "{\iota^\Phi_{\alpha,\beta}}"] & \PP^\Phi_\beta \arrow[d, "{\pi^{\Phi,\theta}_{\beta,i}(\gamma)}"] \\
			\PP^\Phi_\alpha \arrow[r, "{\iota^\Phi_{\alpha,\beta}}"]                                                     & \PP^\Phi_\beta                                                   
		\end{tikzcd}
		\end{center}
	\end{definition}

	Our goal for this section is to provide a canonical extension of $\pi^{\Phi, \theta}_{1, i}$ as defined in Definition~\ref{DEF_Action_Induced} to an increasing sequence of $\Gamma$-actions of length $\aleph_1$.
	Since the iterands of the forcing in Definition~\ref{DEF_Forcing} are definable from the parameters $\dot{\T}^{\Phi, \theta}_{\alpha}$, it is crucial that the group action fixes these parameters, which allows for an extension through the iteration.
	Before we consider the successor step,we show that for limit steps by the universal property of the direct limit there is a unique way to extend an increasing sequence of $\Gamma$-actions:

	\begin{lemma} \label{LEM_Limits_Of_Increasing_Sequences_Are_Unique}
		Let $\Phi$ be a $\Theta$-indexing function, $\theta \in \dom(\Phi)$, $i \in \Phi(\theta)$ and let $\epsilon \leq \aleph_1$ be a limit.
		Assume
		$$
			\seq{\pi^{\Phi, \theta}_{\alpha,i}:\Gamma \to \Aut(\PP^\Phi_\alpha)}{0 < \alpha < \epsilon}
		$$
		is an increasing sequence of $\Gamma$-actions.
		Then there is a unique group homomorphism $\pi^{\Phi, \theta}_{\epsilon,i}$ so that
		$$
			\seq{\pi^{\Phi, \theta}_{\alpha,i}:\Gamma \to \Aut(\PP^\Phi_\alpha)}{0 < \alpha \leq \epsilon}
		$$
		is an increasing sequences of $\Gamma$-actions.
	\end{lemma}

	\begin{proof}
		By definition of an increasing sequence of $\Gamma$-actions (cf. Definition~\ref{DEF_Increasing_Sequence_Of_Actions}) we have a directed system of maps
		$$
			\seq{\iota^\Phi_{\alpha, \epsilon} \circ (\pi^{\Phi, \theta}_{\alpha, i}(\gamma)):\PP^\Phi_\alpha \to \PP^\Phi_\epsilon}{0 < \alpha < \epsilon}
		$$
		By the universal property of $\PP^\Phi_\epsilon$ there is a  unique map $\pi^{\Phi, \theta}_{\epsilon,i}(\gamma):\PP^\Phi_\epsilon\to \PP^\Phi_\epsilon$, so that the following diagram commutes for every $0 < \alpha \leq \epsilon$ and $\gamma \in \Gamma$:
		
		\begin{center}
			\begin{tikzcd}
				\PP^\Phi_\alpha \arrow[d, "{\pi^{\Phi,\theta}_{\alpha,i}(\gamma)}"] \arrow[r, "{\iota^\Phi_{\alpha,\epsilon}}"] & \PP^\Phi_\epsilon \arrow[d, "{\pi^{\Phi,\theta}_{\epsilon,i}(\gamma)}"] \\
				\PP^\Phi_\alpha \arrow[r, "{\iota^\Phi_{\alpha,\epsilon}}"]                                                     & \PP^\Phi_\epsilon                                                   
			\end{tikzcd}
		\end{center}
		
		Next, fix $\gamma, \delta \in \Gamma$.
		We need to verify that $\pi^{\Phi, \theta}_{\epsilon,i}(\gamma) \circ \pi^{\Phi, \theta}_{\epsilon,i}(\delta) = \pi^{\Phi, \theta}_{\epsilon,i}(\gamma + \delta)$, so let $p \in \PP^{\Phi, \theta}_\epsilon$.
		Choose $\alpha < \epsilon$ such that $\iota^\Phi_{\alpha, \epsilon}(p \restr \alpha) = p$.
		Then, we compute
		\begin{align*}
			\pi^{\Phi, \theta}_{\epsilon,i}(\gamma)(\pi^{\Phi, \theta}_{\epsilon,i}(\delta)(p))
			&= \pi^{\Phi, \theta}_{\epsilon,i}(\gamma)(\pi^{\Phi, \theta}_{\epsilon,i}(\delta)(\iota^\Phi_{\alpha, \epsilon}(p \restr \alpha))) && (\text{choice of } \alpha)\\
			&= \pi^{\Phi, \theta}_{\epsilon,i}(\gamma)(\iota^\Phi_{\alpha,\epsilon}(\pi^{\Phi, \theta}_{\alpha,i}(\delta)(p \restr \alpha))) && (\text{choice of } \pi^{\Phi, \theta}_{\epsilon,i}(\gamma))\\
			&= \iota^\Phi_{\alpha,\epsilon}(\pi^{\Phi, \theta}_{\alpha,i}(\gamma)(\pi^{\Phi, 
			\theta}_{\alpha,i}(\delta)(p \restr \alpha)))&& (\text{choice of } \pi^{\Phi, \theta}_{\epsilon,i}(\gamma))\\
			&= \iota^\Phi_{\alpha,\epsilon}(\pi^{\Phi, \theta}_{\alpha,i}(\gamma + \delta)(p \restr \alpha)) && (\pi^{\Phi, \theta}_{\alpha, i}\text{ is group homomorphism})\\
			&= \pi^{\Phi, \theta}_{\epsilon,i}(\gamma + \delta)(\iota^\Phi_{\alpha,\epsilon}(p \restr \alpha)) && (\text{choice of } \pi^{\Phi, \theta}_{\epsilon,i}(\gamma))\\
			&= \pi^{\Phi, \theta}_{\epsilon,i}(\gamma + \delta)(p) && (\text{choice of } \alpha).
		\end{align*}
		
		Thus, $\pi^{\Phi, \theta}_{\epsilon,i}:\Gamma \to \Aut(\PP^\Phi_\epsilon)$ is a group homomorphism.
		Finally, by Definition~\ref{DEF_Forcing} $\dot{\T}^{\Phi, \eta}_{\epsilon}$ is the canonical name for $\bigcup_{\alpha < \epsilon} \iota^{\Phi}_{\alpha, \epsilon}(\dot{\T}^{\Phi, \eta}_{\alpha})$.
		Thus, for any $\gamma \in \Gamma$ and $\eta \in \dom(\Phi)$ we compute
		\begin{align*}
			\pi^{\Phi, \theta}_{\epsilon, i}(\gamma)(\dot{\T}^{\Phi, \eta}_{\epsilon}) 
			&= \pi^{\Phi, \theta}_{\epsilon, i}(\gamma)(\bigcup_{\alpha < \epsilon} \iota^{\Phi}_{\alpha, \epsilon}(\dot{\T}^{\Phi, \eta}_{\alpha})) && (\text{Definition~\ref{DEF_Forcing}})\\
			&= \bigcup_{\alpha < \epsilon} \pi^{\Phi, \theta}_{\epsilon, i}(\gamma)(\iota^{\Phi}_{\alpha, \epsilon}(\dot{\T}^{\Phi, \eta}_{\alpha})) && (\text{canonical name})\\
			&= \bigcup_{\alpha < \epsilon} \iota^{\Phi}_{\alpha, \epsilon}(\pi^{\Phi, \theta}_{\alpha, i}(\gamma)(\dot{\T}^{\Phi, \eta}_{\alpha})) && (\text{choice of } \pi^{\Phi, \theta}_{\epsilon,i}(\gamma))\\
			&= \bigcup_{\alpha < \epsilon} \iota^{\Phi}_{\alpha, \epsilon}(\dot{\T}^{\Phi, \eta}_{\alpha}) && (\text{Definition~\ref{DEF_Increasing_Sequence_Of_Actions}})\\
			&= \dot{\T}^{\Phi, \eta}_{\alpha} && (\text{Definition~\ref{DEF_Forcing}}).\qedhere
		\end{align*}
	\end{proof}
	
	Next, we consider the successor case.
	In this case, there is no unique extension of the increasing sequence of $\Gamma$-actions.
	However, we prove that there is a canonical one in the following sense:
	
	\begin{definition}\label{DEF_Canonical_Extension}
		Let $\Phi$ be a $\Theta$-indexing function, $\theta \in \dom(\Phi)$, $i \in \Phi(\theta)$ and $\epsilon < \aleph_1$.
		Assume
		$$
			\seq{\pi^{\Phi, \theta}_{\alpha,i}:\Gamma \to \Aut(\PP^\Phi_\alpha)}{0 < \alpha \leq \epsilon}
		$$
		is an increasing sequence of $\Gamma$-actions.
		For every $\gamma \in \Gamma$ define $\pi^{\Phi, \theta}_{\epsilon + 1, i}(\gamma): \PP^{\Phi}_{\epsilon + 1} \to \PP^\Phi_{\epsilon + 1}$ by 
		$$
			\pi^{\Phi, \theta}_{\epsilon + 1, i}(\gamma)(p) := \pi^{\Phi, \theta}_{\epsilon, i}(\gamma)(p \restr \epsilon) \concat \pi^{\Phi, \theta}_{\epsilon, i}(\gamma)(p(\epsilon)).
		$$
		Then, we call $\pi^{\Phi, \theta}_{\epsilon + 1, i}$ the canonical extension of $\seq{\pi^{\Phi, \theta}_{\alpha,i}:\Gamma \to \Aut(\PP^\Phi_\alpha)}{0 < \alpha \leq \epsilon}$.
	\end{definition}
	
	\begin{lemma} \label{LEM_Canonical_Extension}
		Let $\Phi$ be a $\Theta$-indexing function, $\theta \in \dom(\Phi)$, $i \in \Phi(\theta)$ and $\epsilon < \aleph_1$.
		Assume
		$$
			\seq{\pi^{\Phi, \theta}_{\alpha,i}:\Gamma \to \Aut(\PP^\Phi_\alpha)}{0 < \alpha \leq \epsilon}
		$$
		is an increasing sequence of $\Gamma$-actions and let $\pi^{\Phi, \theta}_{\epsilon + 1,i}$ be the canonical extension.
		Then
		$$
			\seq{\pi^{\Phi, \theta}_{\alpha,i}:\Gamma \to \Aut(\PP^\Phi_\alpha)}{0 < \alpha \leq \epsilon + 1}
		$$
		is an increasing sequence of $\Gamma$-actions.
	\end{lemma}

	\begin{proof}
		First, by definition of an increasing sequence of $\Gamma$-actions (cf. Definition~\ref{DEF_Increasing_Sequence_Of_Actions}) for every $\eta \in \dom(\Phi)$ and $\gamma \in \Gamma$ we have
		$$
			\pi^{\Phi, \theta}_{\epsilon,i}(\gamma)(\dot{\T}^{\Phi, \eta}_\epsilon) = \dot{\T}^{\Phi, \eta}_\epsilon.
		$$
		By Definition~\ref{DEF_Forcing} $\dot{\QQ}^\Phi_\epsilon$ is the canonical $\PP^\Phi_\epsilon$-name for $\prod_{\eta \in \dom(\Phi)}\TT(\T^{\Phi, \eta}_\epsilon)$.
		Thus, we obtain
		$$
			\pi^{\Phi, \theta}_{\epsilon,i}(\gamma)(\dot{\QQ}^\Phi_\epsilon) = \dot{\QQ}^\Phi_\epsilon,
		$$
		as both $\prod_{\eta \in \dom(\Phi)}\TT(\T^{\Phi, \eta}_\epsilon)$ as well as the order $\leq$ are definable from the parameters $\T^{\Phi, \eta}_\epsilon$.
		Thus, we get $\pi^{\Phi, \theta}_{\epsilon + 1,i} \in \Aut(\PP^\Phi_{\epsilon + 1})$.
		Next, we verify that for every $\gamma \in \Gamma$ the following diagram commutes:
		\begin{center}
		\begin{tikzcd}
			\PP^\Phi_\epsilon \arrow[d, "{\pi^{\Phi,\theta}_{\epsilon,i}(\gamma)}"] \arrow[r, "{\iota^\Phi_{\epsilon,\epsilon+1}}"] & \PP^\Phi_{\epsilon + 1} \arrow[d, "{\pi^{\Phi,\theta}_{\epsilon+1,i}(\gamma)}"] \\
			\PP^\Phi_\epsilon \arrow[r, "{\iota^\Phi_{\epsilon,\epsilon+1}}"]                                                       & \PP^\Phi_{\epsilon + 1}                                                        
		\end{tikzcd}
		\end{center}
		Let $\gamma \in \Gamma$ and $p \in \PP^{\Phi}_\epsilon$.
		Then, we compute
		\begin{align*}
			\pi^{\Phi,\theta}_{\epsilon + 1,i}(\gamma)(\iota^\Phi_{\epsilon,\epsilon+1}(p))
			&= \pi^{\Phi, \theta}_{\epsilon, i}(\gamma)(\iota^\Phi_{\epsilon,\epsilon+1}(p) \restr \epsilon) \concat \pi^{\Phi, \theta}_{\epsilon, i}(\gamma)(\iota^\Phi_{\epsilon,\epsilon+1}(p)(\epsilon)) && (\text{Definition~\ref{DEF_Canonical_Extension}})\\
			&= \pi^{\Phi, \theta}_{\epsilon, i}(\gamma)(p) \concat \pi^{\Phi, \theta}_{\epsilon, i}(\gamma)(\mathds{1}) && (\text{definition of } \iota^{\Phi}_{\epsilon, \epsilon + 1})\\
			&= \pi^{\Phi, \theta}_{\epsilon, i}(\gamma)(p) \concat \mathds{1} && (\pi^{\Phi, \theta}_{\epsilon, i}(\gamma) \in \Aut(\PP^\Phi_\epsilon))\\
			&= \iota^\Phi_{\epsilon,\epsilon+1}(\pi^{\Phi,\theta}_{\epsilon,i}(\gamma)(p)) && (\text{definition of } \iota^{\Phi}_{\epsilon, \epsilon + 1}).
		\end{align*}
		
		Now, let $\gamma, \delta \in \Gamma$.
		We need to verify that $\pi^{\Phi, \theta}_{\epsilon + 1,i}(\gamma) \circ \pi^{\Phi, \theta}_{\epsilon + 1,i}(\delta) = \pi^{\Phi, \theta}_{\epsilon + 1,i}(\gamma + \delta)$, so let $p \in \PP^{\Phi, \theta}_{\epsilon + 1}$.
		Then, we compute
		\begin{align*}
			\pi^{\Phi, \theta}_{\epsilon + 1,i}(\gamma)(\pi^{\Phi, \theta}_{\epsilon + 1,i}(\delta)(p))
			&= \pi^{\Phi, \theta}_{\epsilon + 1,i}(\gamma)(\pi^{\Phi, \theta}_{\epsilon, i}(\delta)(p \restr \epsilon) \concat \pi^{\Phi, \theta}_{\epsilon, i}(\delta)(p(\epsilon))) && (\text{Definition~\ref{DEF_Canonical_Extension}})\\
			&= \pi^{\Phi, \theta}_{\epsilon,i}(\gamma)(\pi^{\Phi, \theta}_{\epsilon, i}(\delta)(p \restr \epsilon)) \concat \pi^{\Phi, \theta}_{\epsilon,i}(\gamma)(\pi^{\Phi, \theta}_{\epsilon, i}(\delta)(p(\epsilon))) && (\text{Definition~\ref{DEF_Canonical_Extension}})\\
			&= \pi^{\Phi, \theta}_{\epsilon,i}(\gamma + \delta)(p \restr \epsilon) \concat \pi^{\Phi, \theta}_{\epsilon,i}(\gamma + \delta)(p(\epsilon)) && (\pi^{\Phi, \theta}_{\epsilon, i}\text{ is gr.hom.})\\
			&= \pi^{\Phi, \theta}_{\epsilon + 1,i}(\gamma + \delta)(p) && (\text{Definition~\ref{DEF_Canonical_Extension}}).
		\end{align*}
		
		Thus, $\pi^{\Phi, \theta}_{\epsilon + 1,i}:\Gamma \to \Aut(\PP^\Phi_\epsilon)$ is a group homomorphism.
		Finally, let $\eta \in \dom(\Phi)$ and $\gamma \in \Gamma$.
		By Definition~\ref{DEF_Forcing} $\dot{\T}_{\epsilon + 1}^{\Phi, \eta}$ is the canonical name for $\iota^\Phi_{\epsilon, \epsilon + 1}(\dot{\T}^{\Phi, \eta}_{\epsilon}) \cup \set{\dot{T}^{\Phi, \eta}_{\epsilon, n}}{n \in \omega}$.
		Since
		\begin{align*}
			\pi^{\Phi, \theta}_{\epsilon + 1,i}(\gamma)(\iota^\Phi_{\epsilon, \epsilon + 1}(\dot{\T}^{\Phi, \eta}_{\epsilon}))
			&= \iota^\Phi_{\epsilon, \epsilon + 1}(\pi^{\Phi, \theta}_{\epsilon,i}(\gamma)(\dot{\T}^{\Phi, \eta}_{\epsilon})) && (\text{by commutativity above}) \\
			&= \iota^\Phi_{\epsilon, \epsilon + 1}(\dot{\T}^{\Phi, \eta}_{\epsilon}) && (\text{by Definition~\ref{DEF_Increasing_Sequence_Of_Actions}}),
		\end{align*}
		it suffices to verify that for all $n < \omega$ we have
		$$
			\pi^{\Phi, \theta}_{\epsilon + 1,i}(\gamma)(\dot{T}^{\Phi, \eta}_{\epsilon,n}) = \dot{T}^{\Phi, \eta}_{\epsilon,n}.
		$$
		But this follows since $\dot{T}^{\Phi, \eta}_{\epsilon,n}$ is the canonical $\PP^\Phi_{\epsilon + 1}$-name for the $n$-th new nowhere dense trees added by $\TT(\T^{\Phi, \theta}_{\epsilon})$ and check-names are fixed by any automorphism; remember that $\dot{T}^{\Phi, \eta}_{\epsilon,n}$ is just canonical name the union of the finite approximations in the generic filter.
	\end{proof}

	\begin{lemma}\label{LEM_Initial_Automorphism_Preservation}
		Let $\Phi$ be a $\Theta$-indexing function, $\theta \in \dom(\Phi)$, $i \in \Phi(\theta)$.
		Then, for every $\eta \in \dom(\Phi)$ and $\gamma \in \Gamma$ we have $\pi^{\Phi, \theta}_{1,i}(\gamma)(\dot{\T}^{\Phi, \eta}_1) = \dot{\T}^{\Phi, \eta}_1$, where $\pi^{\Phi, \theta}_{1,i}$ is defined as in Definition~\ref{DEF_Action_Induced}.
		In other words
		$$
			\seq{\pi^{\Phi, \theta}_{\alpha,i}:\Gamma \to \Aut(\PP^\Phi_\alpha)}{0 < \alpha \leq 1}
		$$
		is an increasing sequence of $\Gamma$-actions (of length 1).
	\end{lemma}

	\begin{proof}
		Let $\eta \in \dom(\Phi)$ and $\gamma \in \Gamma$.
		By definition \ref{DEF_Names_For_Trees} $\dot{\T}^{\Phi, \eta}_{1}$ is the canonical $\CC^\Phi$-name for the set		
		$$
			\bigcup_{j \in \Phi(\theta)}\dot{\T}^{\Phi, \eta}_{1, j},
		$$
		so it suffices to check that for all $j \in \Phi(\eta)$ we have $\pi^{\Phi, \theta}_{1,i}(\gamma)(\dot{\T}^{\Phi, \eta}_{1,j}) = \dot{\T}^{\Phi, \eta}_{1,j}$, so fix some $j \in \Phi(\eta)$.
		By Definition~\ref{DEF_Names_For_Trees} $\dot{\T}^{\Phi, \eta}_{1,j}$ is the canonical $\CC^\Phi$-name for the set
		$$
			\set{\pi^{\Phi, \eta}_j(\delta)(\dot{T}^{\Phi, \eta}_j)}{\delta \in \Gamma}.
		$$
		Thus, in case that $(\theta,i) = (\eta,j)$ we compute
		\begin{align*}
			\pi^{\Phi, \theta}_{1,i}(\gamma)(\dot{\T}^{\Phi, \theta}_{1,i})
			&= \pi^{\Phi, \theta}_{1,i}(\gamma)(\set{\pi^{\Phi, \theta}_i(\delta)(\dot{T}^{\Phi, \theta}_i)}{\delta \in \Gamma}) && (\text{Definition \ref{DEF_Names_For_Trees}})\\
			&= \set{\pi^{\Phi, \theta}_{1,i}(\gamma)(\pi^{\Phi, \theta}_i(\delta)(\dot{T}^{\Phi, \theta}_i))}{\delta \in \Gamma} && (\text{canonical name})\\
			&= \set{\pi^{\Phi, \theta}_{1,i}(\gamma + \delta)(\dot{T}^{\Phi, \theta}_i)}{\delta \in \Gamma} && (\pi^{\Phi, \theta}_{1,i}\text{ is gr.hom.})\\
			&= \set{\pi^{\Phi, \theta}_{1,i}(\delta)(\dot{T}^{\Phi, \theta}_i)}{\delta \in \Gamma} && (\Gamma \text{ is a group})\\
			&= \dot{\T}^{\Phi, \theta}_{1,j} && (\text{Definition \ref{DEF_Names_For_Trees}}).
		\end{align*}
		Otherwise, $\pi^{\Phi, \eta}_j(\delta)(\dot{T}^{\Phi, \eta}_j)$ has no information in the $(\theta, i)$-th coordinate for every $\delta \in \Gamma$, so that
		\begin{align*}
			\pi^{\Phi, \theta}_{1,i}(\gamma)(\dot{\T}^{\Phi, \eta}_{1,j})
			&= \pi^{\Phi, \theta}_{1,i}(\gamma)(\set{\pi^{\Phi, \eta}_j(\delta)(\dot{T}^{\Phi, \eta}_j)}{\delta \in \Gamma}) && (\text{Definition \ref{DEF_Names_For_Trees}})\\
			&= \set{\pi^{\Phi, \theta}_{1,i}(\gamma)(\pi^{\Phi, \eta}_j(\delta)(\dot{T}^{\Phi, \eta}_j))}{\delta \in \Gamma} && (\text{canonical name})\\
			&= \set{\pi^{\Phi, \eta}_{1,j}(\delta)(\dot{T}^{\Phi, \eta}_j)}{\delta \in \Gamma} && ((\theta, i) \neq (\eta, j))\\
			&= \dot{\T}^{\Phi, \eta}_{1,j}&& (\text{Definition \ref{DEF_Names_For_Trees}}).\qedhere
		\end{align*}
	\end{proof}

	\begin{corollary}\label{COR_Induced_Sequence_Of_Actions}
		Let $\Phi$ be a $\Theta$-indexing function, $\theta \in \dom(\Phi)$, $i \in \Phi(\theta)$.
		Then, there is an increasing sequence of $\Gamma$-actions 
		$$
			\seq{\pi^{\Phi, \theta}_{\alpha,i}:\Gamma \to \Aut(\PP^\Phi_\alpha)}{0 < \alpha \leq \aleph_1}
		$$
		such that $\pi^{\Phi, \theta}_{\epsilon + 1, i}$ the canonical extension of $\seq{\pi^{\Phi, \theta}_{\alpha,i}:\Gamma \to \Aut(\PP^\Phi_\alpha)}{0 < \alpha \leq \epsilon}$ for every $\epsilon < \aleph_1$.
		We call this sequence the induced sequence of group actions of $\pi^{\Phi, \theta}_{1,i}$ and will reserve the notions $\seq{\pi^{\Phi, \theta}_{\alpha,i}}{0 < \alpha \leq \aleph_1}$ for it.
	\end{corollary}

	\begin{proof}
		We iteratively construct the desired sequence.
		By Lemma~\ref{LEM_Initial_Automorphism_Preservation} we may start with $\pi^{\Phi, \theta}_{1,i}$ as in Definition~\ref{DEF_Action_Induced}, use Lemma~\ref{LEM_Canonical_Extension} for the successor step and Lemma~\ref{LEM_Limits_Of_Increasing_Sequences_Are_Unique} for the limit step.
	\end{proof}

	\section{A nice dense subset} \label{SEC_Dense_Subset}
	
	In the following sections we will need to work with a nice dense subset $D^\Phi_\alpha$ of $\PP^\Phi_\alpha$.
	A condition $p \in \PP^{\Phi}_\alpha$ has finite support, where $p(0) \in \CC^\Phi$ and for $\alpha \in \supp(p) \setminus \simpleset{0}$ we have
	$$
		p \restr \alpha \forces p(\alpha) \in \dot{\QQ}^\Phi_\alpha = \prod_{\theta \in \dom(\Phi)}\TT(\dot{\T}^{\Phi, \theta}_\alpha).
	$$
	We will define $D^\Phi_\alpha$, so that as many parameters for $p(\alpha)$ as possible are decided as ground model objects.
	First, we will need the following definition of a nice name for a real.
	
	\begin{definition} \label{DEF_Nice_Name_Real}
		Let $\PP$ be a forcing and $p \in \PP$.
		A nice $\PP$-name for a real below $p$ is a sequence $\seq{(\A_n, f_n)}{n < \omega}$ such that
		\begin{enumerate}[$\bullet$]
			\item for all $n < \omega$ the set $\A_n$ is a maximal antichain below $p$ and $f_n:\A_n \to 2^{>n}$,
			\item for all $n < m$ the antichain $\A_m$ refines $\A_n$, i.e.\ every $b \in \A_m$ there is $a \in \A_n$ with $b \extends a$,
			\item for all $n < m$, $a \in \A_n$ and $b \in \A_m$ with $b \extends a$ we have $f_n(a) \trianglelefteq f_m(b)$.
		\end{enumerate}
		Further, we write $\name(\seq{(\A_n, f_n)}{n < \omega})$ for the canonical $\PP$-name of $\seq{(\A_n, f_n)}{n < \omega}$, i.e.
		$$
		\name(\seq{(\A_n, f_n)}{n < \omega}) := \set{(a,(n,f_{n}(a)(n)))}{n < \omega \text{ and } a \in \A_n}.
		$$
	\end{definition}

	\begin{remark} \label{REM_Nice_Name_Real}
		Remember, that for every $p \in \PP$ and $\PP$-name $\dot{g}$ for a real below $p$ we may inductively define a nice $\PP$-name $\seq{(\A_n, f_n)}{n < \omega}$ for a real below $p$ such that
		$$
			p \forces \dot{f} = \name(\seq{(\A_n, f_n)}{n < \omega}).
		$$
		Further, if $\PP$ is c.c.c., then for any $p \in \PP$ there are at most $\left| \PP \right|^{\aleph_0}$ many nice names for reals below~$p$.
		We also have that nice names and their canonical names behave nicely under automorphisms in the following sense:
	\end{remark}

	\begin{remark}\label{REM_Automorphism_Of_Nice_Names}
		If $\seq{(\A_n,f_n)}{n < \omega}$ is a nice $\PP$-name for a real below $p$ and $\pi \in \Aut(\PP)$, then $\pi(\seq{(\A_n,f_n)}{n < \omega}) := \seq{(\B_n, g_n)}{n < \omega}$, where $\B_n = \pi[\A_n]$ and
		$$
			g_n(\pi(a)) := f_n(a),
		$$
		is a nice $\PP$-name for a real below $\pi(p)$ with
		$$
			\pi(\name(\seq{(\A_n,f_n)}{n < \omega})) = \name(\seq{(\B_n,g_n)}{n < \omega}).
		$$
	\end{remark}
	
	\begin{definition} \label{DEF_Nice_Condition}
		Let $\Phi$ be a $\Theta$-indexing function and $0 < \alpha \leq \aleph_1$.
		$D^\Phi_\alpha$ is the set of all nice conditions in $\PP^\Phi_\alpha$, where inductively $p \in \PP^\Phi_\alpha$ is a nice condition
		\begin{enumerate}[$\bullet$]
			\item for $\alpha = 1$: iff $p(0) = c^p$ for some $c^p \in \CC^\Phi$,
			\item for $\alpha + 1 > 1$: iff $p \restr \alpha \in D^\Phi_\alpha$ and 
			\begin{enumerate}[$\circ$]
				\item there is a finite set $\Theta_{\alpha}^p \subseteq \dom(\Phi)$,
				\item for every $\theta \in \Theta_{\alpha}^p$ there is a finite set $I^p_{\alpha, \theta} \subseteq \omega$,
				\item for every $i \in I_{\alpha, \theta}^p$ there is $n^p_{\alpha, \theta, i} < \omega$ and an $n^p_{\alpha, \theta, i}$-tree $s^p_{\alpha, \theta, i}$ and a finite set $F^p_{\alpha, \theta, i}$ of $D^{\Phi}_{\alpha}$-names, where every $\dot{f} \in F^p_{\alpha, \theta, i}$ is the canonical $D^\Phi_\alpha$-name of some nice $D^\Phi_\alpha$-name for a real below some $q \in D^{\Phi}_\alpha$ with $p \restr \alpha \extends q$,
				\item such that $p(\alpha)$ is the canonical name for the condition in $\dot{\QQ}^\Phi_\alpha = \prod_{\theta \in \dom(\Phi)}\TT(\dot{\T}^{\Phi, \theta}_\alpha)$
				with $\supp(p(\alpha)) = \Theta^p_{\alpha}$ and for every $\theta \in \Theta^p_{\alpha}$
				with $\supp(p(\alpha)(\theta)) = I^p_{\alpha, \theta}$ and for every $i \in I^p_{\alpha, \theta}$ we have $p(\alpha)(\theta)(i) = (s^p_{\alpha, \theta, i}, F^p_{\alpha, \theta, i})$,
			\end{enumerate}
			\item for limit $\alpha$: iff $p \restr \beta \in D^\Phi_\beta$ for all $\beta < \alpha$.
		\end{enumerate}
	\end{definition}
	
	\begin{remark} \label{REM_Construct_Condition_From_Parameters}
		Note that for any $p \in D^\Phi_\alpha$ the parameter $c^p$ and for every $\beta < \alpha$ the parameters $\Theta^p_\beta$, $I^p_{\beta, \theta}$, $n^p_{\beta, \theta, i}$, $s^p_{\beta, \theta, i}$ and $F^p_{\beta, \theta, i}$ are uniquely determined by $p$.
		Conversely, we may reconstruct $p$ from these parameters.
		Further, by definition of $\dot{\QQ}^\Phi_\alpha$ for every $\dot{f} \in F^p_{\alpha,\theta,i}$ as above we have that 
		$$
			p \restr \alpha \forces \dot{f} \restr n^p_{\alpha, \theta, i} \in s^p_{\alpha, \theta, i} \text{ and } \dot{f} \notin \bigcup_{T \in \dot{\T}^{\Phi, \theta}_\alpha} [T].
		$$
		Conversely, if $\dot{g}$ is the canonical $D^\Phi_\alpha$-name of some nice $D^\Phi_\alpha$-name for a real below some $q \in D^\Phi_\alpha$ with $p \restr \alpha \extends q$ and for some $\eta \in \Theta^p_\alpha$ and $j \in I^p_{\alpha, \eta}$ we have
		$$
			p \restr \alpha \forces \dot{f} \restr n^p_{\alpha, \eta, j} \in s^p_{\alpha, \eta, j} \text{ and } \dot{f} \notin \bigcup_{T \in \dot{\T}^{\Phi, \eta}_\alpha} [T],
		$$
		then we may extend $p \in D^\Phi_\alpha$ to a condition $r \in D^{\Phi}_\alpha$ by stipulating $r \restr \alpha := p \restr \alpha$ and
		\begin{enumerate}[$\bullet$]
			\item $\Theta^r_\alpha := \Theta^p_\alpha$,
			\item $I^r_{\alpha, \theta} := I^p_{\alpha, \theta}$ for every $\theta \in \Theta^r_{\alpha}$,
			\item $n^r_{\alpha, \theta, i} := n^p_{\alpha, \theta, i}$, $s^r_{\alpha, \theta, i} := s^p_{\alpha, \theta, i}$ and
			$$
				F^r_{\alpha, \theta, i} :=
				\begin{cases}
					F^p_{\alpha, \theta, i} \cup \simpleset{\dot{g}} & \text{if } (\theta, i) = (\eta, j),\\
					F^p_{\alpha, \theta, i} & \text{otherwise}.
				\end{cases}
			$$
			for every $\theta \in \Theta^r_{\alpha}$ and $i \in I^r_{\alpha, \theta}$.
		\end{enumerate}
	\end{remark}
	
	\begin{remark}\label{REM_Ds_Are_Increasing}
		 For $0 < \beta \leq \alpha \leq \aleph_1$ we have $\iota^\Phi_{\beta,\alpha}(D_{\beta}^{\Phi}) \subseteq D_{\alpha}^\Phi$ and for limit $\alpha \leq \aleph_1$ we have
		$$
			D_{\alpha}^\Phi = \bigcup_{\beta < \alpha} \iota^\Phi_{\beta,\alpha}(D_{\beta}^\Phi).
		$$
	\end{remark}
	
	\begin{lemma} \label{LEM_Ds_Are_Dense}
		Let $\Phi$ be a $\Theta$-indexing function and $0 < \alpha \leq \aleph_1$.
		Then, $D^\Phi_\alpha$ is dense in $\PP^\Phi_\alpha$.
	\end{lemma}

	\begin{proof}
		By induction.
		Case $\alpha = 1$ follows from $\PP^\Phi_1 \cong \CC^\Phi$.
		For limit $\alpha$ let $p \in \PP^\Phi_\alpha$.
		Choose $\beta < \alpha$ and such that $\iota^\Phi_{\beta, \alpha}(p \restr \beta) = p$.
		By induction choose $q \in D^\Phi_\beta$ with $q \extends p \restr \beta$.
		By Remark~\ref{REM_Ds_Are_Increasing} we have $\iota^\Phi_{\alpha,\beta}(q) \in D^\Phi_\alpha$ and $\iota^\Phi_{\beta, \alpha}(q) \leq \iota^\Phi_{\beta, \alpha}(p \restr \beta) = p$.
		
		Finally, for $\alpha + 1$ let $p \in \PP^\Phi_{\alpha + 1}$.
		Then
		$$
			p \restr \alpha \forces p(\alpha) \in \dot{\QQ}^\Phi_\alpha = \prod_{\theta \in \dom(\Phi)}\TT(\dot{\T}^{\Phi, \theta}_\alpha)
		$$
		and by induction $D^\Phi_\alpha$ is dense in $\PP^\Phi_\alpha$ we may choose $q \in D^\Phi_\alpha$ which decides all necessary parameters of an element in $\prod_{\theta \in \dom(\Phi)}\TT(\dot{\T}^{\Phi, \theta}_\alpha)$.
		By Remark~\ref{REM_Nice_Name_Real} there a nice $D^\Phi_\alpha$-names for all $D^\Phi_\alpha$-names for reals below $q$ which occur in some $F^p_{\alpha, \theta, i}$.
		Then, the canonical name $\dot{q}_\alpha$ for $p(\alpha)$ as defined in Definition~\ref{DEF_Nice_Condition} satisfies
		$$
			q \forces p(\alpha) = \dot{q}_\alpha.
		$$
		Hence, $q \concat \dot{q}_\alpha \in D^\Phi_{\alpha + 1}$ and $q \concat \dot{q}_\alpha \extends p$.
	\end{proof}

	\begin{definition} \label{DEF_Hereditary_Support}
		Let $\Phi$ be a $\Theta$-indexing function, $0 < \alpha \leq \aleph_1$ and $p \in D^\Phi_\alpha$.
		We will inductively define countable subsets $\hsupp_\Theta(p) \subseteq \dom(\Phi)$ and $\hsupp(p) \subseteq \bigcup_{\theta \in \hsupp_\Theta(p)} (\simpleset{\theta} \times \Phi(\theta))$ called the hereditary support of $p$.
		
		Given this definition we define for the canonical $D^\Phi_\alpha$-name $\dot{f}$ of a nice $D^\Phi_\alpha$ name $\seq{(\A_n, f_n)}{n < \omega}$ for a real below $p$ the countable sets
		\begin{align*}
			\hsupp_\Theta(\dot{f}) &:= \bigcup_{n < \omega, a \in \A_n} \hsupp_\Theta(a)\\
			\hsupp(\dot{f}) &:= \bigcup_{n < \omega, a \in \A_n} \hsupp(a)
		\end{align*}
		For $\alpha = 1$, we define $\hsupp(p) := \supp(p(0))$ and let $\hsupp_\Theta(p)$ be the projection of $\hsupp(p)$ onto the first component.
		Next, for limit $\alpha$ we may choose $\beta < \alpha$ with $\iota^\Phi_{\beta, \alpha}(p \restr \beta) = p$ and define $\hsupp_\Theta(p) := \hsupp_\Theta(p \restr \beta)$ and $\hsupp(p) := \hsupp(p \restr \beta)$.
		Finally, for $\alpha + 1 > 1$ we define
		\begin{align*}
			\hsupp_\Theta(p) &:= \hsupp_\Theta(p \restr \alpha) \cup \Theta^p_\alpha \cup \bigcup \set{\hsupp_\Theta(\dot{f})}{\theta \in \Theta^p_\alpha, i \in I^p_{\alpha, \theta} \text{ and } \dot{f} \in F^p_{\alpha, \theta, i}},\\
			\hsupp(p) &:= \hsupp(p \restr \alpha) \cup \bigcup \set{\hsupp(\dot{f})}{\theta \in \Theta^p_\alpha, i \in I^p_{\alpha, \theta} \text{ and } \dot{f} \in F^p_{\alpha, \theta, i}}.
		\end{align*}
	\end{definition}
	
	\begin{lemma} \label{LEM_Counting_Hereditary}
		Assume $\sf{CH}$ and let $\Phi$ be a $\Theta$-indexing function, $0 < \alpha \leq \aleph_1$ and assume that both $\Theta_0 \subseteq \Theta$ and $I_0 \subseteq \bigcup_{\theta \in \Theta_0} (\simpleset{\theta} \times \Phi(\theta))$ are countable.
		Then, there are at most $\aleph_1$-many $p \in D^\Phi_\alpha$ with $\hsupp_\Theta(p) \subseteq \Theta_0$ and $\hsupp(p) \subseteq I_0$.
		Thus, for any $p \in D^\Phi_\alpha$ there are at most $\aleph_1$-many canonical $D^\Phi_\alpha$-names $\dot{f}$ of nice $D^\Phi_\alpha$-names for reals below $p$ with $\hsupp_\Theta(\dot{f}) \subseteq \Theta_0$ and $\hsupp(\dot{f}) \subseteq I_0$.
	\end{lemma}
	
	\begin{proof}
		In order to see the second part of the statement, let $\dot{f}$ be the canonical $D^\Phi_\alpha$-name of a nice $D_\alpha^\Phi$-name $\seq{(\A_n, f_n)}{n < \omega}$ for a real below $p \in D^\Phi_\alpha$ with $\hsupp_\Theta(\dot{f}) \subseteq \Theta_0$ and $\hsupp(\dot{f}) \subseteq I_0$.
		Then, for any $n < \omega$ and $a \in \A_n$ we also have $\hsupp(a) \subseteq I_0$ and $\hsupp_\Theta(a) \subseteq \Theta_0$.
		But by the first part of the statement
		$$
			\left| \set{p \in D^\Phi_\alpha}{\hsupp(p) \subseteq I_0 \text{ and } \hsupp_\Theta(p) \subseteq \Theta_0} \right| \leq \aleph_1,
		$$
		so that Remark~\ref{REM_Nice_Name_Real} using $\sf{CH}$ and the fact that $\PP^{\Phi}_\alpha$ is c.c.c., we may compute the number of nice $D^\Phi_\alpha$-names for reals below $p$ as at most
		$$
			\aleph_1^{\aleph_0} = (\aleph_0^{\aleph_0})^{\aleph_0} = \aleph_0^{\aleph_0 \cdot \aleph_0} = \aleph_0^{\aleph_0} = \aleph_1.
		$$
		We prove the first part of the statement by induction.
		For $\alpha = 1$, as $\left|\CC\right| = \aleph_0$ and $I_0$ is countable there are at most $\aleph_0^{\aleph_0} = \aleph_1$-many conditions in $\PP^\Phi_1 \cong \CC^\Phi$ with $\hsupp(p) \subseteq I_0$.
		For limit $\alpha$, note that by Remark~\ref{REM_Ds_Are_Increasing} we have
		\begin{align*}
			&\set{p \in D^\Phi_\alpha}{\hsupp(p) \subseteq I_0 \text{ and } \hsupp_\Theta(p) \subseteq \Theta_0}\\
			= &\bigcup_{\beta < \alpha} \set{\iota^\Phi_{\beta, \alpha}(p)}{p \in D^\Phi_\beta, \hsupp(p) \subseteq I_0 \text{ and } \hsupp_\Theta(p) \subseteq \Theta_0}
		\end{align*}
		Thus, by induction we compute
		$$
			\left| \set{p \in D^\Phi_\alpha}{\hsupp(p) \subseteq I_0 \text{ and } \hsupp_\Theta(p) \subseteq \Theta_0} \right| \leq \left|\alpha\right| \cdot \aleph_1 = \aleph_1.
		$$
		Finally, for $\alpha + 1 > 1$ and $p \in D^\Phi_{\alpha + 1}$ we have $p \restr \alpha \in D^\Phi_\alpha$ and $\hsupp_\Theta(p \restr \alpha) \subseteq \hsupp_\Theta(p) \subseteq \Theta_0$ and $\hsupp(p \restr \alpha) \subseteq \hsupp(p) \subseteq I_0$, so by induction there are at most $\aleph_1$-many choices for $p \restr \alpha$.
		Also, $\Theta^p_\alpha \subseteq \hsupp_\Theta(p) \subseteq \Theta_0$, so there are at most countably many choices for $\Theta^p_\alpha$.
		Further, for any of the finitely many $\theta \in \Theta^p_\alpha$ there are at most countably many choices $I^p_{\alpha, \theta}$ and for any of the finitely many $i \in I^p_{\alpha, \theta}$ there are at most countably many choices for $n^p_{\alpha, \theta, i}$ and $s^p_{\alpha, \theta, i}$.
		Finally, for any $\dot{f} \in F^p_{\alpha, \theta, i}$ choose $q \in D^{\Phi}_\alpha$ such that $\dot{f}$ is the canonical $D^{\Phi}_\alpha$-name of some nice $D^{\Phi}_\alpha$-name for a real below $q$ with $p \restr \alpha \extends q$.
		Then, we have $\hsupp_\Theta(q) \subseteq \hsupp_\Theta(p \restr \alpha) \subseteq \hsupp_\Theta(p) \subseteq \Theta_0$ and $\hsupp(q) \subseteq \hsupp(p \restr \alpha) \subseteq \hsupp(p) \subseteq I_0$, so by induction assumption there at most $\aleph_1$-many choices for $q$.
		Analogously, $\hsupp_\Theta(\dot{f}) \subseteq \hsupp_\Theta(p) \subseteq \Theta_0$ and $\hsupp(\dot{f}) \subseteq \hsupp(p) \subseteq I_0$, so by induction assumption there are at most $\aleph_1$-many choices for $\dot{f}$.
		Hence, there are at most $\aleph_1$-many choices for $F^p_{\alpha, \theta, i}$.
		By Remark~\ref{REM_Construct_Condition_From_Parameters} $p(\alpha)$ is uniquely determined by these parameters, so that there are at most $\aleph_1$-many choices for $p$.
	\end{proof}
	
	Next, we prove that the action of $\Gamma$ on $\PP^\Phi_\alpha$ restricts to actions on our nice dense set $D^\Phi_\alpha$.
	
	\begin{lemma}\label{LEM_Ds_Fixed_Under_Action}
		Let $\Phi$ be an $\Theta$-indexing function, $\theta \in \dom(\Phi)$, $i \in \Phi(\theta)$, $\gamma \in \Gamma$ and $0 < \alpha \leq \aleph_1$.
		Then, $\pi^{\Phi, \theta}_{\alpha, i}(\gamma)(D^\Phi_\alpha) = D^\Phi_\alpha$.
	\end{lemma}

	\begin{proof}
		It suffices to verify that $\pi^{\Phi, \theta}_{\alpha, i}(\gamma)(D^\Phi_\alpha) \subseteq D^\Phi_\alpha$, which we prove by induction.
		For $\alpha = 1$, let $p \in D^\Phi_1$.
		Then, we compute
		$$
			\pi^{\Phi, \theta}_{1, i}(\gamma)(p)(0) = \pi^{\Phi, \theta}_{1,i}(\gamma)(p(0)) = \pi^{\Phi, \theta}_{1,i}(\gamma)(c^p) \in \CC^\Phi,
		$$
		so that $\pi^{\Phi, \theta}_{1, i}(\gamma)(p) \in D^\Phi_1$.
		For limit $\alpha$, let $p \in D^\Phi_\alpha$ and choose $\beta < \alpha$ such that $\iota^\Phi_{\beta, \alpha}(p \restr \beta) = p$.
		By induction assumption $\pi^{\Phi, \theta}_{\beta, i}(\gamma)(p \restr \beta) \in D^\Phi_\beta$.
		By Remark~\ref{REM_Ds_Are_Increasing} we have $\iota^\Phi_{\beta, \alpha}(\pi^{\Phi, \theta}_{\beta, i}(\gamma)(p \restr \beta)) \in D^\Phi_\alpha$.
		Hence, by Definition~\ref{DEF_Increasing_Sequence_Of_Actions} we compute
		$$
			\pi^{\Phi, \theta}_{\alpha,i}(\gamma)(p)
			= \pi^{\Phi, \theta}_{\alpha,i}(\gamma)(\iota^\Phi_{\beta, \alpha}(p \restr \beta))
			= \iota^\Phi_{\beta, \alpha}(\pi^{\Phi, \theta}_{\beta,i}(\gamma)(p \restr \beta)) \in D^\Phi_\alpha.
		$$
		Finally, for $\alpha + 1 > 1$ let $p \in D^\Phi_{\alpha + 1}$.
		Then, $p \restr \alpha \in D^\Phi_\alpha$ and by Definition~\ref{DEF_Canonical_Extension}
		$$
			\pi^{\Phi, \theta}_{\alpha + 1, i}(\gamma)(p) = \pi^{\Phi, \theta}_{\alpha, i}(\gamma)(p \restr \alpha) \concat \pi^{\Phi, \theta}_{\alpha, i}(\gamma)(p(\alpha)).
		$$
		By induction assumption we obtain $\pi^{\Phi, \theta}_{\alpha, i}(\gamma)(p \restr \alpha) \in D^\Phi_\alpha$.
		By Remark~\ref{REM_Automorphism_Of_Nice_Names} $\pi^{\Phi, \theta}_{\alpha, i}(\gamma)(p(\alpha))$ is the canonical name for the condition in $\dot{\QQ}^\Phi_\alpha = \prod_{\theta \in \dom(\Phi)}\TT(\dot{\T}^{\Phi, \theta}_\alpha)$
		with $\supp(\pi^{\Phi, \theta}_{\alpha, i}(\gamma)(p(\alpha))) = \Theta^p_{\alpha}$ and for every $\theta \in \Theta^p_{\alpha}$
		with $\supp(\pi^{\Phi, \theta}_{\alpha, i}(\gamma)(p(\alpha))(\theta)) = I^p_{\alpha, \theta}$ and for every $i \in I^p_{\alpha, \theta}$ we have $\pi^{\Phi, \theta}_{\alpha, i}(\gamma)(p(\alpha))(\theta)(i) = (s^p_{\alpha, \theta, i}, \pi^{\Phi, \theta}_{\alpha, i}(\gamma)(F^p_{\alpha, \theta, i}))$.
		Hence, $\pi^{\Phi, \theta}_{\alpha + 1, i}(\gamma)(p) \in D^\Phi_{\alpha + 1}$.
	\end{proof}
	
	\begin{lemma} \label{LEM_Condition_Below_Stable_For_Automorphism}
		Let $\Phi$ be an $\Theta$-indexing function, $\theta \in \dom(\Phi)$, $i \in \Phi(\theta)$, $\gamma \in \Gamma$ and $p \in D_{\alpha}^{\Phi}$ for some $0 < \alpha \leq \aleph_1$ such that $\pi^{\Phi, \theta}_{1, i}(\gamma)(p \restr 1) = p \restr 1$.
		Then, there is $q \leq p$ in $D_{\alpha}^\Phi$ with $q(0) = p(0)$ and $\pi^{\Phi, \theta}_{\alpha, i}(\gamma)(q) = q$.
	\end{lemma}

	\begin{proof}
		By induction.
		The case $\alpha = 1$ is exactly the assumption given on $\gamma$.
		For limit $\alpha$ choose $\beta < \alpha$ with $\iota^\Phi_{\beta, \alpha}(p \restr \beta) = p$.
		By induction assumption choose $q \leq p \restr \beta$ in $D^\Phi_\beta$ such that $q(0) = p(0)$  and $\pi^{\Phi, \theta}_{\beta, i}(\gamma)(q) = q$.
		Then, we have $\iota^\Phi_{\beta, \alpha}(q)(0) = q(0) = p(0)$, by Remark~\ref{REM_Ds_Are_Increasing} $\iota^\Phi_{\beta, \alpha} \in D^\Phi_\alpha$ and by Definition~\ref{DEF_Increasing_Sequence_Of_Actions} we compute
		$$
			\pi^{\Phi, \theta}_{\alpha, i}(\gamma)(\iota^\Phi_{\beta, \alpha}(q)) = \iota^\Phi_{\beta, \alpha}(\pi^{\Phi, \theta}_{\beta, i}(\gamma)(q)) = \iota^\Phi_{\beta, \alpha}(q).
		$$
		Finally, for $\alpha + 1 > 1$ let $p \in D^\Phi_{\alpha + 1}$.
		Then, by induction assumption we may choose $q \leq p \restr \alpha$ in $D^\Phi_\alpha$ with $q(0) = p(0)$ and $\pi^{\Phi, \theta}_{\alpha, i}(\gamma)(q) = q$.
		We define
		\begin{enumerate}[$\bullet$]
			\item $\Theta^q_\alpha := \Theta^p_\alpha$,
			\item $I^q_{\alpha, \theta} := I^p_{\alpha, \theta}$ for every $\theta \in \Theta^q_{\alpha}$,
			\item $n^q_{\alpha, \theta, i} := n^p_{\alpha, \theta, i}$ and $s^q_{\alpha, \theta, i} := s^p_{\alpha, \theta, i}$ for every $\theta \in \Theta^q_\alpha$ and $i \in I^q_{\alpha, \theta}$,
			\item $F^q_{\alpha, \theta, i} := F^p_{\alpha, \theta, i} \cup \pi^{\Phi, \theta}_{\alpha, i}(\gamma)(F^p_{\alpha, \theta, i})$.
		\end{enumerate}
		Let $\dot{q}_\alpha$ be the canonical name for the condition in $\dot{\QQ}^\Phi_\alpha = \prod_{\theta \in \dom(\Phi)}\TT(\dot{\T}^{\Phi, \theta}_\alpha)$
		with $\supp(\dot{q}_\alpha) = \Theta^q_{\alpha}$, for every $\theta \in \Theta^q_{\alpha}$
		with $\supp(\dot{q}_\alpha(\theta)) = I^q_{\alpha, \theta}$ and for every $i \in I^q_{\alpha, \theta}$ we have $\dot{q}_\alpha(\theta)(i) = (s^q_{\alpha, \theta, i}, F^q_{\alpha, \theta, i})$.
		We claim that $q \concat \dot{q}_\alpha$ is as desired.
		To obtain $q \concat \dot{q}_\alpha \in D^\Phi_{\alpha + 1}$ by Remark~\ref{REM_Construct_Condition_From_Parameters} it suffices to verify that for every $\theta \in \Theta^p_\alpha$, $i \in I^p_{\alpha, \theta}$ and $\dot{f} \in F^p_{\alpha, \theta, i}$ we have
		$$
			q \forces \pi^{\Phi, \theta}_{\alpha, i}(\gamma)(\dot{f}) \restr n^p_{\alpha, \theta, i} \in s^p_{\alpha, \theta, i} \text{ and } \pi^{\Phi, \theta}_{\alpha, i}(\gamma)(\dot{f}) \notin \bigcup_{T \in \dot{\T}^{\Phi, \theta}_\alpha} [T].
		$$
		To this end, notice that $p \in D^\Phi_{\alpha + 1}$ implies
		$$
			p \restr \alpha \forces \dot{f} \restr n^p_{\alpha, \theta, i} \in s^p_{\alpha, \theta, i} \text{ and } \dot{f} \notin \bigcup_{T \in \dot{\T}^{\Phi, \theta}_\alpha} [T],
		$$
		so also $q \leq p \restr \alpha$ forces this.
		Further, $\pi^{\Phi, \theta}_{\alpha, i}(\gamma)(q) = q$ and $\pi^{\Phi, \theta}_{\alpha, i}(\gamma)(\dot{\T}^{\Phi, \theta}_{\alpha}) = \dot{\T}^{\Phi, \theta}_{\alpha}$, so applying the automorphism theorem to the previous statement yields the desired conclusion.
		Next, we have $\pi^{\Phi, \theta}_{\alpha, i}(\gamma)(F^q_{\alpha, \theta, i}) = F^q_{\alpha, \theta, i}$ since $\pi^{\Phi, \theta}_{\alpha, i}(\gamma)$ is an involution.
		This implies
		$$
			\pi^{\Phi, \theta}_{\alpha + 1, i}(\gamma)(q \concat \dot{q}_\alpha) = \pi^{\Phi, \theta}_{\alpha, i}(\gamma)(q) \concat \pi^{\Phi, \theta}_{\alpha, i}(\gamma)(\dot{q}_\alpha) = q \concat \dot{q}_\alpha.
		$$
		Finally, by definition we have $q \concat \dot{q}_\alpha \extends p$ and $(q \concat \dot{q}_\alpha)(0) = q(0) = p(0)$.
	\end{proof}

	\section{Complete embeddings} \label{SEC_Complete_Embeddings}

	In this section we combine the results of the previous sections in order to prove that our forcing in Definition~\ref{DEF_Forcing} has enough complete subforcings to carry out our isomorphism-of-names argument for Main Theorem~\ref{THM_AT_Spectrum}.
	The whole section will be devoted towards the proof of the following Theorem~\ref{THM_Complete_Subforcings} as it is an elaborate inductive construction of complete embeddings.

	\begin{theorem}\label{THM_Complete_Subforcings}
		Let $\Phi \subseteq \Psi$ be a $\Theta$-subindexing function and assume $\Phi$ is countable.
		Then, $\PP_{\alpha}^\Phi \completesubposet \PP_{\alpha}^\Psi$ for all $\alpha \leq \aleph_1$.
	\end{theorem}

	By induction over $\alpha \leq \aleph_1$ we define embeddings $\iota^{\Phi,\Psi}_\alpha:\PP_{\alpha}^\Phi \to \PP_{\alpha}^\Psi$ and prove that they admit reductions from $\PP^\Psi_{\alpha}$ to $\PP^\Phi_\alpha$.
	Thus, $\iota^{\Phi, \Psi}_\alpha$ will be a complete embedding.
	Additionally, we will verify the following properties along our iteration:
		
	\begin{enumerate}[(A)]
		\item For all $\beta \leq \alpha$ the following diagram commutes:
		\begin{center}
			\begin{tikzcd}
				\PP^\Phi_\beta \arrow[d, "{\iota^\Phi_{\beta, \alpha}}"] \arrow[r, "{\iota^{\Phi, \Psi}_{\beta}}"] & \PP^\Psi_\beta \arrow[d, "{\iota^\Psi_{\beta, \alpha}}"] \\
				\PP^\Phi_\alpha \arrow[r, "{\iota^{\Phi, \Psi}_{\alpha}}"]                                         & \PP^\Psi_\alpha                                         
			\end{tikzcd}
		\end{center}
		\item For all $\theta \in \dom(\Phi)$ and $i \in \Phi(\theta)$ the embedding $\iota_{\alpha}^{\Phi, \Psi}: \PP^\Phi_\alpha \to \PP^\Psi_\alpha$ is a morphism of $\Gamma$-sets, i.e.\ the following diagram commutes for every $\gamma \in \Gamma$:	
		\begin{center}
			\begin{tikzcd}
				\PP^\Phi_\alpha \arrow[d, "{\pi^{\Phi,\theta}_{\alpha,i}(\gamma)}"] \arrow[r, "{\iota^{\Phi,\Psi}_{\alpha}}"] & \PP^\Psi_{\alpha} \arrow[d, "{\pi^{\Psi,\theta}_{\alpha,i}(\gamma)}"] \\					\PP^\Phi_\alpha \arrow[r, "{\iota^{\Phi,\Psi}_{\alpha}}"]                                                     & \PP^\Psi_{\alpha}                                                    
			\end{tikzcd}
		\end{center}	
		\item For all $\theta \in \dom(\Phi)$ and $i \in \Phi(\theta)$ we have
		$$
			\iota_{1}^{\Phi, \Psi}(\dot{c}^{\Phi, \theta}_{i}) = \dot{c}^{\Psi, \theta}_{i} \text{ and thus } \iota_{1}^{\Phi, \Psi}(\dot{T}^{\Phi, \theta}_{i}) = \dot{T}^{\Psi, \theta}_{i}.
		$$
		\item For all $ \alpha = \beta + 1 > 1$, $\theta \in \dom(\Phi)$ and $n < \omega$ we have
		$$
			\iota^{\Phi, \Psi}_{\beta}(\dot{T}^{\Phi, \theta}_{\beta, n}) = \dot{T}^{\Psi, \theta}_{\beta, n}.
		$$
		\item If $\alpha > 0$, then for all $\theta \in \dom(\Phi)$, the name $\dot{\T}^{\Psi, \theta}_\alpha$ is the canonical $\PP^\Psi_\alpha$-name for 
		$$
			\iota^{\Phi, \Psi}_{\alpha}(\dot{\T}^{\Phi, \theta}_\alpha) \cup \bigcup_{i \in \Psi(\theta) \setminus \Phi(\theta)}\iota^{\Psi}_{1, \alpha}(\T^{\Psi, \theta}_{i}).
		$$
		\item
		For all $\theta \in \dom(\Phi)$, $i \in \Psi(\theta) \setminus \Phi(\theta)$, $\gamma \in \Gamma$ we have that $\pi^{\Psi, \theta}_{\alpha, i}(\gamma)$ acts trivially on $\iota_\alpha^{\Phi, \Psi}(\PP^\Phi_\alpha)$.
		\item
		For all $\theta \in \dom(\Phi)$, $i \in \Psi(\theta) \setminus \Phi(\theta)$, $\gamma \in \Gamma$ and  $\PP_{\alpha}^\Phi$-name $\dot{f}$ for a real
		$$
			\PP^\Psi_{\alpha} \forces  \iota^{\Phi, \Psi}_{\alpha}(\dot{f}) \neq \iota^\Psi_{1, \alpha}(\pi^{\Psi, \theta}_{1, i}(\gamma)(\dot{c}^{\Psi, \theta}_i)).
		$$
	\end{enumerate}
	First, we prove that $(\text{G})$ follows from $(\text{F})$, so that we only need to verify $(\text{A})$ to $(\text{F})$ inductively:
		
	\begin{proof}
		Let $p \in \PP_{\alpha}^\Psi$.
		By Lemma~\ref{LEM_Ds_Are_Dense}, we may assume $p \in D^\Psi_\alpha$.
		Choose $N\notin \dom(p(0)(\theta, i))$.
		Let $\delta \in \Gamma$ be defined by $\delta(N) = 1$ and $0$ otherwise.
		Then, $\pi^{\Psi,\theta}_{1,i}(\delta)(p \restr 1) = p \restr 1$, so by the Lemma~\ref{LEM_Condition_Below_Stable_For_Automorphism} we may choose $q \leq p$ in $D^\Psi_\alpha$ such that $q(0) = p(0)$ and $\pi^{\Psi, \theta}_{\alpha, i}(\delta)(q) = q$.
		Thus, $N \notin \dom(q(0)(\theta, i))$ and we may define $q_k \leq q$ which replaces $p(0)(\theta, i)$ by $p(0)(\theta, i) \cup \simpleseq{N, j}$ for $j \in 2$.
		Then, we have $\pi^{\Psi, \theta}_{\alpha,i}(\delta)(q_j) = q_{1-j}$ for $j \in 2$ and there is a $k \in 2$ with
		$$
			q_0 \forces \iota^\Psi_{1, \alpha}(\pi^{\Psi, \theta}_{1, i}(\gamma)(\dot{c}^{\Psi, \theta}_i))(N) = k.
		$$
		Further, using $(\delta + \gamma)(N) = 1 - \gamma(N)$ we compute
		\begin{align*}
			\pi^{\Psi, \theta}_{\alpha,i}(\delta)(\iota^\Psi_{1, \alpha}(\pi^{\Psi, \theta}_{1, i}(\gamma)(\dot{c}^{\Psi, \theta}_i)))(N)
			&= \iota^\Psi_{1, \alpha}(\pi^{\Psi, \theta}_{1,i}(\delta)(\pi^{\Psi, \theta}_{1, i}(\gamma)(\dot{c}^{\Psi, \theta}_i)))(N)\\
			&= \iota^\Psi_{1, \alpha}(\pi^{\Psi, \theta}_{1,i}(\delta + \gamma)(\dot{c}^{\Psi, \theta}_i))(N)\\
			&= 1 - \iota^\Psi_{1, \alpha}(\pi^{\Psi, \theta}_{1,i}(\gamma)(\dot{c}^{\Psi, \theta}_i))(N).
		\end{align*}
		Thus, by the automorphism theorem we obtain
		$$
			q_{1} \forces \iota^\Psi_{1, \alpha}(\pi^{\Psi, \theta}_{1, i}(\gamma)(\dot{c}^{\Psi, \theta}_i))(N) = 1- k.
		$$
		Choose $r_0 \leq q_0$ such that $r_0 \forces \iota^{\Phi, \Psi}_{\alpha}(\dot{f})(N) = l$ for some $l \in 2$.
		Since $\dot{f}$ is a $\PP_{\alpha}^\Phi$-name by $(\text{F})$ we have $\pi^{\Psi, \theta}_{\alpha, i}(\delta)( \iota^{\Phi, \Psi}_{\alpha}(\dot{f})) = \iota^{\Phi, \Psi}_{\alpha}(\dot{f})$.
		Thus, the automorphism theorem yields
		$$
			\pi^{\Psi, \theta}_{\alpha, i}(\delta)(r_0) \forces  \iota^{\Phi, \Psi}_{\alpha}(\dot{f})(N) = l.
		$$
		But then either $r_0 \leq q_0 \leq q \leq p$ and
		$$
			r_0 \forces \iota^\Psi_{1, \alpha}(\pi^{\Psi, \theta}_{1, i}(\gamma)(\dot{c}^{\Psi, \theta}_i))(N) = k \neq l =  \iota^{\Phi, \Psi}_{\alpha}(\dot{f})(N)
		$$
		or $\pi^{\Psi,\theta}_{\alpha,i}(\delta)(r_0) \leq \pi^{\Psi,\theta}_{\alpha,i}(\delta)(q_0) = q_1 \leq q \leq p$ and
		\[
			\pi^{\Psi,\theta}_{\alpha,i}(\delta)(r_0) \forces \iota^\Psi_{1, \alpha}(\pi^{\Psi, \theta}_{1, i}(\gamma)(\dot{c}^{\Psi, \theta}_i))(N) = 1 - k \neq l =  \iota^{\Phi, \Psi}_{\alpha}(\dot{f})(N).\hfill\qedhere
		\]
	\end{proof}
	Next, we inductively define $\iota^{\Phi, \Psi}_\alpha$ and verify properties (A) to (F), so consider $\alpha = 1$ first.
	In this case we already defined $\iota^{\Phi, \Psi}_i:\PP^\Phi_1 \to \PP^\Psi_1$ as the complete embedding corresponding to $\iota^{\Phi, \Psi}: \CC^\Phi \to \CC^\Psi$.
		
	\begin{enumerate}[(A)]
		\item There is nothing to show.
		\item Follows immediately from Remark~\ref{REM_Morphism_Of_Gamma_Sets_On_Cohen}.
		\item By definition of $\dot{c}_i^{\Phi, \theta}$, $\dot{c}_i^{\Psi, \theta}$ and $\iota_1^{\Phi, \Psi}$.
		\item There is nothing to show.
		\item Let $\theta \in \dom(\Phi)$.
		Then, we compute
		\begin{align*}
			\dot{\T}^{\Psi, \theta}_1
			&= \bigcup_{i \in \Psi(\theta)}\dot{\T}^{\Psi, \theta}_{i} && (\text{Remark~\ref{REM_Properties_Of_Automorphism_Subgroup}})\\
			&= \bigcup_{i \in \Phi(\theta)}\dot{\T}^{\Psi, \theta}_{i} \cup \bigcup_{i \in \Psi(\theta) \setminus \Phi(\theta)}\dot{\T}^{\Psi, \theta}_{i} \\
			&= \bigcup_{i \in \Phi(\theta)}\iota_1^{\Phi, \Psi}(\dot{\T}^{\Phi, \theta}_{i}) \cup \bigcup_{i \in \Psi(\theta) \setminus \Phi(\theta)}\dot{\T}^{\Psi, \theta}_{i} && (\text{Remark~\ref{REM_Properties_Of_Automorphism_Subgroup}})\\
			&= \iota_1^{\Phi, \Psi}(\bigcup_{i \in \Phi(\theta)}\dot{\T}^{\Phi, \theta}_{i}) \cup \bigcup_{i \in \Psi(\theta) \setminus \Phi(\theta)}\dot{\T}^{\Psi, \theta}_{i} && (\text{canonical name})\\
			&= \iota_1^{\Phi, \Psi}(\dot{\T}^{\Phi, \theta}_1) \cup \bigcup_{i \in \Psi(\theta) \setminus \Phi(\theta)}\dot{\T}^{\Psi, \theta}_{i} && (\text{Remark~\ref{REM_Properties_Of_Automorphism_Subgroup}}).
		\end{align*}
		\item Follows immediately from the fact that $\pi^{\Psi, \theta}_{\alpha,i}(\gamma)$ only acts on Cohen information outside of the indexing of $\Phi$.
	\end{enumerate}
	Next, we consider limit $\alpha$.
	Then, by $\text{(A)}$ for every $\beta' \leq \beta < \alpha$ the following diagram commutes:
	\begin{center}
		\begin{tikzcd}
			\PP^\Phi_{\beta'} \arrow[d, "{\iota^\Phi_{\beta', \beta}}"] \arrow[r, "{\iota^{\Phi, \Psi}_{\beta'}}"] & \PP^\Psi_{\beta'} \arrow[d, "{\iota^\Psi_{\beta', \beta}}"] \\
			\PP^\Phi_\beta \arrow[r, "{\iota^{\Phi, \Psi}_{\beta}}"]                                               & \PP^\Psi_\beta                                             
		\end{tikzcd}
	\end{center}
	
	By the universal property of the direct limit there is a unique map $\iota^{\Phi,\Psi}_\alpha:\PP^\Phi_\alpha \to \PP^\Psi_\alpha$ such that for every $\beta \leq \alpha$ the diagram in $(\text{A})$ commutes.
	Further, as a direct limit of complete embeddings, also $\iota^{\Phi, \Psi}_\alpha$ is a complete embedding.
	Note that (C) and (D) are vacuous at limits.
		
	\begin{enumerate}[(A)]
		\item Follows from the universal property of the direct limit.
		\item Let $\theta \in \dom(\Phi)$, $i \in \Phi(\theta)$, $\gamma \in \Gamma$ and $p \in \PP^\Phi_\alpha$.
		Choose $\beta < \alpha$ such that $\iota^\Phi_{\beta, \alpha}(p \restr \beta) = p$.
		Then, we compute
		\begin{align*}
			\pi^{\Psi, \theta}_{\alpha, i}(\gamma)(\iota^{\Phi,\Psi}_\alpha(p)) 
			&= \pi^{\Psi, \theta}_{\alpha, i}(\gamma)(\iota^{\Phi,\Psi}_\alpha(\iota^\Phi_{\beta, \alpha}(p \restr \beta))) && (\text{choice of } \beta)\\
			&= \pi^{\Psi, \theta}_{\alpha, i}(\gamma)(\iota^\Psi_{\beta, \alpha}(\iota^{\Phi,\Psi}_\beta(p \restr \beta))) && (\text{A})\\
			&= \iota^\Psi_{\beta, \alpha}(\pi^{\Psi, \theta}_{\beta, i}(\gamma)(\iota^{\Phi,\Psi}_\beta(p \restr \beta))) && (\text{Definition~\ref{DEF_Increasing_Sequence_Of_Actions}})\\
			&= \iota^\Psi_{\beta, \alpha}(\iota^{\Phi,\Psi}_\beta(\pi^{\Phi, \theta}_{\beta, i}(\gamma)(p \restr \beta))) && (\text{(B) inductively})\\
			&= \iota^{\Phi,\Psi}_\alpha(\iota^\Phi_{\beta, \alpha}(\pi^{\Phi, \theta}_{\beta, i}(\gamma)(p \restr \beta))) && (\text{A})\\
			&= \iota^{\Phi,\Psi}_\alpha(\pi^{\Phi, \theta}_{\alpha, i}(\gamma)(\iota^\Phi_{\beta, \alpha}(p \restr \beta))) && (\text{Definition~\ref{DEF_Increasing_Sequence_Of_Actions}})\\
			&= \iota^{\Phi,\Psi}_\alpha(\pi^{\Phi, \theta}_{\alpha, i}(\gamma)(p)) && (\text{choice of } \beta).
		\end{align*}
		\setcounter{enumi}{4}
		\item Let $\theta \in \dom(\Phi)$.
		Then, we compute
		\begin{align*}
			\dot{\T}_\alpha^{\Psi, \theta}
			&= \bigcup_{\beta < \alpha} \iota^\Psi_{\beta, \alpha}(\dot{\T}^{\Psi, \theta}_\beta) && (\text{Definition~\ref{DEF_Forcing}})\\
			&= \bigcup_{\beta < \alpha} \iota^\Psi_{\beta, \alpha}\left[\iota^{\Phi, \Psi}_{\beta}(\dot{\T}^{\Phi, \theta}_\beta) \cup \bigcup_{i \in \Psi(\theta) \setminus \Phi(\theta)}\iota^{\Psi}_{1, \beta}(\T^{\Psi, \theta}_{i})\right] && (\text{(E) inductively})\\
			&= \bigcup_{\beta < \alpha} \left[ \iota^\Psi_{\beta, \alpha}(\iota^{\Phi, \Psi}_{\beta}(\dot{\T}^{\Phi, \theta}_\beta)) \cup \bigcup_{i \in \Psi(\theta) \setminus \Phi(\theta)}\iota^\Psi_{\beta, \alpha}(\iota^{\Psi}_{1, \beta}(\T^{\Psi, \theta}_{i}))\right] && (\text{canonical name})\\
			&= \bigcup_{\beta < \alpha} \left[ \iota^{\Phi, \Psi}_{\alpha}(\iota^\Phi_{\beta, \alpha}(\dot{\T}^{\Phi, \theta}_\beta)) \cup \bigcup_{i \in \Psi(\theta) \setminus \Phi(\theta)}\iota^{\Psi}_{1, \alpha}(\T^{\Psi, \theta}_{i})\right] && (\text{A})\\
			&= \iota^{\Phi, \Psi}_{\alpha}\left[\bigcup_{\beta < \alpha} \iota^\Phi_{\beta, \alpha}(\dot{\T}^{\Phi, \theta}_\beta)\right] \cup \bigcup_{i \in \Psi(\theta) \setminus \Phi(\theta)}\iota^{\Psi}_{1, \alpha}(\T^{\Psi, \theta}_{i}) && (\text{canonical name})\\
			&= \iota^{\Phi, \Psi}_{\alpha}(\dot{\T}_\alpha^{\Phi, \theta}) \cup \bigcup_{i \in \Psi(\theta) \setminus \Phi(\theta)}\iota^{\Psi}_{1, \alpha}(\T^{\Psi, \theta}_{i}) && (\text{Definition~\ref{DEF_Forcing}}).
		\end{align*}
		\item Let $\theta \in \dom(\Phi)$, $i \in \Psi(\theta) \setminus \Phi(\theta)$, $\gamma \in \Gamma$ and $p \in \PP^\Phi_\alpha$.
		Choose $\beta < \alpha$ with $\iota^\Phi_{\beta,\alpha}(p \restr \beta) = p$.
		Then, we compute
		\begin{align*}
			\pi^{\Psi, \theta}_{\alpha, i}(\gamma)(\iota_{\alpha}^{\Phi, \Psi}(p)) 
			&= \pi^{\Psi, \theta}_{\alpha, i}(\gamma)(\iota_{\alpha}^{\Phi, \Psi}(\iota^\Phi_{\beta,\alpha}(p \restr \beta))) && (\text{choice of } \beta)\\
			&= \pi^{\Psi, \theta}_{\alpha, i}(\gamma)(\iota^\Psi_{\beta,\alpha}(\iota_{\beta}^{\Phi, \Psi}(p \restr \beta))) && (\text{A})\\
			&= \iota^\Psi_{\beta,\alpha}(\pi^{\Psi, \theta}_{\beta, i}(\gamma)(\iota_{\beta}^{\Phi, \Psi}(p \restr \beta))) && (\text{Definition~\ref{DEF_Increasing_Sequence_Of_Actions}})\\
			&= \iota^\Psi_{\beta,\alpha}(\iota_{\beta}^{\Phi, \Psi}(p \restr \beta)) && (\text{(F) inductively})\\
			&= \iota_{\alpha}^{\Phi, \Psi}(\iota^\Phi_{\beta,\alpha}(p \restr \beta)) && (\text{A})\\
			&= \iota_{\alpha}^{\Phi, \Psi}(p) && (\text{choice of } \beta).
		\end{align*}
	\end{enumerate}
	Finally, consider $\alpha + 1 > 1$.
	By induction we have that $\iota^{\Phi, \Psi}_\alpha:\PP^{\Phi}_\alpha \to \PP^\Psi_\alpha$ is a complete embedding.
	Thus, we may naturally define for $p \in \PP^\Phi_{\alpha + 1}$
	$$
		\iota^{\Phi, \Psi}_{\alpha + 1}(p) := \iota^{\Phi, \Psi}_{\alpha}(p \restr \alpha) \concat \iota^{\Phi, \Psi}_{\alpha}(p(\alpha)).
	$$
	However, we need to verify that
	$$
		\iota^{\Phi, \Psi}_{\alpha}(p \restr \alpha) \forces \iota^{\Phi, \Psi}_{\alpha}(p(\alpha	)) \in \dot{\QQ}^\Psi_\alpha.
	$$
	Since $\Phi \subseteq \Psi$, by definition of $\dot{\QQ}^\Psi_\alpha$ it suffices to prove that if $\theta \in \dom(\Phi)$ and $\dot{f}$ is a $\PP^\Phi_\alpha$-name with
	$$
		p \restr \alpha \forces \dot{f} \notin \bigcup_{T \in \dot{\T}^{\Phi, \theta}_\alpha} [T],
	$$
	then also
	$$
		\iota^{\Phi, \Psi}_{\alpha}(p \restr \alpha) \forces \iota^{\Phi, \Psi}_{\alpha}(\dot{f}) \notin \bigcup_{T \in \dot{\T}^{\Psi, \theta}_\alpha} [T].
	$$
	By induction assumption of $(\text{E})$ we may distinguish the following three different types of trees in $ \dot{\T}^{\Psi, \theta}_\alpha$. First, let $i \in \Phi(\theta)$ and $\gamma \in \Gamma$.
	By assumption on $\dot{f}$ we have
	$$
		p \restr \alpha \forces \dot{f} \neq \iota^{\Phi}_{1, \alpha}(\pi^{\Phi, \theta}_{1,i}(\gamma)(\dot{c}^{\Phi, \theta}_i)),
	$$
	so that
	$$
		\iota^{\Phi, \Psi}_{\alpha}(p \restr \alpha) \forces \iota^{\Phi, \Psi}_{\alpha}(\dot{f}) \neq \iota^{\Phi, \Psi}_{\alpha}(\iota^{\Phi}_{1, \alpha}(\pi^{\Phi, \theta}_{1,i}(\gamma)(\dot{c}^{\Phi, \theta}_i))).
	$$
	Secondly, let $\beta < \alpha$ and $n < \omega$.
	By assumption on $\dot{f}$ we have
	$$
		p \restr \alpha \forces \dot{f} \notin [\iota^\Phi_{\beta, \alpha}(\dot{T}^{\Phi, \theta}_{\beta, n})].
	$$
	Thus, by induction assumption of $(\text{A})$ and $(\text{C})$ we get
	$$
		\iota^{\Phi, \Psi}_{\alpha}(p \restr \alpha) \forces \iota^{\Phi, \Psi}_{\alpha}(\dot{f}) \notin [\iota^{\Phi, \Psi}_{\alpha}(\iota^\Phi_{\beta, \alpha}(\dot{T}^{\Phi, \theta}_{\beta, n}))] =
		[\iota^\Psi_{\beta, \alpha}(\iota^{\Phi, \Psi}_{\beta}(\dot{T}^{\Phi, \theta}_{\beta, n}))] = [\iota^{\Psi}_{\beta, \alpha}(\dot{T}^{\Psi, \theta}_{\beta, n})].
	$$
	Finally, for  $i \in \Psi(\theta) \setminus \Phi(\theta)$ by induction assumption of $(\text{G})$ we get
	$$
		\iota^{\Phi, \Psi}_{\alpha}(p \restr \alpha) \forces  \iota^{\Phi, \Psi}_{\alpha}(\dot{f}) \neq \iota^\Psi_{1, \alpha}(\pi^{\Psi, \theta}_{1, i}(\gamma)(\dot{c}^{\Psi, \theta}_i)).
	$$
	
	Next, given $p \in \PP_{\alpha + 1}^\Psi$ we have to find a reduction $q \in \PP_{\alpha + 1}^\Phi$ with respect to the embedding $\iota^{\Phi,\Psi}_{\alpha + 1}$.
	By Lemma~\ref{LEM_Ds_Are_Dense} we may assume $p \in D^\Psi_{\alpha + 1}$.
	By induction, pick a reduction $q \in \PP^\Phi_\alpha$ of $p \restr \alpha \in D^\Psi_\alpha$ with respect to $\iota^{\Phi, \Psi}_\alpha$.
	Remember, that for every $\theta \in \Theta^p_\alpha$, $i \in I^p_{\alpha, \theta}$ and $\dot{f} \in F^p_{\alpha, \theta, i}$ we have
	$$
		p \restr \alpha \forces \dot{f} \notin \bigcup_{T \in \dot{\T}^{\Psi, \theta}_\alpha}[T].
	$$
	Thus, we will need to find a reduction $\dot{g}$ of $\dot{f}$ which satisfies
	$$
		q \forces \dot{g} \notin \bigcup_{T \in \dot{\T}^{\Phi, \theta}_\alpha}[T].
	$$
	Note that the standard canonical projection of a real (cf. \cite{FischerTornquist_2015}) need not satisfy this requirement.
	Thus, we introduce the following technical notions.
	For technical reasons, we need to enumerate the finite set $\bigcup\set{\simpleset{\theta} \times \simpleset{i} \times F^p_{\alpha, \theta, i}}{\theta \in \Theta^p_\alpha, i \in I^p_{\alpha, \theta}}$ by $\seq{(\theta_k, i_k, \dot{f}_k)}{k \in K}$.
	In particular, we have $\theta_\bullet:K \to \Theta^p_\alpha$.
	For every $\theta \in \Theta^p_\alpha$ by assumption on $\Phi$ the family $\dot{\T}^{\Phi, \theta}_\alpha$ is countable, so we may enumerate it as $\seq{\dot{S}^\theta_n}{n < \omega}$.
	Next, we will need the 	following refinement of the definition of a nice name for a real below $p$ in Definition~\ref{DEF_Nice_Name_Real}.
	
	\begin{definition} \label{DEF_Nice_Name_Set_Of_Reals}
		Let $\PP$ be a forcing, $p \in \PP$ and $K$ a finite set.
		A nice $\PP$-name for $K$-many reals below $p$ is a sequence $\seq{(\A_n, K_n)}{n < \omega}$ such that
		\begin{enumerate}[$\bullet$]
			\item for all $n < \omega$ the set $\A_n$ is a maximal antichain below $p$ and $K_n:K \times \A_n \to 2^{>n}$,
			\item for all $n < m$ the antichain $\A_m$ refines $\A_n$, i.e.\ every $b \in \A_m$ there is $a \in \A_n$ with $b \extends a$,
			\item for all $n < m$, $k \in K$, $a \in \A_n$ and $b \in \A_m$ with $b \extends a$ we have $K_n(k, a) \trianglelefteq K_m(k,  b)$.
		\end{enumerate}
		Further, we write $\name(\seq{(\A_n, K_n)}{n < \omega})$ for the canonical $\PP$-name of $\seq{(\A_n, K_n)}{n < \omega}$, i.e.
		$$
		\name(\seq{(\A_n, K_n)}{n < \omega}) := \set{(a,((k, n),K_{n}(k, a)(n)))}{n < \omega \text{ and } a \in \A_n} \in {^{K \times \omega}2}.
		$$
	\end{definition}

	\begin{remark}
		Notice that if $\seq{(\A_n, K_n)}{n < \omega}$ is a nice $\PP$-name for $K$-many reals below $p$, then for every $k \in K$ the sequence $\seq{(\A_n, K_n(k))}{n < \omega}$ is a nice $\PP$-name for a real below $p$ with
		$$
			\name(\seq{(\A_n, K_n(k))}{n < \omega}) = \name(\seq{(\A_n, K_n)}{n < \omega}) \restr (\simpleset{k} \times \omega).
		$$
		However, $\seq{(\A_n, K_n)}{n < \omega}$ is more than just the product of $K$-many nice $\PP$-names for reals below $p$ as all antichains have to coincide.
	\end{remark}
	
	With respect to the fixed $p \in D^\Psi_{\alpha + 1}$, $\theta_{\bullet}:K \to \Theta^p_\alpha$ and sequence $\seq{\dot{S}^\theta_n}{n < \omega}$ above, we define the following notion:
	
	\begin{definition}\label{DEF_Nice_Name_WRT_Trees}
		Let $\seq{(\A_n,K_n)}{n < \omega}$ be a nice $\PP^\Psi_\alpha$-name for $K$-many reals below $p \restr \alpha$.
		Then, we say $\seq{(\A_n,K_n)}{n < \omega}$ is a nice $\PP$-name for $K$-many reals below $p \restr \alpha$ with respect to $\theta_\bullet$ and $\seq{\seq{\dot{S}^\theta_n}{n < \omega}}{\theta \in \Theta^p_\alpha}$ iff for all $n < \omega$, $k \in K$ and $a \in \A_n$ we have
		$$
			a \forces K_n(k, a) \notin \iota^{\Phi, \Psi}_\alpha(\dot{S}^{\theta_k}_n).
		$$
	\end{definition}

	First, we argue that there is such a nice $\PP^\Psi_\alpha$-name $\seq{(\A_n,K_n)}{n < \omega}$ of $K$-many reals below $p \restr \alpha$ with respect to $\theta_\bullet$ and $\seq{\seq{\dot{S}^\Phi_n}{n < \omega}}{\theta \in \Theta^p_\alpha}$, so that for every $k \in K$ we have
	$$
		p \restr \alpha \forces \dot{f}_k = \name(\seq{(\A_n, K_n(k))}{n < \omega}).
	$$
	\begin{proof}
		We construct the nice name by recursion on $n$.
		Set $\A_{-1} := \simpleset{p \restr \alpha}$.
		Now, assume $\A_n$ is defined.
		For every $a \in \A_n$ choose a maximal antichain $\B(a)$ below $a$ such that for every $b \in \B(a)$ and $k \in K$ there is $K_{n + 1}(k, a) \in 2^{>n}$ with $K_{n}(k,b) \trianglelefteq K_{n + 1}(k,b)$ if $n \neq -1$ and such that
		$$
			b \forces K_{n + 1}(k, b) \trianglelefteq \dot{f}_k \text{ and } K_{n + 1}(k, b) \notin \iota^{\Phi, \Psi}_\alpha(\dot{S}^{\theta_k}_{n + 1}).
		$$
		This is possible as $b \extends a$, $K$ is finite and by assumption on $\dot{f}_k$ we have for every $k \in K$
		$$
			p \restr \alpha \forces \dot{f}_k \notin [\iota^{\Phi, \Psi}_\alpha(\dot{S}^{\theta_k}_{n + 1})].
		$$
		Finally, set $\A_{n + 1} := \bigcup_{a \in \A_n} \B(a)$.
		Clearly, $\seq{(\A_n, K_n)}{n < \omega}$ then has the desired properties.
	\end{proof}

	In \cite{FischerTornquist_2015}[Lemma~3.8] the existence of a reduction of a nice name for a real is proven.
	We will need an analogous result for nice names of $K$-many reals:
	
	\begin{lemma}\label{LEM_Reduction_Of_Nice_Name}
		Let $\QQ$ be a complete suborder of $\PP$, $p \in \PP$, $q \in \QQ$ a reduction of $p$ and assume that $\set{(\A_n,K_n)}{n < \omega}$ a nice $\PP$-name for $K$-many reals below $p$.
		Then, there is a nice $\QQ$-name $\set{(\B_n,L_n)}{n < \omega}$ for $K$-many reals below $q$ such that for all $n < \omega$ and $b \in \B_n$ there is an $a \in \A_n$ such that $b$ is a reduction of $a$ and $K_n(k, a) = L_n(k, b)$ for all $k \in K$.
	\end{lemma}

	\begin{proof}
		Exactly the same proof as for Lemma~3.8 in \cite{FischerTornquist_2015}.
	\end{proof}

	Analogously to \cite{FischerTornquist_2015}, we will call the nice $\PP^\Phi_\alpha$-name $\set{(\B_n,L_n)}{n < \omega}$ a canonical projection of the nice $\PP^\Psi_\alpha$-name $\set{(\A_n,K_n)}{n < \omega}$ below $q$.
	
	\begin{lemma}\label{LEM_Projection_Of_Nice_WRT_Sequence}
		Assume $\seq{(\A_n,K_n)}{n < \omega}$ is a nice $\PP^\Psi_\alpha$-name for $K$-many reals below $p \restr \alpha$ with respect to $\theta_\bullet$ and $\seq{\seq{\dot{S}_n^\theta}{n < \omega}}{\theta \in \Theta^p_\alpha}$.
		Further, assume that $\seq{(\B_n,L_n)}{n < \omega}$ is a canonical projection of $\set{(\A_n,K_n)}{n < \omega}$ below $q$.
		Then, for every $k \in K$
		$$
			q \forces \name(\set{(\B_n,L_n(k))}{n < \omega}) \notin \bigcup_{n < \omega}[\dot{S}^{\theta_k}_n].
		$$
	\end{lemma}

	\begin{proof}
		Assume not, so choose $k \in K$, $n < \omega$ and $r_0 \extends q$ such that
		$$
			r_0 \forces \name(\set{(\B_n,L_n(k))}{n < \omega}) \in [\dot{S}^{\theta_k}_n].
		$$
		Choose $b \in \B_n$ such that $b \compat r_0$.
		Choose $r_1 \in \PP^{\Phi}_\alpha$ with $r_1 \extends b, r_0$.
		Since $\set{(\B_n,L_n)}{n < \omega}$ is a canonical projection below $q$ of $\set{(\A_n,K_n)}{n < \omega}$ choose $a \in \A_n$ such that $b$ is a reduction of $a$ and $K_n(k, a) = L_n(k, b)$.
		Thus, $\iota^{\Phi, \Psi}_\alpha(r_1) \compat a$.
		Then, by assumption we have
		$$
			r_1 \forces L_n(k, b) \in \dot{S}^{\theta_k}_n,
		$$
		which implies
		$$
			\iota^{\Phi, \Psi}_\alpha(r_1) \forces K_n(k, a) = L_n(k, b) \in \iota^{\Phi, \Psi}_\alpha(\dot{S}^{\theta_k}_n).
		$$
		On the other hand, since $\seq{(\A_n,f_n)}{n < \omega}$ is a nice name with respect to $\theta_\bullet$ and  $\seq{\dot{S}_n}{n < \omega}$
		$$
			a \forces K_n(k, a) \notin \iota^{\Phi, \Psi}_\alpha(\dot{S}^{\theta_k}_n)
		$$
		contradicting $\iota^{\Phi, \Psi}_\alpha(r_1) \compat a$.
	\end{proof}
	
	Finally, we define a reduction of $p$ as follows:
	By the previous discussion choose a nice $\PP^\Psi$-name $\seq{(\A_n,K_n)}{n < \omega}$ of $K$-many reals below $p \restr \alpha$ with respect to $\theta_\bullet$ and $\seq{\seq{\dot{S}^\theta_n}{n < \omega}}{\theta \in \Theta^p_\alpha}$, so that for every $k \in K$ we have
	$$
	p \restr \alpha \forces \dot{f}_k = \name(\seq{(\A_n, K_n(k))}{n < \omega}).
	$$
	By Lemma~\ref{LEM_Reduction_Of_Nice_Name} choose a canonical projection $\seq{(\B_n, L_n)}{n < \omega}$ of $\seq{(\A_n, K_n)}{n < \omega}$.
	Now, for $\theta \in \Theta^p_\alpha$ and $i \in I^p_{\alpha, \theta}$ we define $G_{\alpha, \theta, i}$ as
	$$
		\set{\name(\seq{(\B_n, L_n(k))}{n < \omega})}{k \in K \text{ with } \theta_k = \theta \text{ and } i_k = i}.
	$$
	Let $\dot{q}_\alpha$ be the canonical name for the condition in $\dot{\QQ}^\Phi_\alpha = \prod_{\theta \in \dom(\Phi)}\TT(\dot{\T}^{\Phi, \theta}_\alpha)$
	with $\supp(\dot{q}_\alpha) = \Theta^p_{\alpha}$, for every $\theta \in \Theta_{\alpha}$
	with $\supp(\dot{q}_\alpha(\theta)) = I^p_{\alpha, \theta}$ and for every $i \in I_{\alpha, \theta}$ we have $\dot{q}_\alpha(\theta)(i) = (s^p_{\alpha, \theta, i}, G_{\alpha, \theta, i})$.
	Since $\seq{\dot{S}_n^{\theta_k}}{n < \omega}$ enumerates $\dot{\T}^{\Phi, \theta_k}_{\alpha}$ by Lemma~\ref{LEM_Projection_Of_Nice_WRT_Sequence} for every $k \in K$ we have
	$$
		q \forces  \name(\seq{(\B_n, L_n(k))}{n < \omega}) \notin \bigcup_{T \in \dot{\T}^{\Phi, \theta_k}_\alpha}[T].
	$$
	Hence, we obtain
	$$
		q \forces \dot{q}_\alpha \in \dot{\QQ}^\Phi_\alpha = \prod_{\theta \in \dom(\Phi)}\TT(\dot{\T}^{\Phi, \theta}_\alpha),
	$$
	i.e.\ $q \concat \dot{q}_\alpha \in \PP^{\Phi}_{\alpha + 1}$.
	It remains to show that $q \concat \dot{q}_\alpha$ is indeed a reduction of $p$ with respect to $\iota^{\Phi, \Psi}_{\alpha + 1}$.
	
	\begin{proof}
		Let $r \extends q \concat \dot{q}_\alpha$.
		We need to show that $\iota^{\Phi,\Psi}_{\alpha + 1}(r) \compat p$.
		By extending $r$ we may assume $r \in D^{\Phi}_{\alpha + 1}$.
		Further, $r \restr \alpha \extends q$.
		Since $r \restr \alpha \forces r(\alpha) \leq \dot{q}_\alpha$ we have
		\begin{enumerate}[$\bullet$]
			\item $\Theta^p_\alpha \subseteq \Theta^r_\alpha$,
			\item $I^p_{\alpha, \theta} \subseteq I^r_{\alpha, \theta}$ for every $\theta \in \Theta^p_\alpha$,
			\item $n^p_{\alpha, \theta, i} \leq n^r_{\alpha, \theta, i}$ for every $\theta \in \Theta^p_\alpha$ and $i \in I^p_{\alpha, \theta}$,
			\item $s^p_{\alpha, \theta, i} \trianglelefteq s^r_{\alpha, \theta, i}$ for every $\theta \in \Theta^p_\alpha$ and $i \in I^p_{\alpha, \theta}$,
			\item For every $k \in K$ there is $\dot{h}_k \in F^r_{\alpha, \theta_k, i_k}$ such that
			$$
				r \restr \alpha \forces \dot{h}_k = \name(\seq{(\B_n, L_n(k))}{n < \omega}).
			$$
		\end{enumerate}
		Let $N := \max\set{n^r_{\alpha, \theta, i}}{\theta \in \Theta^p_\alpha, i \in I^p_{\alpha, \theta}}$.
		Since $r \restr \alpha \extends q$ and $\B_N$ is a maximal antichain below $q$ choose $b \in \B_N$ and $\bar{r} \in \PP^{\Phi}_\alpha$ with $\bar{r} \extends r \restr \alpha, b$.
		As $\seq{(\B_n, L_n)}{n < \omega}$ is a canonical projection of $\seq{(\A_n, K_n)}{n < \omega}$ choose $a \in \A_N$, so that $b$ is a reduction of $a$ and for all $k \in K$ we have $K_N(k,a) = L_N(k,b)$.
		Hence, $\iota^{\Phi,\Psi}_\alpha(\bar{r}) \compat a$, so choose $\bar{p} \in \PP^\Psi_\alpha$ with $\bar{p} \extends \iota^{\Phi,\Psi}_\alpha(\bar{r}), a$.
		We define
		\begin{enumerate}[$\bullet$]
			\item $\Theta^{\bar{p}}_\alpha := \Theta^{r}_\alpha$,
			\item $I^{\bar{p}}_{\alpha, \theta} := I^r_{\alpha, \theta}$ for every $\theta \in \Theta^{\bar{p}}_\alpha$,
			\item $n^{\bar{p}}_{\alpha, \theta, i} := n^r_{\alpha, \theta, i}$ and $s^{\bar{p}}_{\alpha, \theta, i} := s^r_{\alpha, \theta, i}$ for every $\theta \in \Theta^{\bar{p}}_\alpha$ and $i \in I^{\bar{p}}_{\alpha, \theta}$,
			\item $F^{\bar{p}}_{\alpha, \theta, i} := F^p_{\alpha, \theta,i} \cup \iota^{\Phi,\Psi}_{\alpha}(F^r_{\alpha, \theta, i})$ for every $\theta \in \Theta^{\bar{p}}_\alpha$ and $i \in I^{\bar{p}}_{\alpha, \theta}$,
		\end{enumerate}
		where every undefined set is to be treated as the empty set.
		Let $\dot{\bar{p}}_\alpha$ be the canonical name for the condition in $\dot{\QQ}^\Phi_\alpha = \prod_{\theta \in \dom(\Phi)}\TT(\dot{\T}^{\Phi, \theta}_\alpha)$
		with $\supp(\dot{\bar{p}}_\alpha) = \Theta^{\bar{p}}_{\alpha}$ and for every $\theta \in \Theta^{\bar{p}}_{\alpha}$
		with $\supp(\dot{\bar{p}}_\alpha(\theta)) = I^{\bar{p}}_{\alpha, \theta}$ and for every $i \in I^{\bar{p}}_{\alpha, \theta}$ we have $\dot{\bar{p}}_\alpha(\theta)(i) = (s^{\bar{p}}_{\alpha, \theta, i}, F^{\bar{p}}_{\alpha, \theta, i})$.
		By definition of $\bar{p} \concat \dot{\bar{p}}_\alpha$ we have $\bar{p} \concat \dot{\bar{p}}_\alpha \extends p, \iota^{\Phi, \Psi}_{\alpha + 1}(r)$, so we finish the proof by showing that $\bar{p} \concat \dot{\bar{p}}_\alpha \in \PP^{\Psi}_\alpha$.
		By definition of $F^{\bar{p}}_{\alpha, \theta, i}$ we distinguish the following two cases.
		First, let $k \in K$, by Remark~\ref{REM_Construct_Condition_From_Parameters} we have to prove
		$$
			\bar{p} \forces \dot{f}_k \restr n^{\bar{p}}_{\alpha, \theta_k, i_k} \in s^{\bar{p}}_{\alpha, \theta_k, i_k} \text{ and } \dot{f} \notin \bigcup_{T \in \dot{\T}^{\Psi, \theta_k}_\alpha} [T].
		$$
		Since $p \in D^{\Psi}_{\alpha + 1}$ we have
		$$
			p \restr \alpha \forces \dot{f} \notin \bigcup_{T \in \dot{\T}^{\Psi, \theta_k}_\alpha} [T],
		$$
		so also $\bar{p} \extends a \extends p \restr \alpha$ forces this.
		For the other property, choose $\dot{h}_k \in F^r_{\alpha, \theta_k, i_k}$ such that
		$$
			r \restr \alpha \forces \dot{h}_k = \name(\seq{(\B_n, L_n(k))}{n < \omega}).
		$$
		Since $r \in D^{\Phi}_{\alpha + 1}$ we have
		$$
			r \restr \alpha \forces \dot{h}_k \restr n^{r}_{\alpha, \theta_k, i_k} \in s^{r}_{\alpha, \theta_k, i_k}.
		$$
		Furthermore, as $N \geq n^{r}_{\alpha, \theta_k,i_k}$ and $b \in \B_N$ we have
		$$
			b \forces \name(\seq{(\B_n, L_n(k))}{n < \omega}) \restr n^{r}_{\alpha, \theta_k,i_k} = \restr L_N(k,b) \restr n^{r}_{\alpha, \theta_k,i_k}.
		$$
		Hence, $\bar{r} \extends r \restr \alpha, b$ implies that
		$$
			\bar{r} \forces L_N(k,b) \restr n^{r}_{\alpha, \theta_k,i_k} = \dot{h}_k \restr n^{r}_{\alpha, \theta_k,i_k} \in s^r_{\alpha, \theta_k, i_k}.
		$$
		Thus, we obtain $L_N(k,a) \restr n^{r}_{\alpha, \theta_k,i_k} \in s^{r}_{\alpha, \theta_k, i_k}$.
		But $n^{\bar{p}}_{\alpha, \theta_k, i_k} = n^{r}_{\alpha, \theta_k, i_k}$, $s^{\bar{p}}_{\alpha, \theta_k, i_k} = s^{r}_{\alpha, \theta_k, i_k}$ and by choice of $b$ we have $L_N(k,b) = K_N(k,a)$, so that
		$$
			K_N(k,a) \restr n^{\bar{p}}_{\alpha, \theta_k,i_k} \in s^{\bar{p}}_{\alpha, \theta_k, i_k}.
		$$
		Finally,
		$$
			a \forces \dot{f}_k = \name(\seq{(\B_n, L_n(k))}{n < \omega})
		$$
		and $\bar{p} \extends a$ yield the desired
		$$
			\bar{p} \forces \dot{f}_k \restr n^{\bar{p}}_{\alpha, \theta_k,i_k} = 	K_N(k,a) \restr n^{\bar{p}}_{\alpha, \theta_k,i_k} \in s^{\bar{p}}_{\alpha, \theta_k, i_k}.
		$$
		Secondly, let $\theta \in \Theta^r_\alpha$, $i \in I^r_{\alpha, \theta}$ and $\dot{h} \in F^r_{\alpha, \theta, i}$.
		Then, $\bar{r} \extends r \restr \alpha$ implies
		$$
			\bar{r} \forces \dot{h} \restr n^{r}_{\alpha, \theta, i} \in s^{r}_{\alpha, \theta, i} \text{ and } \dot{h} \notin \bigcup_{T \in \dot{\T}^{\Phi, \theta}_\alpha} [T].
		$$
		As before, we obtain
		$$
			\iota^{\Phi ,\Psi}_{\alpha}(\bar{r}) \forces \iota^{\Phi, \Psi}_{\alpha}(\dot{h}) \restr n^{r}_{\alpha, \theta, i} \in s^{r}_{\alpha, \theta, i} \text{ and } \iota^{\Phi, \Psi}_{\alpha}(\dot{h}) \notin \bigcup_{T \in \dot{\T}^{\Psi, \theta}_\alpha} [T].
		$$
		Hence, $\bar{p} \extends \iota^{\Phi, \Psi}_{\alpha}(\bar{r})$ implies that
		$$
		\bar{p} \forces \iota^{\Phi, \Psi}_{\alpha}(\dot{h}) \restr n^{\bar{p}}_{\alpha, \theta, i} \in s^{\bar{p}}_{\alpha, \theta, i} \text{ and } \iota^{\Phi, \Psi}_{\alpha}(\dot{h}) \notin \bigcup_{T \in \dot{\T}^{\Psi, \theta}_\alpha} [T].
		$$
		Thus, we finished proving $\bar{p} \concat \dot{\bar{p}}_\alpha \in \PP^{\Psi}_\alpha$.
	\end{proof}
		
	To complete the induction, it remains to prove $(\text{A})$ to $(\text{F})$:
	\begin{enumerate}[(A)]
		\item By induction on (A) it suffices to verify the following.
		Let $p \in \PP^{\Phi}_\alpha$.
		Then, we compute
		\begin{align*}
			\iota^{\Phi, \Psi}_{\alpha + 1}(\iota^{\Phi}_{\alpha, \alpha + 1}(p))
			&= \iota^{\Phi, \Psi}_{\alpha + 1}(p \concat \mathds{1})\\
			&= \iota^{\Phi, \Psi}_{\alpha}(p) \concat \mathds{1}\\
			&= \iota^{\Psi}_{\alpha, \alpha + 1}(\iota^{\Phi, \Psi}_{\alpha}(p)).
		\end{align*}
		\item Let $\theta \in \dom(\Phi)$, $i \in \Phi(\theta)$, $\gamma \in \Gamma$ and $p \in \PP^\Phi_{\alpha + 1}$.
		Then, we compute
		\begin{align*}
			\pi^{\Psi, \theta}_{\alpha + 1, i}(\gamma)(\iota^{\Phi,\Psi}_{\alpha + 1}(p))
			&= \pi^{\Psi, \theta}_{\alpha + 1, i}(\gamma)(\iota^{\Phi, \Psi}_{\alpha}(p \restr \alpha) \concat \iota^{\Phi, \Psi}_{\alpha}(p(\alpha))) && (\text{definition of } \iota^{\Phi,\Psi}_{\alpha + 1})\\
			&= \pi^{\Psi, \theta}_{\alpha, i}(\gamma)(\iota^{\Phi, \Psi}_{\alpha}(p \restr \alpha)) \concat \pi^{\Psi, \theta}_{\alpha, i}(\gamma)(\iota^{\Phi, \Psi}_{\alpha}(p(\alpha))) && (\text{Definition~\ref{DEF_Canonical_Extension}})\\
			&= \iota^{\Phi, \Psi}_{\alpha}(\pi^{\Phi, \theta}_{\alpha, i}(\gamma)(p \restr \alpha)) \concat \iota^{\Phi, \Psi}_{\alpha}(\pi^{\Phi, \theta}_{\alpha, i}(\gamma)(p(\alpha))) && (\text{(B) inductively})\\
			&= \iota^{\Phi, \Psi}_{\alpha + 1}(\pi^{\Phi, \theta}_{\alpha, i}(\gamma)(p \restr \alpha) \concat \pi^{\Phi, \theta}_{\alpha, i}(\gamma)(p(\alpha))) && (\text{definition of } \iota^{\Phi,\Psi}_{\alpha + 1})\\
			&= \iota^{\Phi, \Psi}_{\alpha + 1}(\pi^{\Phi, \theta}_{\alpha + 1, i}(\gamma)(p)) && (\text{Definition~\ref{DEF_Canonical_Extension}}).
		\end{align*}
		\item There is nothing to show.
		\item Let $\theta \in \dom(\Phi)$ and $n < \omega$.
		Then, $\iota^{\Phi, \Psi}_{\alpha + 1}(\dot{T}^{\Phi, \theta}_{\alpha, n}) = \dot{T}^{\Psi, \theta}_{\alpha, n}$
		immediately follows, since $\iota^{\Phi, \Psi}_\alpha$ preserves check-names.
		\item Let $\theta \in \dom(\Phi)$.
		Then, we compute using (E) inductively, (D), (A) and the fact that every name is chosen as a canonical name:
		\begin{align*}
			\dot{\T}_{\alpha + 1}^{\Psi, \theta}
			&= \iota^{\Psi}_{\alpha, \alpha + 1}(\dot{\T}^{\Psi, \theta}_{\alpha}) \cup \set{\dot{T}^{\Psi, \theta}_{\alpha, n}}{n \in \omega}\\
			&= \iota^{\Psi}_{\alpha, \alpha + 1}\left[\iota^{\Phi, \Psi}_{\alpha}(\dot{\T}^{\Phi, \theta}_\alpha) \cup \bigcup_{i \in \Psi(\theta) \setminus \Phi(\theta)}\iota^{\Psi}_{1, \alpha}(\T^{\Psi, \theta}_{i})\right] \cup \set{\iota^{\Phi, \Psi}_{\alpha + 1}(\dot{T}^{\Phi, \theta}_{\alpha, n})}{n \in \omega}\\
			&= \iota^{\Psi}_{\alpha, \alpha + 1}(\iota^{\Phi, \Psi}_{\alpha}(\dot{\T}^{\Phi, \theta}_\alpha)) \cup \bigcup_{i \in \Psi(\theta) \setminus \Phi(\theta)}\iota^{\Psi}_{\alpha, \alpha + 1}(\iota^{\Psi}_{1, \alpha}(\T^{\Psi, \theta}_{i})) \cup \iota^{\Phi, \Psi}_{\alpha + 1}(\set{\dot{T}^{\Phi, \theta}_{\alpha, n}}{n \in \omega})\\
			&= \iota^{\Phi, \Psi}_{\alpha + 1}(\iota^{\Phi}_{\alpha, \alpha + 1}(\dot{\T}^{\Phi, \theta}_\alpha)) \cup  \iota^{\Phi, \Psi}_{\alpha + 1}(\set{\dot{T}^{\Phi, \theta}_{\alpha, n}}{n \in \omega}) \cup \bigcup_{i \in \Psi(\theta) \setminus \Phi(\theta)}(\iota^{\Psi}_{1, \alpha + 1}(\T^{\Psi, \theta}_{i}))\\
			&= \iota^{\Phi, \Psi}_{\alpha + 1}\left[\iota^{\Phi}_{\alpha, \alpha + 1}(\dot{\T}^{\Phi, \theta}_\alpha) \cup \set{\dot{T}^{\Phi, \theta}_{\alpha, n}}{n \in \omega}\right] \cup \bigcup_{i \in \Psi(\theta) \setminus \Phi(\theta)}(\iota^{\Psi}_{1, \alpha + 1}(\T^{\Psi, \theta}_{i}))\\
			&= \iota^{\Phi, \Psi}_{\alpha + 1}(\dot{\T}_{\alpha + 1}^{\Phi, \theta}) \cup \bigcup_{i \in \Psi(\theta) \setminus \Phi(\theta)}(\iota^{\Psi}_{1, \alpha + 1}(\T^{\Psi, \theta}_{i})).
		\end{align*}
		\item Let $\theta \in \dom(\Phi)$, $i \in \Psi(\theta) \setminus \Phi(\theta)$, $\gamma \in \Gamma$ and $p \in \PP^{\Phi}_{\alpha + 1}$.
		Then, we compute
		\begin{align*}
			\pi^{\Psi, \theta}_{\alpha + 1, i}(\gamma)(\iota^{\Phi,\Psi}_{\alpha + 1}(p))
			&= \pi^{\Psi, \theta}_{\alpha + 1, i}(\gamma)(\iota^{\Phi, \Psi}_{\alpha}(p \restr \alpha) \concat \iota^{\Phi, \Psi}_{\alpha}(p(\alpha))) && (\text{definition of } \iota^{\Phi,\Psi}_{\alpha + 1})\\
			&= \pi^{\Psi, \theta}_{\alpha, i}(\gamma)(\iota^{\Phi, \Psi}_{\alpha}(p \restr \alpha)) \concat \pi^{\Psi, \theta}_{\alpha, i}(\gamma)(\iota^{\Phi, \Psi}_{\alpha}(p(\alpha))) && (\text{Definition~\ref{DEF_Canonical_Extension}})\\
			&= \iota^{\Phi, \Psi}_{\alpha}(p \restr \alpha) \concat \iota^{\Phi, \Psi}_{\alpha}(p(\alpha)) && \text{((F) inductively)}\\
			&= \iota^{\Phi,\Psi}_{\alpha + 1}(p) && (\text{definition of } \iota^{\Phi,\Psi}_{\alpha + 1}).
		\end{align*}
	\end{enumerate}
	This completes the induction and thus the proof of Theorem~\ref{THM_Complete_Subforcings}.\hfill\qedsymbol
	
	\section{Extending Isomorphisms through the iteration} \label{SEC_Extending_Isomorphisms}
	
	In Section~\ref{SEC_Extend_Automorphisms} we considered how to extend automorphisms of certain group actions through the iteration.
	Similarly, given bijections between the index sets of the Cohen reals of our iteration we will show how to extend these bijections to isomorphisms of the full iteration.
	These extension have a very categorical flavour, nevertheless we provide a self-contained presentation.
	
	\begin{definition}\label{DEF_Isomorphism}
		Let $\Phi, \Psi$ be $\Theta$-indexing functions.
		Then, we say $\mathbf{x} = (g, \set{h^{\theta}}{\theta \in \dom(\Phi)})$ is an isomorphism from $\Phi$ to $\Psi$ iff the following properties hold:
		\begin{enumerate}
			\item $g: \dom(\Phi) \to \dom(\Psi)$ is a bijection,
			\item for every $\theta \in \dom(\Phi)$ also $h^\theta:\Phi(\theta) \to \Psi(g(\theta))$ is a bijection.
		\end{enumerate}
	\end{definition}

	\begin{definition}\label{DEF_Identity_Isomorphism}
		Let $\Phi$ be a $\Theta$-indexing function.
		Then, we define the identity isomorphism from $\Phi$ to $\Phi$ by $\mathbf{1}_\Phi := (\id_{\dom(\Phi)}, \set{\id_{\Phi(\theta)}}{\theta \in \dom(\Phi)})$.
	\end{definition}

	\begin{definition}\label{DEF_Composition_Of_Isomorphisms}
		Let $\Phi, \Psi, \Chi$ be $\Theta$-indexing functions, $\mathbf{x}_0 = (g_0, \set{h_0^{\theta}}{\theta \in \dom(\Phi)})$ an isomorphism from $\Phi$ to $\Psi$ and $\mathbf{x}_1 = (g_1, \set{h_1^{\theta}}{\theta \in \dom(\Phi)})$ is an isomorphism from $\Psi$ to $\Phi$.
		Then, we define its composition $\mathbf{x}_1 \circ \mathbf{x}_0 := (g_2, \set{h_2^{\theta}}{\theta \in \dom(\Phi)})$ by
		\begin{enumerate}
			\item $g_2 := g_1 \circ g_0$,
			\item for every $\theta \in \dom(\Phi)$ we define $h_2^\theta := h^{g_0(\theta)}_1 \circ h^\theta_0$.
		\end{enumerate}
		Clearly, $\mathbf{x}_1 \circ \mathbf{x}_0$ is an isomorphism from $\Phi$ to $\Psi$ and it is easy to check, that composition is associative and the identity isomorphism satisfies left and right unit laws.
		In other words, the class of all $\Theta$-indexing functions with isomorphisms as morphisms is a category.
	\end{definition}
	
	\begin{definition}\label{DEF_Inverse_Of_Isomorphism}
		Let $\Phi, \Psi$ be $\Theta$-indexing functions and $\mathbf{x} = (g, \set{h^{\theta}}{\theta \in \dom(\Phi)})$ an isomorphism from $\Phi$ to $\Psi$.
		Then, we define its inverse $\mathbf{x}^{-1} := (g_\ast, \set{h_\ast^{\theta}}{\theta \in \dom(\Psi)})$ by
		\begin{enumerate}
			\item $g_\ast := g^{-1}$,
			\item for every $\theta \in \Theta$ we define $h_\ast^\theta := (h^{g^{-1}(\theta)})^{-1}$.
		\end{enumerate}
		Clearly, $\mathbf{x}^{-1}$ is an isomorphism from $\Psi$ to $\Phi$ and it is easy to check, that it is the unique isomorphism which satisfies $\mathbf{x}^{-1} \circ \mathbf{x} = \mathbf{1}_\Phi$ and $\mathbf{x} \circ \mathbf{x}^{-1} = \mathbf{1}_\Psi$.
		In other words, we not only have a category but a groupoid.
	\end{definition}
	
	\begin{definition}\label{DEF_Initial_Functor}
		Let $\Phi, \Psi$ be $\Theta$-indexing functions and $\mathbf{x} = (g, \set{h^{\theta}}{\theta \in \dom(\Phi)})$ an isomorphism from $\Phi$ to $\Psi$.
		Define $\kappa_\mathbf{x}:\CC^\Phi \to \CC^\Psi$ for $p \in \CC^\theta, \theta \in \dom(\Psi)$ and $i \in \Psi(\theta)$ by
		$$
			\kappa_\mathbf{x}(p)(\theta, i) := p(g^{-1}(\theta), (h^{g^{-1}(\theta)})^{-1}(i)).
		$$
		In other words, the information of $p$ is swapped around as given by the bijections $g$ and $h^\theta$.
		Clearly, $\kappa_{\mathbf{x}}$ is an isomorphism from the partial order $\CC^\Phi$ to $\CC^\Psi$.
	\end{definition}

	\begin{lemma}\label{LEM_Initial_Functor}
		Let $\Phi, \Psi, \Chi$ be $\Theta$-indexing functions, $\mathbf{x}_0 = (g_0, \set{h_0^{\theta}}{\theta \in \dom(\Phi)})$ an isomorphism from $\Phi$ to $\Psi$ and $\mathbf{x}_1 = (g_1, \set{h_1^{\theta}}{\theta \in \dom(\Psi)})$ is an isomorphism from $\Chi$ to $\Psi$.
		Then, we have
		\begin{enumerate}
			\item $\kappa_{\mathbf{1}_\Phi} = \id_{\CC^\Phi}$,
			\item $\kappa_{\mathbf{x}_1 \circ \mathbf{x}_0} = \kappa_{\mathbf{x}_1} \circ \kappa_{\mathbf{x}_0}$.
		\end{enumerate}
		In other words, $\kappa_{\bullet}$ is a functor between groupoids.
	\end{lemma}

	\begin{proof}
		For the first statement let $p \in \CC^\Phi$, $\theta \in \dom(\Phi)$ and $i \in \Phi(\theta)$.
		Then, we compute
		\begin{align*}
			\kappa_{\mathbf{1}_\Phi}(p)(\theta, i)
			&= p((\id_{\dom(\Phi)}^{-1}(\theta), (\id_{\Phi(\theta)})^{-1}(i)) && (\text{Definition~\ref{DEF_Identity_Isomorphism}~and~\ref{DEF_Initial_Functor}})\\
			&= p(\theta, i).
		\end{align*}
		Secondly, let $p \in \CC^\Phi$, $\theta \in \dom(\Chi)$ and $i \in \Chi(\theta)$.
		Then, we compute
		\begin{align*}
			\kappa_{\mathbf{x}_1 \circ \mathbf{x}_0}(p)(\theta, i)
			&= p((g_1 \circ g_0)^{-1}(\theta), (h_1^{(g_0 \circ (g_1 \circ g_0)^{-1})(\theta)} \circ h_0^{(g_1 \circ g_0)^{-1}(\theta)})^{-1}(i)) && (\text{Definition~\ref{DEF_Composition_Of_Isomorphisms}~and~\ref{DEF_Initial_Functor}})\\
			&= p(g_0^{-1}(g_1^{-1}(\theta)), (h_1^{g_1^{-1}(\theta)} \circ h_0^{g_0^{-1}(g_1^{-1}(\theta))})^{-1}(i))\\
			&= p(g_0^{-1}(g_1^{-1}(\theta)), (h_0^{g_0^{-1}(g_1^{-1}(\theta))})^{-1}((h_1^{g_1^{-1}(\theta)})^{-1}(i)))\\
			&= \kappa_{\mathbf{x}_0}(p)(g_1^{-1}(\theta), (h_1^{g_1^{-1}(\theta)})^{-1}(i)) && (\text{Definition~\ref{DEF_Initial_Functor}})\\
			&= \kappa_{\mathbf{x}_1}(\kappa_{\mathbf{x}_0}(p))(\theta, i) && (\text{Definition~\ref{DEF_Initial_Functor}})\\
			&= (\kappa_{\mathbf{x}_1} \circ \kappa_{\mathbf{x}_0})(p)(\theta, i). &&\hfill\qedhere
		\end{align*}
	\end{proof}
	
	Next, we need to verify that the canonical $\CC^\Phi$-names $\dot{\T}^{\Phi, \theta}_i$ are mapped to $\dot{\T}^{\Psi, g(\theta)}_{h^\theta(i)}$ by $\kappa_\mathbf{x}$.
	To this end, we prove that $\kappa_\mathbf{x}$ behaves nicely with respect to the $\Gamma$-actions.
	
	\begin{lemma}\label{LEM_Funtor_Commutes_With_Gamma}
		Let $\Phi, \Psi$ be $\Theta$-indexing functions and $\mathbf{x} = (g, \set{h^{\theta}}{\theta \in \dom(\Phi)})$ is an isomorphism from $\Phi$ to $\Psi$.
		Let $\theta \in \dom(\Phi)$ and $i \in \Phi(\theta)$.
		Then, $\kappa_\mathbf{x}:\CC^\Phi \to \CC^\Psi$ is a morphism of $\Gamma$-sets, i.e.\ the following diagram commutes for every $\gamma \in \Gamma$:
		\begin{center}
			\begin{tikzcd}
				\CC^\Phi \arrow[r, "\kappa_{\mathbf{x}}"] \arrow[d, "{\pi^{\Phi, \theta}_i(\gamma)}"] & \CC^{\Psi} \arrow[d, "{\pi^{\Psi, g(\theta)}_{h^{\theta}(i)}(\gamma)}"] \\
				\CC^\Phi \arrow[r, "\kappa_{\mathbf{x}}"]                                             & \CC^{\Psi}                                                             
			\end{tikzcd}
		\end{center}
	\end{lemma}

	\begin{proof}
		Let $\gamma \in \Gamma$ and $p \in \CC^\Phi$.
		Further, let $\eta \in \dom(\Psi)$ and $j \in \Psi(\eta)$.
		In case that $\theta = g^{-1}(\eta)$ and $i = (h^{g^{-1}(\eta)})^{-1}(j)$ we compute:
		\begin{align*}
			\kappa_{\mathbf{x}}(\pi^{\Phi, \theta}_i(\gamma)(p))(\eta, j) 
			&= \pi^{\Phi, \theta}_i(\gamma)(p)(g^{-1}(\eta), (h^{g^{-1}(\eta)})^{-1}(j)) && (\text{Definition~\ref{DEF_Initial_Functor}})\\
			&= \pi^{\Phi, \theta}_i(\gamma)(p)(\theta, i) && (\text{case property of } \eta, i)\\
			&= \pi(\gamma)(p(\theta, i)) && (\text{Definition~\ref{DEF_Action_Induced}})\\
			&= \pi(\gamma)(p(g^{-1}(\eta), (h^{g^{-1}(\eta)})^{-1}(j))) && (\text{case property of } \eta, i)\\
			&= \pi(\gamma)(\kappa_\mathbf{x}(p)(\eta, j)) && (\text{Definition~\ref{DEF_Initial_Functor}})\\
			&= \pi^{\Psi, \eta}_{j}(\gamma)(\kappa_\mathbf{x}(p))(\eta, j) && (\text{Definition~\ref{DEF_Action_Induced}})\\
			&= \pi^{\Psi, g(\theta)}_{h^\theta(i)}(\gamma)(\kappa_\mathbf{x}(p))(\eta, j) && (\text{case property of } \eta, i).
		\end{align*}
		Otherwise, we have that $\pi^{\Phi, \theta}_i$ acts trivially on the $(g^{-1}(\eta), (h^{g^{-1}(\eta)})^{-1}(j))$-component of $p$ and $\pi^{\Psi, g(\theta)}_{h^\theta(i)}(\gamma)$ acts trivially on the $(\eta, j)$-component of $\kappa_\mathbf{x}(p)$, so we compute
		\begin{align*}
			\kappa_{\mathbf{x}}(\pi^{\Phi, \theta}_i(\gamma)(p))(\eta, j)
			&= \pi^{\Phi, \theta}_i(\gamma)(p)(g^{-1}(\eta), (h^{g^{-1}(\eta)})^{-1}(j)) && (\text{Definition~\ref{DEF_Initial_Functor}})\\
			&= p(g^{-1}(\eta), (h^{g^{-1}(\eta)})^{-1}(j)) && (\pi^{\Phi, \theta}_i \text{ acts trivially})\\
			&= \kappa_\mathbf{x}(p)(\eta, j) && (\text{Definition~\ref{DEF_Initial_Functor}})\\
			&= \pi^{\Psi, g(\theta)}_{h^\theta(i)}(\gamma)(\kappa_\mathbf{x}(p))(\eta, j)  && (\pi^{\Psi, g(\theta)}_{h^\theta(i)} \text{ acts trivially}). \qedhere
		\end{align*}
	\end{proof}
	
	\begin{lemma} \label{LEM_Initial_Functor_Preserves_Trees}
		Let $\Phi, \Psi$ be $\Theta$-indexing functions and $\mathbf{x} = (g, \set{h^{\theta}}{\theta \in \dom(\Phi)})$ is an isomorphism from $\Phi$ to $\Psi$.
		Let $\theta \in \dom(\Phi)$ and $i \in \Phi(\theta)$.
		Then, we have
		\begin{enumerate}
			\item $\kappa_\mathbf{x}(\dot{c}^{\Phi, \theta}_i) = \dot{c}^{\Psi, g(\theta)}_{h^\theta(i)}$ and thus $\kappa_\mathbf{x}(\dot{T}^{\Phi, \theta}_i) = \dot{T}^{\Psi, g(\theta)}_{h^\theta(i)}$,
			\item $\kappa_\mathbf{x}(\dot{\T}^{\Phi, \theta}_{i}) = \dot{\T}^{\Psi, g(\theta)}_{h^\theta(i)}$,
			\item $\kappa_\mathbf{x}(\dot{\T}^{\Phi, \theta}) = \dot{\T}^{\Psi, g(\theta)}$.
		\end{enumerate}
	\end{lemma}

	\begin{proof}
		$(1)$ immediately follows from the definition of $\kappa_\mathbf{x}$ and the definition of the canonical name for a Cohen real.
		For $(2)$ by Definition~\ref{DEF_Names_For_Trees} remember $\dot{\T}^{\Phi, \theta}_{i}$ is the canonical $\CC^\Phi$-name for the set
		$$
			\set{\pi^{\Phi, \theta}_i(\gamma)(\dot{T}^{\Phi, \theta}_i)}{\gamma \in \Gamma}.
		$$
		Hence, we compute
		\begin{align*}
			\kappa_\mathbf{x}(\dot{\T}^{\Phi, \theta}_{i})
			&= \kappa_\mathbf{x}(\set{\pi^{\Phi, \theta}_i(\gamma)(\dot{T}^{\Phi, \theta}_i)}{\gamma \in \Gamma}) && (\text{Definition~\ref{DEF_Names_For_Trees}})\\
			&= \set{\kappa_\mathbf{x}(\pi^{\Phi, \theta}_i(\gamma)(\dot{T}^{\Phi, \theta}_i))}{\gamma \in \Gamma} && (\text{canonical name})\\
			&= \set{\pi^{\Psi, g(\theta)}_{h^\theta(i)}(\gamma)(\kappa_\mathbf{x}(\dot{T}^{\Phi, \theta}_i))}{\gamma \in \Gamma} && (\text{Lemma~\ref{LEM_Funtor_Commutes_With_Gamma}})\\
			&= \set{\pi^{\Psi, g(\theta)}_{h^\theta(i)}(\gamma)(\dot{T}^{\Psi, g(\theta)}_{h^\theta(i)})}{\gamma \in \Gamma} && (1)\\
			&= \dot{T}^{\Psi, g(\theta)}_{h^\theta(i)} && (\text{Definition~\ref{DEF_Names_For_Trees}}).
		\end{align*}
		Finally, for $(3)$ we compute
		\begin{align*}
			\kappa_\mathbf{x}(\dot{\T}^{\Phi, \theta})
			&= \kappa_\mathbf{x}(\bigcup_{i \in \Phi(\theta)}\dot{\T}^{\Phi, \theta}_{i}) && (\text{Remark~\ref{REM_Properties_Of_Automorphism_Subgroup}})\\
			&= \bigcup_{i \in \Phi(\theta)}\kappa_\mathbf{x}(\dot{\T}^{\Phi, \theta}_{i}) && (\text{canonical name})\\
			&= \bigcup_{i \in \Phi(\theta)}\dot{\T}^{\Psi, g(\theta)}_{h^\theta(i)} && (2)\\
			&= \bigcup_{i \in \Psi(g(\theta))}\dot{\T}^{\Psi, g(\theta)}_{i} && (h^\theta:\Phi(\theta) \to \Psi(g(\theta)) \text{ is a bijection})\\
			&= \dot{\T}^{\Psi, g(\theta)} && (\text{Remark~\ref{REM_Properties_Of_Automorphism_Subgroup}}).\qedhere
		\end{align*}
	\end{proof}
	
	So far, we have constructed a functor $\kappa_\bullet$ mapping indexing functions $\Phi$ to posets of the form $\CC^\Phi$.
	In terms of our iteration this corresponds to a functor $\kappa^1_\bullet$ mapping $\Theta$-indexing functions to posets of the form $\PP^\Phi_1$.
	We will extend these functors through the iteration to obtain an increasing sequence of functors in the following sense:
	
	\begin{definition} \label{DEF_Increasing_Sequence_Of_Functors}
		Let $\epsilon \leq \aleph_1$.
		We say that
		$$
			\seq{\kappa^\alpha_\bullet}{0 < \alpha \leq \epsilon}
		$$
		is an increasing sequence of functors iff every $\kappa^\alpha_\bullet$ is a functor mapping $\Theta$-indexing functions $\Phi$ to posets $\PP^\Phi_\alpha$, for all $0 < \alpha \leq \epsilon$, $\Phi, \Psi$ $\Theta$-indexing functions, $\mathbf{x} = (g, \set{h^{\theta}}{\theta \in \dom(\Phi)})$ an isomorphism from $\Phi$ to $\Psi$ and $\theta \in \dom(\Phi)$ we have
		$$
			\kappa_\mathbf{x}(\dot{\T}^{\Phi, \theta}) = \dot{\T}^{\Psi, g(\theta)}
		$$
		and  for every $0 < \alpha \leq \beta \leq \aleph_1$, $\Theta$-indexing functions $\Phi, \Psi$ and isomorphism $\mathbf{x}$ from $\Phi$ to $\Psi$ the following diagram commutes:
		\begin{center}
			\begin{tikzcd}
				\PP^{\Phi}_{\alpha} \arrow[r, "\kappa_{\mathbf{x}}^\alpha"] \arrow[d, "{\iota^{\Phi}_{\alpha,\beta}}"] & \PP^{\Psi}_{\alpha} \arrow[d, "{\iota^{\Psi}_{\alpha,\beta}}"] \\
				\PP^{\Phi}_{\beta} \arrow[r, "\kappa^\beta_{\mathbf{x}}"]                                              & \PP^{\Psi}_{\beta}                                            
			\end{tikzcd}
		\end{center}
		In other words, for every $0 < \alpha \leq \beta \leq \aleph_1$ the maps $\iota^\bullet_{\alpha,\beta}$ are a natural transformation from the functor $\kappa^\alpha_\bullet$ to the functor $\kappa^\beta_\bullet$.
	\end{definition}

	\begin{corollary}
		$\seq{\kappa^\alpha_\bullet}{0 < \alpha \leq 1}$ is an increasing sequence of functors (of length 1).
	\end{corollary}
	
	\begin{proof}
		By Lemma~\ref{LEM_Initial_Functor} $\kappa^\alpha_\bullet$ is a functor, the second property of Definition~\ref{DEF_Increasing_Sequence_Of_Functors} holds by Lemma~\ref{LEM_Initial_Functor_Preserves_Trees}, and the third property is vacuous for a sequence of length 1.
	\end{proof}

	Note the similarity to Definition~\ref{DEF_Increasing_Sequence_Of_Actions} and Lemma~\ref{LEM_Initial_Automorphism_Preservation}.
	In Section~\ref{SEC_Extend_Automorphisms} we made sure to preserve some group structure of automorphisms through the iteration.
	Similarly, in this section we need to preserve the groupoid structure given by isomorphisms between $\Theta$-indexing functions.
	
	\begin{proposition} \label{LEM_Limits_Of_Increasing_Functors_Are_Unique}
		Let $\epsilon \leq \aleph_1$ be a limit.
		Assume
		$$
			\seq{\kappa^\alpha_\bullet}{0 < \alpha < \epsilon}
		$$
		is an increasing sequence of functors.
		Then, there is a unique functor $\kappa^\epsilon_\bullet$ so that
		$$
			\seq{\kappa^\alpha_\bullet}{0 < \alpha \leq \epsilon}
		$$
		is an increasing sequences of functors.
	\end{proposition}
	
	\begin{proof}
		Define $\kappa^\epsilon_\bullet$ as the pointwise direct limit of $\seq{\kappa^\alpha_\bullet}{0 < \alpha < \epsilon}$.
		That is, for given $\Theta$-indexing functions $\Phi, \Psi$ and an isomorphism $\mathbf{x} = (g, \set{h^\theta}{\theta \in \Theta})$ from $\Phi$ to $\Psi$ we define $\kappa^\epsilon_\mathbf{x}$ to be the direct limit of $\seq{\kappa^\alpha_\mathbf{x}}{\alpha < \epsilon}$.
		Then, argue as in Lemma~\ref{LEM_Limits_Of_Increasing_Sequences_Are_Unique}.
	\end{proof}
	
	Analogously to Definition~\ref{DEF_Canonical_Extension} the extension at successor steps is not unique.
	However, there is a canonical way to extend an increasing sequence of functors.
	
	\begin{definition}\label{DEF_Canonical_Extension_Functor}
		Let $\epsilon \leq \aleph_1$.
		Assume
		$
			\seq{\kappa^\alpha_\bullet}{0 < \alpha \leq \epsilon}
		$
		is an increasing sequence of functors.
		Let $\Phi, \Psi$ be $\Theta$-indexing functions and $\mathbf{x} = (g, \set{h^\theta}{\theta \in \dom(\Phi)})$ an isomorphism from $\Phi$ to $\Psi$.
		Then, we define $\kappa^{\epsilon + 1}_\mathbf{x}:\PP^{\Phi}_{\epsilon + 1} \to \PP^{\Psi}_{\epsilon + 1}$ for $p \in \PP^\Phi_{\epsilon + 1}$ by
		$$
			\kappa^{\epsilon + 1}_{\mathbf{x}}(p) := \kappa^{\epsilon}_{\mathbf{x}}(p \restr \alpha) \concat \kappa_\mathbf{x}^\epsilon(p(\epsilon)).
		$$
		Then, we call $\kappa^{\epsilon + 1}_\bullet$ the canonical extension of $\seq{\kappa^\alpha_\bullet}{0 < \alpha \leq \epsilon}$.
	\end{definition}
	
	Finally, analogous to Lemma~\ref{LEM_Canonical_Extension} and Corollary~\ref{COR_Induced_Sequence_Of_Actions} we obtain our desired induced sequence of with the following lemma.
	
	\begin{lemma} \label{LEM_Canonical_Extension_Functor}
		Let $\epsilon < \aleph_1$.
		Assume
		$
			\seq{\kappa^\alpha_\bullet}{0 < \alpha \leq \epsilon}
		$
		is an increasing sequence of functors and let $\kappa^{\epsilon + 1}_\bullet$ be the canonical extension.
		Then
		$
			\seq{\kappa^\alpha_\bullet}{0 < \alpha \leq \epsilon + 1}
		$
		is an increasing sequence of functors.
	\end{lemma}

	\begin{corollary}\label{COR_Induced_Sequence_Of_Functors}
		There is an increasing sequence of functors
		$
			\seq{\kappa^\alpha_\bullet}{0 < \alpha \leq \aleph_1}
		$
		such that $\kappa^{\epsilon + 1}_\bullet$ the canonical extension of $\seq{\kappa^\alpha_\bullet}{0 < \alpha \leq \epsilon}$ for every $\epsilon < \aleph_1$.
		We call this sequence the induced sequence of functors and will reserve the notions $\seq{\kappa^\alpha_\bullet}{0 < \alpha \leq \aleph_1}$ for it.
	\end{corollary}
	
	\begin{proof}
		We iteratively construct the desired sequence.
		By Lemma~\ref{LEM_Initial_Functor_Preserves_Trees} we may start with $\kappa^\alpha_\bullet$ as in Definition~\ref{DEF_Initial_Functor}, use Lemma~\ref{LEM_Canonical_Extension_Functor} for the successor step and Lemma~\ref{LEM_Limits_Of_Increasing_Functors_Are_Unique} for the limit step.
	\end{proof}

	The final ingredient we will need for the proof of Main Theorem~\ref{THM_AT_Spectrum} is a notion of restriction for isomorphisms between $\Theta$-indexing functions.
	We also show inductively that our increasing sequence of functors in Corollary~\ref{COR_Induced_Sequence_Of_Functors} maps restrictions to restrictions.

	\begin{definition}\label{DEF_Restriction}
		Let $\Phi, \Psi$ be $\Theta$-indexing functions, $\mathbf{x} = (g, \set{h^{\theta}}{\theta \in \dom(\Phi)})$ is an isomorphism from $\Phi$ to $\Psi$ and $\Phi_0 \subseteq \Phi$ a $\Theta$-subindexing function.
		Then, we define the image of $\Phi_0$ under $\mathbf{x}$ denoted by $\mathbf{x}[\Phi_0]$ as the $\Theta$-subindexing function of $\Psi$ defined by $\dom(\mathbf{x}[\Phi_0]) := g(\dom(\Phi_0))$ and for $\theta \in \dom(\mathbf{x}[\Phi_0])$ by
		$$
		\mathbf{x}[\Phi_0](\theta) := \set{h^{g^{-1}(\theta)}(i)}{ i \in \Phi_0(g^{-1}(\theta))}.
		$$
		The restriction of $\mathbf{x}$ to $\Phi_0$ denoted by $\mathbf{x} \restr \Phi_0$ is the isomorphism from $\Phi_0$ to $\mathbf{x}[\Phi_0]$ is defined by
		$$
		\mathbf{x} \restr \Phi_0 := (g \restr \dom(\Phi_0), \set{h^{\theta} \restr \Phi_0(\theta)}{\theta \in \dom(\Phi_0)}).
		$$
	\end{definition}
	
	\begin{lemma}\label{LEM_Restriction_Maps_To_Restriction}
		Let $\Phi, \Psi$ be $\Theta$-indexing functions, $\mathbf{x} = (g, \set{h^{\theta}}{\theta \in \dom(\Phi)})$ is an isomorphism from $\Phi$ to $\Psi$ and $\Phi_0 \subseteq \Phi$ a $\Theta$-subindexing function.
		Set $\Psi_0 := \mathbf{x}[\Phi_0]$.
		Then, the following diagram commutes
		\begin{center}
			\begin{tikzcd}
				\CC^{\Phi_0} \arrow[r, "\kappa_{\mathbf{x} \restr \Phi_0}"] \arrow[d, "{\iota^{\Phi_0, \Phi}}"] & \CC^{\Psi_0} \arrow[d, "{\iota^{\Psi_0, \Psi}}"] \\
				\CC^\Phi \arrow[r, "\kappa_{\mathbf{x}}"]                                                       & \CC^{\Psi}                                      
			\end{tikzcd}
		\end{center}
	\end{lemma}
	
	\begin{proof}
		Let $p \in \CC^{\Phi_0}$, $\theta \in \dom(\Psi)$ and $i \in \Psi(\theta)$.
		If $\theta \in \dom(\Psi_0)$ and $i \in \Psi_0(\theta)$, we compute
		\begin{align*}
			\iota^{\Psi_0, \Psi}(\kappa_{\mathbf{\mathbf{x} \restr \Phi_0}}(p))(\theta, i)
			&= \kappa_{\mathbf{\mathbf{x} \restr \Phi_0}}(p)(\theta, i) && (i \in \Psi_0(\theta))\\
			&= p(g^{-1}(\theta), (h^{g^{-1}(\theta)})^{-1}(i)) && (\text{Definition~\ref{DEF_Initial_Functor}})\\
			&= \iota^{\Phi_0, \Phi}(p)(g^{-1}(\theta), (h^{g^{-1}(\theta)})^{-1}(i)) && ((h^{g^{-1}(\theta)})^{-1}(i) \in \Phi_0(\theta))\\
			&= \kappa_{\mathbf{x}}(\iota^{\Phi_0, \Phi}(p))(\theta, i) && (\text{Definition~\ref{DEF_Initial_Functor}})
		\end{align*}
		Otherwise, $\theta \in \dom(\Psi) \setminus \dom(\Psi_0)$ or $i \in \Psi(\theta) \setminus \Psi_0(\theta)$.
		Then, we have $g^{-1}(\theta) \in \dom(\Phi) \setminus \dom(\Phi_0)$ or $(h^{g^{-1}(\theta)})^{-1}(i) \in \Phi(\theta) \setminus \Phi_0(\theta)$, respectively.
		Then, we compute
		\begin{align*}
			\iota^{\Psi_0, \Psi}(\kappa_{\mathbf{\mathbf{x} \restr \Phi_0}}(p))(\theta, i)
			&= \mathds{1}\\
			&= \iota^{\Phi_0, \Phi}(p)(g^{-1}(\theta), (h^{g^{-1}(\theta)})^{-1}(i))\\
			&= \kappa_{\mathbf{x}}(\iota^{\Phi_0, \Phi}(p))(\theta, i) && (\text{Definition~\ref{DEF_Initial_Functor}}). \qedhere
		\end{align*}
	\end{proof}

	Inductively, we show that this commutative diagram not only holds for $\kappa^1_\bullet$, but for the entire increasing of functors $\seq{\kappa^\alpha_\bullet}{0 < \alpha \leq \aleph_1}$.

	\begin{lemma}\label{LEM_Restriction_Maps_To_Restriction_Succ}
		Let $\epsilon < \aleph_1$.
		Let $\Phi, \Psi$ be $\Theta$-indexing function, $\mathbf{x} = (g, \set{h^{\theta}}{\theta \in \Theta})$ is an isomorphism from $\Phi$ to $\Psi$ and $\Phi_0 \subseteq \Phi$ a $\Theta$-subindexing function.
		Set $\Psi_0 := \mathbf{x}[\Phi_0]$.
		Then, the following diagram commutes
		\begin{center}
			\begin{tikzcd}
				\PP^{\Phi_0}_{\epsilon + 1} \arrow[r, "\kappa^{\epsilon + 1}_{\mathbf{x} \restr \Phi_0}"] \arrow[d, "{\iota^{\Phi_0, \Phi}_{\epsilon + 1}}"] & \PP^{\Psi_0}_{\epsilon + 1} \arrow[d, "{\iota^{\Psi_0, \Psi}_{\epsilon + 1}}"] \\
				\PP^\Phi_{\epsilon + 1} \arrow[r, "\kappa^{\epsilon + 1}_{\mathbf{x}}"]                                                       & \PP^{\Psi}_{\epsilon + 1}                                     
			\end{tikzcd}
		\end{center}
	\end{lemma}

	\begin{proof}
		Let $p \in \PP^{\Phi_0}_{\epsilon + 1}$.
		Then, we compute
		\begin{align*}
			\iota^{\Psi_0, \Psi}_{\epsilon + 1}(\kappa^{\epsilon + 1}_{\mathbf{x} \restr \Phi_0}(p))
			&= \iota^{\Psi_0, \Psi}_\epsilon(\kappa^{\epsilon + 1}_{\mathbf{x} \restr \Phi_0}(p) \restr \epsilon) \concat \iota^{\Psi_0, \Psi}_\epsilon(\kappa^{\epsilon + 1}_{\mathbf{x} \restr \Phi_0}(p)(\epsilon)) && (\text{see Section~\ref{SEC_Complete_Embeddings}})\\
			&= \iota^{\Psi_0, \Psi}_\epsilon(\kappa^{\epsilon}_{\mathbf{x} \restr \Phi_0}(p \restr \epsilon)) \concat \iota^{\Psi_0, \Psi}_\epsilon(\kappa_\mathbf{x}^\epsilon(p(\epsilon))) && (\text{Definition~\ref{DEF_Canonical_Extension_Functor}})\\
			&= \kappa^{\epsilon}_{\mathbf{x}}(\iota^{\Phi_0, \Phi}_{\epsilon}(p \restr \epsilon)) \concat \kappa_\mathbf{x}^\epsilon(\iota^{\Phi_0, \Phi}_\epsilon(p(\epsilon))) && (\text{induction})\\
			&= \kappa^{\epsilon}_{\mathbf{x}}(\iota^{\Phi_0, \Phi}_{\epsilon + 1}(p) \restr \epsilon) \concat \kappa_\mathbf{x}^\epsilon(\iota^{\Phi_0, \Phi}_{\epsilon + 1}(p)(\epsilon)) && (\text{see Section~\ref{SEC_Complete_Embeddings}})\\
			&= \kappa^{\epsilon + 1}_{\mathbf{x}}(\iota^{\Phi_0, \Phi}_{\epsilon + 1}(p))  && (\text{Definition~\ref{DEF_Canonical_Extension_Functor}}). \qedhere
		\end{align*}
	\end{proof}

	\begin{lemma}\label{LEM_Restriction_Maps_To_Restriction_Lim}
		Let $\epsilon \leq \aleph_1$ be a limit.
		Let $\Phi, \Psi$ be $\Theta$-indexing function, $\mathbf{x} = (g, \set{h^{\theta}}{\theta \in \Theta})$ is an isomorphism from $\Phi$ to $\Psi$ and $\Phi_0 \subseteq \Phi$ a $\Theta$-subindexing function.
		Set $\Psi_0 := \mathbf{x}[\Phi_0]$.
		Then, the following diagram commutes
		\begin{center}
			\begin{tikzcd}
				\PP^{\Phi_0}_{\epsilon} \arrow[r, "\kappa^{\epsilon}_{\mathbf{x} \restr \Phi_0}"] \arrow[d, "{\iota^{\Phi_0, \Phi}_{\epsilon}}"] & \PP^{\Psi_0}_{\epsilon} \arrow[d, "{\iota^{\Psi_0, \Psi}_{\epsilon}}"] \\
				\PP^\Phi_{\epsilon} \arrow[r, "\kappa^{\epsilon}_{\mathbf{x}}"]                                                       & \PP^{\Psi}_{\epsilon}                                     
			\end{tikzcd}
		\end{center}
	\end{lemma}

	\begin{proof}
		Let $p \in \PP^{\Phi_0}_\epsilon$.
		Choose $\alpha < \epsilon$ such that $\iota^{\Phi_0}_{\alpha, \epsilon}(p \restr \alpha) = p$.
		Then, we compute
		\begin{align*}
			\iota^{\Psi_0, \Psi}_{\epsilon}(\kappa^{\epsilon}_{\mathbf{x} \restr \Phi_0}(p))
			&= \iota^{\Psi_0, \Psi}_{\epsilon}(\kappa^{\epsilon}_{\mathbf{x} \restr \Phi_0}(\iota^{\Phi_0}_{\alpha, \epsilon}(p \restr \alpha))) && (\text{choice of } \alpha)\\
			&= \iota^{\Psi_0, \Psi}_{\epsilon}(\iota^{\Psi_0}_{\alpha, \epsilon}(\kappa^{\alpha}_{\mathbf{x} \restr \Phi_0}(p \restr \alpha))) && (\text{Definition~\ref{DEF_Increasing_Sequence_Of_Functors}})\\
			&= \iota^{\Psi}_{\alpha, \epsilon}(\iota^{\Psi_0, \Psi}_{\alpha}(\kappa^{\alpha}_{\mathbf{x} \restr \Phi_0}(p \restr \alpha))) && (\text{(A) in Section~\ref{SEC_Complete_Embeddings}})\\
			&= \iota^{\Psi}_{\alpha, \epsilon}(\kappa^{\alpha}_{\mathbf{x}}(\iota^{\Phi_0, \Phi}_{\alpha}(p \restr \alpha))) && (\text{induction})\\
			&= \kappa^{\epsilon}_{\mathbf{x}}(\iota^{\Phi}_{\alpha, \epsilon}(\iota^{\Phi_0, \Phi}_{\alpha}(p \restr \alpha))) && (\text{Definition~\ref{DEF_Increasing_Sequence_Of_Functors}})\\
			&= \kappa^{\epsilon}_{\mathbf{x}}(\iota^{\Phi_0, \Phi}_{\epsilon}(\iota^{\Phi_0}_{\alpha, \epsilon}(p \restr \alpha))) && (\text{(A) in Section~\ref{SEC_Complete_Embeddings}})\\
			&= \kappa^{\epsilon}_{\mathbf{x}}(\iota^{\Phi_0, \Phi}_{\epsilon}(p)) && (\text{choice of } \alpha).\qedhere
		\end{align*}
	\end{proof}
	
	\section{Proof of the Main Theorem} \label{SEC_Proof}
	
	Finally, we prove the our Main Theorem~\ref{THM_AT_Spectrum}.
	The main part of the proof is an isomorphism-of-names argument to exclude values from $\spec(\aT)$.
	For similar arguments, also see \cite{Brian_2021}, \cite{ShelahSpinas_2015}.
	
	\setcounter{section}{3}
	\begin{maintheorem}
		Assume $\sf{GCH}$ and let $\Theta$ be a set of uncountable cardinals such that
		\begin{enumerate}[\normalfont (I)]
			\item $\max(\Theta)$ exists and has uncountable cofinality,
			\item $\Theta$ is closed under singular limits,
			\item If $\theta \in \Theta$ with $\cof(\theta) = \omega$, then $\theta^+ \in \Theta$,
			\item $\aleph_1 \in \Theta$.
		\end{enumerate}
		Then, there is a c.c.c.\ forcing extension in which $\spec(\aT) = \Theta$ holds.
	\end{maintheorem}
\setcounter{section}{9}
	
	\begin{proof}
		For technical reasons we assume that $\max(\Theta)$ appears $\max(\Theta)$ many times in $\Theta$, so that $\Theta$ has size $\max(\Theta)$ and we add $\max(\Theta)$ many partitions of $\cantorspace$ into $F_\sigma$-sets of size $\max(\Theta)$.
		Let $\Psi$ be the $\Theta$-indexing function defined by $\Psi(\theta) := \theta$ for every $\theta \in \Theta$.
		We show that $\PP^\Psi_{\aleph_1} \forces \spec(\aT) = \Theta$.
		Since $\PP_{\aleph_1}^\Psi$ is c.c.c.\ no cardinals are collapsed and since $\left|\PP_{\aleph_1}^\Psi\right| = \max(\Theta)$ and $\max(\Theta)^{\aleph_0} = \max(\Theta)$ we have $\PP_{\aleph_1}^\Psi \forces \c = \max(\Theta)$.
		Further, as in Lemma~\ref{LEM_Realize_Witness} we have
		$$
			\PP^\Psi_{\aleph_1} \forces \Theta \subseteq \spec(\aT),
		$$
		so we only have to prove the reverse inclusion.
		Let $\lambda \notin \Theta$, $p \in \PP_{\aleph_1}^\Psi$ and $\seq{\dot{T}_\alpha}{\alpha < \lambda}$ be a family of $\PP_{\aleph_1}^\Psi$-names such that
		$$
			p \forces \seq{\dot{T}_\alpha}{\alpha < \lambda} \text{ is an almost disjoint family trees}.
		$$
		Since trees can be coded by reals we may assume that $\dot{T}_\alpha$ is a nice $\PP^\Psi_{\aleph_1}$-name as in Definition~\ref{DEF_Nice_Name_Real}.
		By assumption on $\Theta$ and {\sf{GCH}} there is a regular uncountable cardinal $\sigma \leq \lambda$ with $[\mu, \lambda] \cap \Theta = \emptyset$ and such that for all $\mu < \sigma$ we have $\mu^{\aleph_0} < \sigma$.
		Now, fix $\alpha < \lambda$.
		We define $\Theta_\alpha := \hsupp_\Theta(\dot{T}_\alpha)$, $D_\alpha := \hsupp(\dot{T}_\alpha)$ and for every $\theta \in \Theta$ let $D_\alpha(\theta) := D_\alpha \cap (\simpleset{\theta} \times V)$ (see Definition~\ref{DEF_Hereditary_Support}).
		Then, $\seq{\Theta_\alpha}{\alpha < \sigma}$ satisfies the assumptions of the generalized $\Delta$-system lemma:
		\begin{enumerate}[$\bullet$]
			\item $\seq{\Theta_\alpha}{\alpha < \sigma}$ is a family of size $\sigma$,
			\item $\left|\Theta_\alpha\right| < \aleph_1$ for all $\alpha < \sigma$,
			\item $\aleph_1 < \sigma$ and for all $\mu < \sigma$ we have $\mu^{<\aleph_1} = \mu^{\aleph_0} < \sigma$.
		\end{enumerate}
		Choose $I_0 \in [\sigma]^\sigma$ and $\Theta_R$ such that $\set{\Theta_\alpha}{\alpha \in I_0}$ is a $\Delta$-system lemma with root $\Theta_R$.
		Since $\left|\Theta\right| = \max(\Theta) > \sigma$, we may assume that we extended every $\Theta_\alpha$ for $\alpha \in I_0$ such that
		\begin{enumerate}
			\item $\Theta_\alpha$ is still countable and $\set{\Theta_\alpha}{\alpha \in I_0}$ is still a $\Delta$-system with root $\Theta_R$,
			\item For every $\alpha \in I_0$ we have $\left|\Theta_\alpha \setminus \Theta_R \right| = \aleph_0$.
		\end{enumerate}
		Next, also $\set{D_\alpha}{\alpha \in I_0}$ satisfies the assumptions of the generalized $\Delta$-system lemma:
		\begin{enumerate}[$\bullet$]
			\item $\set{D_\alpha}{\alpha \in I_0}$ is a family of size $\sigma$,
			\item $\left|D_\alpha\right| < \aleph_1$ for all $\alpha \in I_0$,
			\item $\aleph_1 < \sigma$ and for all $\mu < \sigma$ we have $\mu^{<\aleph_1} = \mu^{\aleph_0} < \sigma$.
		\end{enumerate}
		Choose $I_1 \in [I_0]^\sigma$ and $R$ such that $\set{D_\alpha}{\alpha \in I_1}$ is a $\Delta$-system lemma with root $R$.
		For every $\theta \in \Theta$ let $R(\theta) := R \cap (\simpleset{\theta} \times V)$.
		For every $\theta > \sigma$ we have $\left|\Psi(\theta)\right| > \sigma$, so we may assume that we extended every $D_\alpha$ for $\alpha \in I_1$ such that
		\begin{enumerate}\setcounter{enumi}{2}
			\item $D_\alpha$ is still countable and $\set{D_\alpha}{\alpha \in I_1}$ is still a $\Delta$-system with root $R$,
			\item For every $\alpha \in I_1$ and $\theta \in \Theta_R$ with $\theta > \sigma$ we have $\left|D_\alpha(\theta) \setminus R(\theta) \right| = \aleph_0$,
			\item For every $\alpha \in I_1$ and $\theta \in \Theta_\alpha \setminus \Theta_R$ we have $\left|D_\alpha(\theta)\right| = \aleph_0$.
		\end{enumerate}
		Now, set $I_2 := \set{\alpha \in I_1}{\text{For all } \theta \in \Theta_R \text{ with } \theta < \sigma \text{ we have } D_\alpha(\theta) \subseteq R(\theta)}$.
		Then, $I_2 \in [I_1]^\sigma$ as for every $\theta \in \Theta_R$ with $\theta < \sigma$ there are only $<\!\!\sigma$-many $\alpha \in I_1$ with $D_\alpha(\theta) \setminus R(\theta) \neq \emptyset$, since $\left|\Psi(\theta)\right| = \theta$ and $\set{D_\alpha}{\alpha \in I_1}$ is a $\Delta$-system of size $\sigma > \theta$.
		Thus, we obtain
		\begin{enumerate}\setcounter{enumi}{5}
			\item For every $\alpha \in I_2$ and $\theta \in \Theta_R$ with $\theta < \sigma$ we have $D_\alpha(\theta) = R(\theta)$.
		\end{enumerate}
		We extend our $\Delta$-system by one more element as follows.
		Choose $\Theta_\lambda \subseteq \Theta$ countable such that $\Theta_R \subseteq \Theta_\lambda$, $\left|\Theta_\lambda \setminus \Theta_R\right| = \aleph_0$ and for all $\alpha < \lambda$ we have $\Theta_\lambda \cap \Theta_\alpha = \Theta_R$.
		This is possible since $\left|\Theta\right| = \max(\Theta) > \lambda$.
		Now, for $\theta \in \Theta$ we define $D_\lambda(\theta)$ as follows:
		\begin{enumerate}[$\bullet$]
			\item If $\theta \in \Theta_R$ and $\theta < \sigma$ define $D_\lambda(\theta) := R(\theta)$,
			\item If $\theta \in \Theta_R$ and $\theta > \sigma$ we have $\left|\Psi(\theta)\right| = \theta > \lambda$, so choose $D_{\lambda}(\theta) \subseteq (\simpleset{\theta} \times \Psi(\theta))$ countable with $R(\theta) \subseteq D_\lambda(\theta)$, $\left|D_{\lambda}(\theta) \setminus R(\theta)\right| = \aleph_0$ and for all $\alpha < \lambda$ we have $D_\lambda(\theta) \cap D_{\alpha}(\theta) = R(\theta)$,
			\item If $\theta \in \Theta_\lambda \setminus \Theta_R$ choose any countable subset $D_\lambda(\theta) \subseteq (\simpleset{\theta} \times \Psi(\theta))$,
			\item If $\theta \in \Theta \setminus \Theta_\lambda$ set $D_\lambda(\theta) := \emptyset$.
		\end{enumerate}
		Finally, we define $D_\lambda := \bigcup_{\theta \in \Theta} D_\lambda(\theta)$.
		By choice of $\Theta_\lambda$ we have that $\set{\Theta_\alpha}{\alpha \in I_2 \cup \simpleset{\lambda}}$ is a $\Delta$-system with root $\Theta_R$ and similarly by choice of $D_\lambda$ also $\set{D_\alpha}{\alpha \in I_2 \cup \simpleset{\lambda}}$ is a $\Delta$-system with root $R$ and properties $(1)$ to $(6)$ still hold for every $\alpha \in I_2 \cup \simpleset{\lambda}$.
		Next, we define a $\Theta$-subindexing function $\Phi_R$ of $\Psi$ by $\dom(\Phi_R) := \Theta_R$ and for $\theta \in \Theta_R$ by
		$$
			\Phi_R(\theta) := 	\set{i \in \Psi(\theta)}{(\theta, i) \in R(\theta)}.
		$$	
		Analogously, for every $\alpha \in \lambda \cup \simpleset{\lambda}$ define a $\Theta$-subindexing function $\Phi_\alpha$ of $\Psi$ by $\dom(\Phi_\alpha) := \Theta_\alpha$ and for $\theta \in \Theta_\alpha$ by
		$$
			\Phi_\alpha(\theta) := 	\set{i \in \Psi(\theta)}{(\theta, i) \in D_\alpha(\theta)}.
		$$
		As $\Theta_R$ and $R$ are roots of their respective $\Delta$-system we obtain $\Phi_R \subseteq \Phi_\alpha$ for every $\alpha \in I_2 \cup \simpleset{\lambda}$.
		Since, $\hsupp(\dot{T}_\alpha) \subseteq D_\alpha$ we may pick a nice $\PP^{\Phi_\alpha}_{\aleph_1}$-name $\dot{T}^*_\alpha$ with $\iota_{\aleph_1}^{\Phi_\alpha, \Psi}(\dot{T}^*_\alpha) = \dot{T}_\alpha$.
		Further, by $(2)$ we may fix bijections $\seq{g_\alpha:\Theta_\alpha \to \omega}{\alpha \in I_2 \cup \simpleset{\lambda}}$ such that $g_\alpha \restr \Theta_R = g_\beta \restr\Theta_R$ for all $\alpha, \beta \in I_2 \cup \simpleset{\lambda}$.
		Then, for $\alpha, \beta \in I_2 \cup \simpleset{\lambda}$ we define $g_{\alpha, \beta}:\Theta_\alpha \to \Theta_\beta$ by 
		$$
			g_{\alpha, \beta}(\theta) := g_\beta^{-1}(g_\alpha(\theta)).
		$$
		Note that $\Theta_\alpha \cap \Theta_\beta = \Theta_R$ and $g_\alpha \restr \Theta_R = g_\beta \restr \Theta_R$ implies that
		$$
			g_{\alpha, \beta}(\theta) = g_\beta^{-1}(g_\alpha(\theta)) = \theta
		$$
		for all $\theta \in \Theta_R$ and $\alpha, \beta \in I_2 \cup \simpleset{\lambda}$.
		Hence, it is easy to verify that we obtain a system of bijections $\set{g_{\alpha, \beta}:\Theta_\alpha \to \Theta_\beta}{\alpha, \beta \in I_2 \cup \simpleset{\lambda}}$ with the following properties for all $\alpha, \beta, \gamma \in I_2 \cup \simpleset{\lambda}$:
		\begin{enumerate}[(G1)]
			\item $g_{\alpha, \alpha} = \id_{\Theta_\alpha}$ and $g_{\alpha, \beta}^{-1} = g_{\beta,\alpha}$,
			\item for all $\theta \in \Theta_R$ we have $g_{\alpha, \beta}(\theta) = \theta$,
			\item $g_{\alpha, \gamma} = g_{\beta, \gamma} \circ g_{\alpha, \beta}$.
		\end{enumerate}
		Next, for every $\alpha \in I_2 \cup \simpleset{\lambda}$ and $\theta \in \Theta_\alpha$ we may fix a bijection $h^{\theta}_\alpha:\Phi_\alpha(\theta) \to \omega$ such that for all $\alpha, \beta \in I_2 \cup \simpleset{\lambda}$, $\theta \in \Theta_R$ and $i \in R(\theta)$ we have $h_\alpha^\theta(i) = h^\theta_\beta(i)$.
		This is possible, since by $(4)$ and $(6)$ we have $\left|D_\alpha(\theta) \setminus R(\theta) \right| = \left|D_\beta(\theta) \setminus R(\theta) \right|$ for every $\theta \in \Theta_R$.
		Then, for $\alpha, \beta \in I_2 \cup \simpleset{\lambda}$ and $\theta \in \Theta_\alpha$ we define a map $h_{\alpha, \beta}^{\theta}:\Phi_\alpha(\theta) \to \Phi_\beta(g_{\alpha,\beta}(\theta))$ for $i \in \Phi_\alpha(\theta)$ by
		$$
			h_{\alpha, \beta}^\theta(i) := ((h_{\beta}^{g_{\alpha, \beta}(\theta)})^{-1} \circ h_{\alpha}^\theta)(i).
		$$
		We verify the following properties for all $\alpha, \beta, \gamma \in I_2 \cup \simpleset{\lambda}$ and $\theta \in \Theta_\alpha$:
		\begin{enumerate}[(H1)]
			\item $h^\theta_{\alpha, \alpha} = \id_{\Phi_\alpha(\theta)}$ and the map $h_{\alpha, \beta}^\theta:\Phi_\alpha(\theta) \to \Phi_\beta(g_{\alpha,\beta}(\theta))$ is a bijection with inverse $h_{\beta, \alpha}^{g_{\alpha,\beta}(\theta)}$,
			\item for all $i \in R(\theta)$ we have $h^\theta_{\alpha,\beta}(i) = i$,
			\item $h^\theta_{\alpha, \gamma} = h_{\beta, \gamma}^{g_{\alpha, \beta}(\theta)} \circ h_{\alpha,\beta}^{\theta}$.
		\end{enumerate}
	
		\begin{proof}\
			\begin{enumerate}[(H1)]
				\item Let $i \in \Phi_\alpha(\theta)$.
				Then, $g_{\alpha, \alpha}(\theta) = \theta$ by $(\text{G3})$, so that
				$$
				h^\theta_{\alpha, \alpha}(i) = ((h_{\alpha}^{g_{\alpha, \alpha}(\theta)})^{-1} \circ h_{\alpha}^\theta)(i) = ((h_{\alpha}^{\theta})^{-1} \circ h_{\alpha}^\theta)(i) = i.
				$$
				Next, by definition we have $h_{\beta, \alpha}^{g_{\alpha,\beta}(\theta)}:\Phi_\beta(g_{\alpha,\beta}(\theta)) \to \Phi_\alpha(g_{\beta,\alpha}(g_{\alpha,\beta}(\theta)))$.
				Further,  by $(\text{G1})$ $g_{\beta,\alpha}(g_{\alpha,\beta}(\theta)) = \theta$, so that $h_{\beta, \alpha}^{g_{\alpha,\beta}(\theta)}:\Phi_\beta(g_{\alpha,\beta}(\theta)) \to \Phi_\alpha(\theta)$, so the domains are correct.
				Now, let $i \in \Phi_\alpha(\theta)$.
				Then, we compute
				\begin{align*}
					(h_{\beta, \alpha}^{g_{\alpha,\beta}(\theta)} \circ h_{\alpha, \beta}^\theta)(i)
					&= ((h_{\alpha}^{g_{\beta, \alpha}(g_{\alpha,\beta}(\theta))})^{-1} \circ h_\beta^{g_{\alpha,\beta}(\theta)} \circ (h_{\beta}^{g_{\alpha, \beta}(\theta)})^{-1}  \circ h_{\alpha}^\theta)(i)\\
					&= ((h_{\alpha}^{\theta})^{-1} \circ h_{\alpha}^\theta)(i) = i.
				\end{align*}
				Analogously, for $i \in \Phi_\beta(g_{\alpha,\beta}(\theta))$ we compute
				\begin{align*}
					(h_{\alpha, \beta}^\theta \circ h_{\beta, \alpha}^{g_{\alpha,\beta}(\theta)})(i)
					&= ((h_{\beta}^{g_{\alpha, \beta}(\theta)})^{-1} \circ h_{\alpha}^\theta \circ (h_{\alpha}^{g_{\beta, \alpha}(g_{\alpha,\beta}(\theta))})^{-1} \circ h_\beta^{g_{\alpha,\beta}(\theta)})(i)\\
					&= ((h_{\beta}^{g_{\alpha, \beta}(\theta)})^{-1} \circ h_{\alpha}^\theta \circ (h_{\alpha}^{\theta})^{-1} \circ h_\beta^{g_{\alpha,\beta}(\theta)})(i)\\
					&= ((h_{\beta}^{g_{\alpha, \beta}(\theta)})^{-1} \circ h_\beta^{g_{\alpha,\beta}(\theta)})(i) = i.
				\end{align*}
				\item Let $i \in R(\theta)$.
				Then, $\theta \in \Theta_R$ and $g_{\alpha,\beta}(\theta) = \theta$ by $(\text{G2})$.
				Hence, by choice of the bijections
				\begin{align*}
					h_{\alpha, \beta}^\theta(i)
					&= ((h_{\beta}^{g_{\alpha, \beta}(\theta)})^{-1} \circ h_{\alpha}^\theta)(i)\\
					&= ((h_{\beta}^{\theta})^{-1} \circ h_{\alpha}^\theta)(i)\\
					&= ((h_{\beta}^{\theta})^{-1} \circ h_{\beta}^\theta)(i)\\
					&= i.
				\end{align*}
				\item Finally, let $\theta \in \Theta_\alpha$ and $i \in \Phi_\alpha(\theta)$.
				Then, $g_{\alpha, \gamma} = g_{\beta, \gamma} \circ g_{\alpha, \beta}$ by (G3), so we compute
				\begin{align*}
					(h_{\beta, \gamma}^{g_{\alpha, \beta}(\theta)} \circ h_{\alpha,\beta}^{\theta})(i)
					&= ((h_\gamma^{g_{\beta,\gamma}(g_{\alpha,\beta}(\theta))})^{-1} \circ h_\beta^{g_{\alpha,\beta}(\theta)} \circ (h_{\beta}^{g_{\alpha, \beta}(\theta)})^{-1} \circ h_{\alpha}^\theta)(i)\\
					&= ((h_\gamma^{g_{\alpha,\gamma}(\theta)})^{-1} \circ h_{\alpha}^\theta)(i)\\
					&= h_{\alpha, \gamma}^\theta(i).\qedhere
				\end{align*}
			\end{enumerate}
		\end{proof}
		Now, if for every $\alpha, \beta, \gamma \in I_2 \cup \simpleset{\lambda}$ we define the tuple $\mathbf{x}_{\alpha, \beta} := (g_{\alpha,\beta}, \set{h_{\alpha, \beta}^{\theta}}{\theta \in \Theta_\alpha})$
		(G1) to (G3) and (H1) to (H3) may be rephrased as a system of isomorphisms of $\Theta$-indexing functions $\seq{\mathbf{x}_{\alpha, \beta}}{\alpha, \beta \in I_2 \cup \simpleset{\lambda}}$ which satisfies
		\begin{enumerate}[(K1')]
			\item $\mathbf{x}_{\alpha, \alpha} = \mathbf{1}_{\Phi_\alpha}$ and $\mathbf{x}_{\alpha, \beta}^{-1} = \mathbf{x}_{\beta, \alpha}$,
			\item $\mathbf{x}_{\alpha, \alpha} \restr \Phi_R = \mathbf{1}_{\Phi_R}$,
			\item $\mathbf{x}_{\alpha, \gamma} = \mathbf{x}_{\beta, \gamma} \circ \mathbf{x}_{\alpha, \beta}$.
		\end{enumerate}
		Applying the functor $\kappa^{\aleph_1}_\bullet$ from Corollary~\ref{COR_Induced_Sequence_Of_Functors} to the system $\seq{\mathbf{x}_{\alpha, \beta}}{\alpha, \beta \in I_2 \cup \simpleset{\lambda}}$, we obtain a system of isomorphisms $\seq{\kappa_{\mathbf{x}_{\alpha, \beta}}:\PP^{\Phi_\alpha}_{\aleph_1} \to \PP^{\Phi_\beta}_{\aleph_1}}{\alpha, \beta \in I_2 \cup \simpleset{\lambda}}$ which satisfies
		\begin{enumerate}[(K1)]
			\item $\kappa_{\mathbf{x}_{\alpha, \alpha}} = \id_{\PP^{\Phi_\alpha}_{\aleph_1}}$ and $\kappa_{\mathbf{x}_{\alpha, \beta}}^{-1} = \kappa_{\mathbf{x}_{\beta, \alpha}}$,
			\item $\kappa_{\mathbf{x}_{\alpha, \beta}} \circ \iota_{\aleph_1}^{\Phi_R, \Phi_\alpha} = \iota_{\aleph_1}^{\Phi_R, \Phi_\beta}$,
			\item $\kappa_{\mathbf{x}_{\alpha, \gamma}} = \kappa_{\mathbf{x}_{\beta, \gamma}} \circ \kappa_{\mathbf{x}_{\alpha, \beta}}$.
		\end{enumerate}	
		Fix $\beta_0 \in I_2$.
		For every $\alpha \in I_2$ we have that $\dot{T}^*_\alpha$ is a nice $\PP^{\Phi_\alpha}_{\aleph_1}$-name for a tree.
		Thus, $\kappa_{\mathbf{x}_{\alpha, \beta_0}}(\dot{T}^*_\alpha)$ is a nice $\PP^{\Phi_{\beta_0}}_{\aleph_1}$-name for a tree.
		However, $\Phi_{\beta_0}$ is countable, so by Lemma~\ref{LEM_Counting_Hereditary} there are only $\aleph_1$-many nice $\PP^{\Phi_{\beta_0}}_{\aleph_1}$-names for such trees.
		Thus, choose $I_3 \in [I_2]^\sigma$ such that $\kappa_{\mathbf{x}_{\alpha, \beta_0}}(\dot{T}^*_\alpha) = \kappa_{\mathbf{x}_{\alpha', \beta_0}}(\dot{T}^*_{\alpha'})$ for all $\alpha, \alpha' \in I_3$.
		
		Finally, choose any $\alpha_0 \in I_3$ and define $\dot{T}^*_\lambda$ to be $\kappa_{\mathbf{x}_{\alpha_0, \lambda}}(\dot{T}^*_{\alpha_0})$.
		Then, $\dot{T}^*_\lambda$ is a nice $\PP^{\Phi_\lambda}_{\aleph_1}$-name for a tree.
		We show that this definition is independent of the choice of $\alpha_0 \in I_3$, so let $\alpha \in I_3$.
		Then, we compute
		\begin{align*}
			\dot{T}^*_\lambda
			&= \kappa_{\mathbf{x}_{\alpha_0, \lambda}}(\dot{T}^*_{\alpha_0}) && (\text{definition of } \dot{T}^*_\lambda)\\
			&= \kappa_{\mathbf{x}_{\beta_0, \lambda}} (\kappa_{\mathbf{x}_{\alpha_0, \beta_0}}(\dot{T}^*_{\alpha_0})) && (\text{K3})\\
			&= \kappa_{\mathbf{x}_{\beta_0, \lambda}} (\kappa_{\mathbf{x}_{\alpha, \beta_0}}(\dot{T}^*_{\alpha})) && (\alpha \in I_3)\\
			&= \kappa_{\mathbf{x}_{\alpha, \lambda}}(\dot{T}^*_{\alpha}) && (\text{K3}).
		\end{align*}
		Finally, let $\beta < \lambda$. Since $\set{\Theta_\alpha}{\alpha \in I_3}$ is a $\Delta$-system with root $\Theta_R$ and $\Theta_\beta$ is countable, there can only be countably many $\alpha \in I_3$ with $\Theta_\alpha \cap \Theta_\beta \not\subseteq \Theta_R$.
		Further, since $\set{D_\alpha}{\alpha \in I_3}$ is a $\Delta$-system with root $R$ and for every $\theta \in \Theta_R$ the set $\Phi_\beta(\theta)$ is countable, there can only be countable many $\alpha \in I_3$ with $\Phi_\alpha(\theta) \cap \Phi_\beta(\theta) \not\subseteq R(\theta)$.
		Thus, we may choose $\alpha \in I_3 \setminus \simpleset{\beta}$ such that $\Theta_\alpha \cap \Theta_\beta \subseteq \Theta_R$ and for all $\theta \in \Theta_R$ we have $\Phi_{\alpha}(\theta) \cap \Phi_{\beta}(\theta) \subseteq R(\theta)$.
		By definition of $\Theta_\lambda$ we also have $\Theta_\lambda \cap \Theta_\beta \subseteq \Theta_R$ and for all $\theta \in \Theta_R$ also $\Phi_{\lambda}(\theta) \cap \Phi_{\beta}(\theta) \subseteq R(\theta)$.
		For $\nu \in \simpleset{\alpha, \lambda}$ we define a $\Theta$-subindexing function $\Phi_\nu^*$ of $\Psi$ by $\dom(\Phi_\nu^*) := \Theta_\nu \cup \Theta_\beta$ and for $\theta \in \Theta_\nu^*$ by
		$$
			\Phi_\nu^*(\theta) := \Phi_\nu(\theta) \cup \Phi_\beta(\theta),
		$$
		where every undefined set is treated as the empty set.
		We define a bijection $g_{\alpha, \lambda}^*:\Theta_\alpha^* \to \Theta_\lambda^*$ for $\theta \in \Theta^*_\alpha$ by
		$$
			g_{\alpha, \lambda}^*(\theta) :=
			\begin{cases}
				g_{\alpha, \lambda}(\theta) & \text{if } \theta \in \Theta_\alpha,\\
				\theta & \text{if } \theta \in \Theta_\beta.
			\end{cases}
		$$
		This is well-defined by (G2) and  $\Theta_\alpha \cap \Theta_\beta \subseteq \Theta_R$.
		Further, for every $\theta \in \Theta_\alpha^*$ define a bijection $h_{\alpha, \beta}^{\theta, *}:\Phi_\alpha^*(\theta) \to \Phi^*_\beta(g_{\alpha,\lambda}^*(\theta))$ as follows:
		\begin{enumerate}[$\bullet$]
			\item If $\theta \in \Theta_R$ we have $\Phi_\alpha^*(\theta) \setminus \Phi_\alpha(\theta) = \Phi_\lambda^*(\theta) \setminus \Phi_\lambda(\theta)$ and $g_{\alpha, \lambda}^*(\theta) = \theta$, so we may extend the bijection $h_{\alpha, \lambda}^{\theta}:\Phi_\alpha(\theta) \to \Phi_\lambda(\theta)$ to $h_{\alpha, \beta}^{\theta, *}:\Phi_\alpha^*(\theta) \to \Phi^*_\beta(\theta)$ by 
			$$
			h^{\theta, *}_{\alpha, \lambda}(i) =
			\begin{cases}
				h_{\alpha, \lambda}^{\theta}(i) & \text{if } i \in \Phi_\alpha(\theta),\\
				i & \text{otherwise}.
			\end{cases}
			$$
			This is well-defined by (H2) and $\Phi_{\alpha}(\theta) \cap \Phi_{\beta}(\theta) \subseteq R(\theta)$.
			\item If $\theta \in \Theta_\alpha \setminus \Theta_R$, then we have $\Phi^*_\alpha(\theta) = \Phi_\alpha(\theta)$, $\Phi^*_\lambda(\theta) = \Phi_\lambda(\theta)$, so we may define $h^{\theta, *}_{\alpha, \lambda} = h_{\alpha, \lambda}^{\theta}$.
			\item If $\theta \in \Theta_\beta \setminus \Theta_R$, then we have $\Phi^*_\alpha(\theta) = \Phi_\beta(\theta) = \Phi^*_\lambda(\theta)$ and $g_{\alpha, \lambda}^*(\theta) = \theta$, so we may define $h^{\theta, *}_{\alpha, \lambda} = \id_{\Phi^*_\alpha(\theta)}$.
		\end{enumerate}
		Then, the tuple $\mathbf{x_{\alpha, \lambda}^*} = (g^*_{\alpha, \lambda}, \set{h^{\theta, *}_{\alpha, \lambda}}{\theta \in \Theta^*_\alpha})$ is an isomorphism from $\Phi_\alpha^*$ to $\Phi_\lambda^*$.
		Further, we have $\Phi_\alpha, \Phi_\beta \subseteq \Phi_\alpha^*$ and $\Phi_\lambda, \Phi_\beta \subseteq \Phi_\lambda^*$ as well as
		\begin{enumerate}[(L1')]
			\item $\mathbf{x_{\alpha, \lambda}^*} \restr \Phi_\alpha = \mathbf{x_{\alpha, \lambda}}$,
			\item $\mathbf{x_{\alpha, \lambda}^*} \restr \Phi_\beta = \mathbf{1}_{\Phi_\beta}$.
		\end{enumerate}
		By Lemma~\ref{LEM_Restriction_Maps_To_Restriction_Lim} applying $\kappa^{\aleph_1}_\bullet$ from Corollary~\ref{COR_Induced_Sequence_Of_Functors} yields
		an automorphism $\kappa_{\mathbf{x}^*_{\alpha, \lambda}}: \PP^{\Phi_\alpha^*}_{\aleph_1} \to \PP^{\Phi_\lambda^*}_{\aleph_1}$ with the following properties:
		\begin{enumerate}[(L1)]
			\item $\kappa_{\mathbf{x}^*_{\alpha, \lambda}} \circ \iota_{\aleph_1}^{\Phi_\alpha, \Phi_\alpha^*} = \iota_{\aleph_1}^{\Phi_\lambda, \Phi_\lambda^*} \circ \kappa_{\mathbf{x}_{\alpha, \lambda}}$,
			\item $\kappa_{\mathbf{x}^*_{\alpha, \lambda}} \circ \iota_{\aleph_1}^{\Phi_\beta, \Phi_\alpha^*} = \iota_{\aleph_1}^{\Phi_\beta, \Phi_\lambda^*}$.
		\end{enumerate}
		Choose $p^* \in \PP^{\Phi_R}_{\aleph_1}$ with $\iota^{\Phi_R, \Psi}_{\aleph_1}(p^*) = p$.
		Then, we compute
		\begin{align*}
			\kappa_{\mathbf{x}^*_{\alpha, \lambda}}(\iota_{\aleph_1}^{\Phi_R, \Phi_\alpha^*}(p^*))
			&= \kappa_{\mathbf{x}^*_{\alpha, \lambda}}(\iota_{\aleph_1}^{\Phi_\beta, \Phi_\alpha^*}(\iota_{\aleph_1}^{\Phi_R, \Phi_\beta}(p^*))) && (\Phi_R \subseteq \Phi_\beta \subseteq \Phi_\alpha^*)\\
			&= \iota_{\aleph_1}^{\Phi_\beta, \Phi_\lambda^*}(\iota_{\aleph_1}^{\Phi_R, \Phi_\beta}(p^*)) && (\text{L2})\\
			&= \iota_{\aleph_1}^{\Phi_R, \Phi_\lambda^*}(p^*) && (\Phi_R \subseteq \Phi_\beta \subseteq \Phi_\lambda^*).
		\end{align*}
		Similarly, by (L2) we have $\kappa_{\mathbf{x}^*_{\alpha, \lambda}}(\iota_{\aleph_1}^{\Phi_\beta, \Phi_\alpha^*}(\dot{T}^*_\beta)) = \iota_{\aleph_1}^{\Phi_\beta, \Phi_\lambda^*}(\dot{T}^*_\beta)$.
		We also compute
		\begin{align*}
			\kappa_{\mathbf{x}^*_{\alpha, \lambda}}(\iota_{\aleph_1}^{\Phi_\alpha, \Phi_\alpha^*}(\dot{T}^*_\alpha))
			&= \iota_{\aleph_1}^{\Phi_\lambda, \Phi_\lambda^*}(\kappa_{\mathbf{x}_{\alpha, \lambda}}(\dot{T}^*_\alpha)) && (\text{L1})\\
			&= \iota_{\aleph_1}^{\Phi_\lambda, \Phi_\lambda^*}(\dot{T}^*_\lambda) && (\alpha \in I_3).
		\end{align*}
		We may now finish the argument.
		Since
		$$
			p \forces_{\PP^{\Psi}_{\aleph_1}} [\dot{T}_\alpha] \cap [\dot{T}_\beta] = \emptyset,
		$$
		we have
		$$
			\iota^{\Phi_\alpha^*, \Psi}_{\aleph_1}(\iota_{\aleph_1}^{\Phi_R, \Phi_\alpha^*}(p^*)) \forces_{\PP^{\Psi}_{\aleph_1}} [\iota^{\Phi_\alpha^*, \Psi}_{\aleph_1}(\iota_{\aleph_1}^{\Phi_\alpha, \Phi_\alpha^*}(\dot{T}^*_\alpha))] \cap [\iota^{\Phi_\alpha^*, \Psi}_{\aleph_1}(\iota_{\aleph_1}^{\Phi_\beta, \Phi_\alpha^*}(\dot{T}^*_\beta))] = \emptyset.
		$$
		By Theorem~\ref{THM_Complete_Subforcings} we may use downwards absoluteness to obtain
		$$
			\iota_{\aleph_1}^{\Phi_R, \Phi_\alpha^*}(p^*) \forces_{\PP^{\Phi^*_\alpha}_{\aleph_1}} [\iota_{\aleph_1}^{\Phi_\alpha, \Phi_\alpha^*}(\dot{T}^*_\alpha)] \cap [\iota_{\aleph_1}^{\Phi_\beta, \Phi_\alpha^*}(\dot{T}^*_\beta)] = \emptyset.
		$$
		Applying the isomorphism $\kappa_{\mathbf{x}^*_{\alpha, \lambda}}:\PP^{\Phi^*_\alpha}_{\aleph_1} \to \PP^{\Phi^*_\lambda}_{\aleph_1}$ and the computation above yields
		$$
			\iota_{\aleph_1}^{\Phi_R, \Phi_\lambda^*}(p^*) \forces_{\PP^{\Phi^*_\lambda}_{\aleph_1}} [\iota_{\aleph_1}^{\Phi_\lambda, \Phi_\lambda^*}(\dot{T}^*_\lambda)] \cap [\iota_{\aleph_1}^{\Phi_\beta, \Phi_\lambda^*}(\dot{T}^*_\beta)] = \emptyset.
		$$
		By Theorem~\ref{THM_Complete_Subforcings} we may use $\Pi^1_1$-absoluteness to obtain
		$$
			p \forces_{\PP^{\Psi}_{\aleph_1}} [\dot{T}_\lambda] \cap [\dot{T}_\beta] = \emptyset,
		$$
		so that
		$$
			p \forces_{\PP^{\Psi}_{\aleph_1}} \seq{\dot{T}_\alpha}{\alpha < \lambda} \text{ is not maximal}.
		$$
	\end{proof}

	\bibliographystyle{plain}
	\bibliography{refs}
	
\end{document}